\documentclass[11pt,a4paper,leqno]{amsart}

\usepackage[latin1]{inputenc}
\usepackage[T1]{fontenc}
\usepackage{amsfonts}
\usepackage{amsmath}
\usepackage{amssymb}
\usepackage{eurosym}
\usepackage{mathrsfs}
\usepackage{palatino}
\usepackage{color}
\usepackage{esint}
\usepackage{url}
\usepackage{verbatim}
\usepackage{mathtools}

\usepackage{enumerate}

\usepackage[pagebackref,hypertexnames=false, colorlinks, citecolor=blue, linkcolor=blue, urlcolor=red]{hyperref}

\newcommand{\C}{\mathbb{C}}
\newcommand{\Z}{\mathbb{Z}}

\newcommand{\N}{\mathbb{N}}

\newcommand{\R}{\mathbb{R}}

\newcommand{\E}{\mathbb{E}}

\newcommand{\calI}{\mathcal{I}}
\newcommand{\calJ}{\mathcal{J}}

\newcommand{\calL}{\mathcal{L}}
\newcommand{\calM}{\mathcal{M}}

\newcommand{\scrA}{\mathscr{A}}

\newcommand{\bmo}[0]{\operatorname{bmo}}

\newcommand{\WBP}{\operatorname{WBP}}

\newcommand{\bla}{\big \langle}
\newcommand{\bra}{\big \rangle}

\newcommand{\si}{\sigma}

\newcommand{\ud}[0]{\,\mathrm{d}}

\newcommand{\esssup}[0]{\operatornamewithlimits{ess\,sup}}

\newcommand{\pair}[2]{\langle #1,#2 \rangle}

\newcommand{\ave}[1]{\langle #1\rangle}
\newcommand{\Ave}[1]{\Big\langle #1\Big\rangle}

\newcommand{\BMO}[0]{\operatorname{BMO}}
\newcommand{\supp}[0]{\operatorname{spt}}

\newcommand{\Id}[0]{\operatorname{id}}

\renewcommand{\Re}[0]{\operatorname{Re}}

\newcommand{\UMD}{\operatorname{UMD}}

\newcommand{\good}[0]{\operatorname{good}}

\newcommand{\ch}[0]{\operatorname{ch}}

\newcommand{\calD}[0]{\mathcal{D}}

\newcommand{\wt}[1]{{\widetilde{#1}}}

\swapnumbers
\theoremstyle{plain}
\newtheorem{thm}[equation]{Theorem}
\newtheorem{lem}[equation]{Lemma}
\newtheorem{prop}[equation]{Proposition}
\newtheorem{cor}[equation]{Corollary}

\theoremstyle{definition}
\newtheorem{defn}[equation]{Definition}

\theoremstyle{remark}
\newtheorem{rem}[equation]{Remark}

\numberwithin{equation}{section}

\title[Mild kernel regularity]{Modern singular integral theory with mild kernel regularity}

\author{Emil Airta}
\author{Henri Martikainen}
\author{Emil Vuorinen}

\address[E.A., H.M. \& E.V.]{Department of Mathematics and Statistics, University of Helsinki, P.O.B. 68, FI-00014 University of Helsinki, Finland}
\email{emil.airta@helsinki.fi}
\email{henri.martikainen@helsinki.fi}
\email{emil.vuorinen@helsinki.fi}

\makeatletter
\@namedef{subjclassname@2010}{%
  \textup{2010} Mathematics Subject Classification}
\makeatother


\thanks
{H.M. was supported by the Academy of Finland through the grants 294840 and 327271, and by the three-year research grant 75160010 of the University of Helsinki. E.A. and E.V. were supported by the Academy of Finland through the grant 327271.
All are members of the Finnish Centre of Excellence in Analysis and Dynamics Research supported by the Academy of Finland (project No. 307333).
}

\begin{document}

\begin{abstract}
We present a framework
based on modified dyadic shifts to prove multiple results of modern singular integral theory under mild kernel regularity. 
Using new optimized representation theorems we first revisit a result of Figiel concerning the $\UMD$-extensions of linear Calder\'on--Zygmund operators with mild kernel regularity 
and extend our new proof to the multilinear setting improving recent $\UMD$-valued estimates of multilinear singular integrals.  Next, we develop the product space theory of the multilinear singular integrals with modified Dini-type assumptions, and use this theory to prove bi-parameter weighted estimates and two-weight commutator estimates.
\end{abstract}

\maketitle
\tableofcontents

\section{Introduction}
The usual definition of a singular integral operator (SIO)
$$
  Tf(x)=\int_{\R^d}K(x,y)f(y)\ud y
$$
involves a H\"older-continuous kernel $K$ with a power-type continuity-modulus $t \mapsto t^{\gamma}$.
However, many results continue to hold with significantly more general assumptions.
Such kernel regularity considerations become non-trivial especially in connection with results that go beyond the classical Calder\'on--Zygmund theory -- an example is the $A_2$ theorem of Hyt\"onen \cite{Hy1} with Dini-continuous kernels by Lacey \cite{La1}. Estimates for SIOs with mild kernel regularity are, for instance, linked to the theory of rough singular integrals, see e.g. \cite{HRT}.

The fundamental question concerning the $L^2$ (or $L^p$) boundedness of an SIO $T$ is usually best answered by so-called $T1$ theorems, where the action of the operator $T$ on the constant function $1$ is key.
We study kernel regularity questions specifically in situations that are very tied to the $T1$ type arguments and the corresponding structural theory -- a prominent example is the
theory of SIOs acting on functions taking values in infinite-dimensional Banach spaces \cite{HNVW1, HNVW2}.
Another major line of investigation concerns the product space theory of SIOs, such as \cite{HPW, LMV:Bloom, OPS},
based on the bi-parameter representation theorem \cite{Ma1}. The distinction between one-parameter and multi-parameter SIOs has to do with the classification of SIOs
according to the size of the singularity of the kernel $K$. We will work both in the Banach-valued and bi-parameter settings, and all of our
abstract theory is, in addition, multilinear.

A concrete definition of kernel regularity is as follows.
It concerns the required regularity of the continuity-moduli $\omega$ appearing in the various kernel estimates, such as, 
\begin{equation*}
|K(x, y) - K(x', y)| \le \omega\Big(\frac{|x-x'|}{|x-y|}\Big) \frac{1}{|x-y|^d}, \,\, |x-x'| \le |x-y|/2.
\end{equation*}
Recently, Grau de la Herr\'an and Hyt\"onen \cite{GH} proved that the modifed Dini condition
$$
\|\omega\|_{\operatorname{Dini}_{\alpha}} := \int_0^1 \omega(t) \Big( 1 + \log \frac{1}{t} \Big)^{\alpha} \frac{dt}{t}
$$
with $\alpha = \frac{1}{2}$ is sufficient to prove a $T1$ theorem even with an underlying measure $\mu$ that can be non-doubling.
This matches the best known sufficient condition for the classical homogeneous $T1$ theorem \cite{DJ} -- such results are implicit
in Figiel \cite{Figiel2} and explicit in Deng, Yan and Yang \cite{DYY}. 
The exponent $\alpha = \frac{1}{2}$ has a fundamental feeling in all of the existing arguments -- it seems very difficult to achieve a $T1$ theorem with a weaker assumption.

The proofs of $T1$ theorems display a fundamental structural decomposition of SIOs into their cancellative parts and so-called paraproducts. It is
this structure that is extremely important for obtaining further estimates beyond the initial scalar-valued $L^p$ boundedness.
The original dyadic representation theorem of Hyt\"onen \cite{Hy2, Hy1} (extending an earlier special case of Petermichl \cite{Pe}) provides a decomposition of the cancellative part of an SIO into so-called {\em dyadic shifts}. These are suitable generalisations of dyadic martingale transforms, also known as Haar multipliers
\begin{equation}\label{eq:HaarMult}
  f=\sum_{Q\in\mathcal D}\pair{f}{h_Q}h_Q\mapsto \sum_{Q\in\mathcal D}\lambda_Q\pair{f}{h_Q}h_Q, \qquad |\lambda_Q| \le 1,
\end{equation}
where $h_Q$ is a cancellative Haar function on a cube $Q$ and $\calD$ is a dyadic grid.

In \cite{GH} a new type of representation theorem appears, where the key difference to the original representation theorems \cite{Hy2, Hy1} is that the decomposition
of the cancellative part is in terms of different operators that package multiple dyadic shifts into one and offer more efficient bounds when it comes to kernel regularity.
Some of the ideas of the decomposition in \cite{GH} are rooted in the work of Figiel \cite{Figiel1, Figiel2}.

Our work begins by giving a streamlined version of the new type of one-parameter representation theorem \cite{GH}. Importantly, we extend it to the multilinear setting \cite{DLMV1, DLMV2, DLMV3, LMOV, MV} -- more on this later. Already in the linear case we identify a useful, explicit and clear exposition of the appearing new dyadic model operators that we call modified dyadic shifts.
For example, the usual generalization of \eqref{eq:HaarMult} takes the form of a dyadic shift
$$
S_{i,j} f= \sum_{K \in \calD} \sum_{I^{(i)} = J^{(j)} = K} a_{IJK} \langle f, h_I \rangle h_J,
$$
while we replace these by the modified dyadic shifts
$$
Q_k f = \sum_{K \in \calD} \sum_{I^{(k)} = J^{(k)} = K} a_{IJK} \langle f, h_I \rangle  H_{I, J}.
$$
Here $I^{(k)} \in \calD$ is the $k$th parent of $I$. The difference is that
some more complicated functions $H_{I,J}$, which are supported on $I \cup J$, constant on the children of $I$ and $J$ and satisfy $\int H_{I,J} = 0$, appear -- on the other hand, we have $i=j=k$. 

The $\UMD$ -- unconditional martingale differences -- property of a Banach space $X$
is a well-known necessary and sufficient condition for the boundedness of various singular integrals on $L^p(\R^d;X) = L^p(X)$.
A Banach space  $X$ has the $\UMD$ property if $X$-valued martingale difference sequences converge unconditionally in $L^p$ for some (equivalently, all) $p\in (1,\infty)$.
By Burkholder \cite{Bu} and Bourgain \cite{Bo} we have that $X$ is a $\UMD$ space if and only if a particular SIO -- the Hilbert transform $Hf(x) = \operatorname{p.v.} \int_{\R} \frac{f(y) dy}{x-y}$
 -- admits an $L^p(X)$-bounded extension.
This theory quickly advanced up to the vector-valued $T1$ theorem of Figiel \cite{Figiel2}. It is important to understand that the fundamental result that all scalar-valued $L^2$ bounded SIOs can be extended to act boundedly on $L^p(X)$, $p \in (1,\infty)$, goes through this $T1$ theorem.

The deep work of Figiel also contains estimates for the required kernel regularity $\alpha$ in terms of some characteristics of the $\UMD$ space
$X$. It turns out that for completely general $\UMD$ spaces
the threshold $\alpha = 1/2$ needs to be replaced by a more complicated expression, while in the simpler function lattice case $\alpha = 1/2$ suffices by a much more elementary
argument.
 We revisit these linear results using the modified dyadic shifts
and obtain a modern proof of the following theorem. In our terminology, a Calder\'on--Zygmund operator (CZO) is an SIO that satisfies the $T1$ assumptions (equivalenty, $T \colon L^2 \to L^2$ boundedly).
\begin{thm}\label{thm:UMDlin}
Let $T$ be a linear $\omega$-CZO  and $X$ be a $\UMD$ space with type $r \in (1,2]$ and cotype $q \in [2, \infty)$.
If $\omega \in \operatorname{Dini}_{1 / \min(r, q')}$, we have
$$
\|Tf \|_{L^p(X)} \lesssim \|f\|_{L^p(X)}, \qquad p \in (1,\infty).
$$
\end{thm}
See the main text for the exact definitions.
If $X$ is a Hilbert space, then $r = q = 2$ and we get the usual $\alpha = 1/2$. Again, this theory is relevant
for $\UMD$ spaces that go beyond the function lattices, such as, non-commutative $L^p$ spaces.

A major part of our arguments has to do with the extension of these, and other results, to the multilinear setting.
A basic model of an $n$-linear SIO $T$ in $\R^d$ is obtained by setting
\begin{equation*}\label{eq:multilinHEUR}
T(f_1,\ldots, f_n)(x) = U(f_1 \otimes \cdots \otimes f_n)(x,\ldots,x), \qquad x \in \R^d,\, f_i \colon \R^d \to \C,
\end{equation*}
where $U$ is a linear singular integral operator in $\R^{nd}$. See e.g. Grafakos--Torres \cite{GT} for the basic theory.
Multilinear SIOs appear in applications ranging from partial differential equations to complex function theory and ergodic theory.
For example, $L^p$ estimates for the homogeneous fractional derivative $D^{\alpha} f=\mathcal F^{-1}(|\xi|^{\alpha} \widehat f(\xi))$ of a product of two or more functions -- the \emph{fractional Leibniz rules} -- are used in the area of dispersive equations. 
Such estimates descend from the multilinear H\"ormander-Mihlin multiplier theorem of Coifman-Meyer \cite{CM} -- 
See e.g. Kato--Ponce \cite{KP} and Grafakos--Oh \cite{GO}.

The original bilinear representation theorem with the usual power-type continuity-modulus is by some of us together with K. Li and Y. Ou \cite{LMOV}.
An $n$-linear extension of \cite{LMOV} -- even to the operator-valued setting -- is by some of us together with K. Li and F. Di Plinio \cite{DLMV2}. 
The structural theory in the $n$-linear setting is quite delicate already in the above works, and becomes more delicate still in our current
setting of mild kernel regularity and modified dyadic operators. On the other hand, the proofs of the representation theorems appear to be now converging to their final and most elegant form and we can clean up some technicalities of \cite{LMOV}, and provide an efficient argument.

The multilinear analogue of Theorem \ref{thm:UMDlin} extending the recent work \cite{DLMV1} goes as follows.
\begin{thm}
Let $\{X_1,\ldots, X_{n+1}\}$ be a $\UMD$ H\"older tuple as in the Definition \ref{defn:productsys1}, and denote
the cotype of $X_j$ by $s_j$.
Suppose that $T$ is an $n$-linear $\omega$-CZO with 
$\omega \in \operatorname{Dini}_{\alpha}$,
where
$$
\alpha = \frac{1}{\min( (n+1)/n, s_{1}', \ldots, s_{n+1}')}.
$$
Then for all exponents $1 < p_1, \ldots, p_n \le \infty$ and $1 / r = \sum_{j=1}^n 1/p_j > 0$
we have
$$
\|T(f_1, \ldots, f_n) \|_{L^{r}(X_{n+1}^*)} \lesssim \prod_{j=1}^{n} \| f_j \|_{L^{p_j}(X_j)}.
$$
\end{thm}
Until recently, vector-valued extensions of multilinear SIOs had mostly been studied in the framework of $\ell^p$ spaces and function lattices, rather than general UMD spaces -- see e.g. \cite{CUDPOU,GM,LMO,LMMOV,Nieraeth}.
Taking the work \cite{DO} much further, the paper \cite{DLMV1} finally established $L^p$
bounds for the extensions of $n$-linear SIOs with the usual H\"older modulus of continuity to tuples of $\UMD$ spaces tied by a natural product structure, such as, the composition of operators in the  Schatten-von Neumann  subclass of the algebra of bounded operators on a Hilbert space. In \cite{DLMV3} the bilinear case of \cite{DLMV1} was applied to prove $\UMD$-extensions for modulation invariant singular integrals,
such as, the bilinear Hilbert transform. With new and refined methods, we are able to prove the above Figiel type result in the multilinear setting.

We move on to the product space theory.
For example, in the linear case the one-parameter kernels that we have seen thus far are singular when $x=y$. Linear bi-parameter SIOs have kernels with singularities on
 $x_1=y_1$ or $x_2 = y_2$, where $x,y\in\R^d$ are written as $x= (x_1, x_2), y = (y_1, y_2) \in\R^{d_1}\times\R^{d_2}$ for a fixed partition $d=d_1+d_2$. For $x,y \in \C = \R \times \R$, compare e.g. the one-parameter Beurling kernel $1/(x-y)^2$ with the bi-parameter kernel $1/[(x_1-y_1)(x_2-y_2)]$ -- the product of Hilbert kernels in both coordinate directions.
A bi-parameter $T1$ theorem was first achieved by Journ\'e \cite{Jo}, and recovered by one of us \cite{Ma1} through a linear bi-parameter dyadic representation theorem. 
The multi-parameter extension of \cite{Ma1} is by Y. Ou \cite{Ou}.

In part due to the failure of bi-parameter sparse domination methods, see \cite{BCRO} (see also \cite{BP} however), representation theorems are even more important
in bi-parameter than in one-parameter.
For example, the dyadic methods have proved very fruitful in connection with bi-parameter commutators and weighted analysis, see Holmes--Petermichl--Wick \cite{HPW}, Ou--Petermichl--Strouse \cite{OPS} and \cite{LMV:Bloom}. See also \cite{EA, ALMV}. In particular, the original bi-parameter weighted estimates of Fefferman--Stein \cite{FS} and Fefferman \cite{RF1, RF2}
were quite difficult in the sense that reaching the natural $A_p$ class instead of $A_{p/2}$ required an involved bootstrapping argument.

We prove a new version of the bi-parameter representation theorem \cite{Ma1} following the new modified shift idea.
In fact, we prove an $n$-linear version thus extending the theory of \cite{LMV} concerning bilinear bi-parameter SIOs with the usual H\"older modulus of continuity.
 An inherent complication of bi-parameter analysis is the appearance of certain hybrid combinations of shifts and paraproducts that are new
compared to the one-parameter case. Therefore, in the product space setting, unlike in the one-parameter case, we need to consider new modified operators involving paraproduct type philosophy.
We mention here that in all of our settings it is possible
to recover from our results an efficient representation of SIOs
with the usual (non-modified) dyadic operators. This is because we prove that all of our modified operators can be split into a \emph{sum} of standard dyadic model operators.

Some corollaries include the weighted
boundedness of bi-parameter CZOs and new estimates for various commutators $[b,T]\colon f \mapsto bTf - T(bf)$ with mild kernel regularity.
A reason why we care about weighted estimates is that, beyond their significant intrinsic interest,
they are of fundamental use in proving other estimates, like obtaining the full multilinear range of estimates
$\prod_{j=1}^n L^{p_j}\to L^r$, $\sum_j 1/p_j = 1/r$, $1 < p_j \le \infty$, $1/n < r < \infty$,
from a single tuple $(p_1, \ldots, p_n, r)$. The very powerful multilinear extrapolation results -- see e.g.
\cite{DU, GM, LMO, LMMOV, Nieraeth} -- are behind this. In the product space setting this viewpoint is particularly useful as many of the classical one-parameter
tools are completely missing. On the other hand,  commutator estimates appear all over analysis implying e.g.
factorizations for Hardy functions \cite{CRW}, certain div-curl lemmas relevant in compensated compactness, and were recently connected to
the Jacobian problem $Ju = f$ in $L^p$ (see \cite{HyCom}).

In the linear case the next theorem extends \cite{HPW} and in the bilinear case \cite{LMV}.
\begin{thm}
Let $p_j \in (1, \infty)$ and $1/r:= \sum_{j=1}^n 1/p_j$. Let $w_j \in A_{p_j}$ be bi-parameter weights on the product space $\R^d = \R^{d_1} \times \R^{d_2}$
and define $w := \prod_{j=1}^n w_j^{r/p_j}$.
Suppose that $T$ is an $n$-linear bi-parameter $(\omega_1, \omega_2)$-CZO.
If either
\begin{enumerate}
\item $n = 1$ and $\omega_i \in \operatorname{Dini}_{1/2}$, or
\item $n \ge 2$ and $\omega_i \in \operatorname{Dini}_{1}$,
\end{enumerate}
we have
\begin{equation}\label{eq:eq6}
\|T(f_1, \ldots, f_n)\|_{L^r(w)} \lesssim \prod_{j=1}^n \|f_j\|_{L^{p_j}(w_j)}
\end{equation}
and
$$
\| [b_m,\cdots[b_2, [b_1, T]_{k_1}]_{k_2}\cdots]_{k_m}(f_1, \ldots, f_n)\|_{L^r(w)} \lesssim \prod_{i=1}^m\|b_i\|_{\bmo}  \prod_{j=1}^n \|f_j\|_{L^{p_j}(w_j)},
$$
where $\bmo$ is the little $\BMO$ space and 
$$
[b,T]_k(f_1, \ldots, f_n) = bT(f_1, \ldots, f_n) - T(f_1, \ldots, f_{k-1}, bf_k, f_{k+1}, \ldots, f_n).
$$
\end{thm}
In the genuinely multilinear case $n\ge 2$ we were not able to establish these weighted bounds with the lowest kernel regularity $\omega_i \in \operatorname{Dini}_{1/2}$.
However, in the unweighted case $w_1 = \cdots = w_n = 1$ the estimate \eqref{eq:eq6} holds for all $n$ even with $\omega_i \in \operatorname{Dini}_{1/2}$ at least in the Banach range $r > 1$ (see Equation \eqref{eq:Dini12nlinBRange}).

In the linear case we also establish various two-weight Bloom type estimates for commutators. The result (1) extends \cite{LMV:Bloom2} and the result (2)
extends \cite{HPW} and \cite{LMV:Bloom}.
\begin{thm}
Suppose that $\R^d = \R^{d_1} \times \R^{d_2}$ is the underlying bi-parameter space, $p \in (1, \infty)$, $\mu, \lambda \in A_p(\R^d)$ are bi-parameter weights and
$\nu = \mu^{1/p} \lambda^{-1/p} \in A_2(\R^d)$ is the Bloom weight.
\begin{enumerate}
\item If $T_i$, $i = 1,2$, is a one-parameter $\omega_i$-CZO on $\R^{d_i}$, where $\omega_i \in \operatorname{Dini}_{3/2}$, then
$$
\| [T_1, [T_2, b]] \|_{L^p(\mu) \to L^p(\lambda)} \lesssim \|b\|_{\BMO_{\textup{prod}}(\nu)}.
$$
\item Suppose that $T$ is a bi-parameter $(\omega_1, \omega_2)$-CZO. Then we have
$$
\| [b_m,\cdots[b_2, [b_1, T]]\cdots]\|_{L^p(\mu) \to L^p(\lambda)} \lesssim \prod_{j=1}^m\|b_j\|_{\bmo(\nu^{1/m})}
$$
if one of the following conditions holds:
\begin{enumerate}
\item $T$ is paraproduct free and $\omega_i \in \operatorname{Dini}_{m/2+1}$;
\item $m=1$ and $\omega_i \in \operatorname{Dini}_{3/2}$;
\item $\omega_i \in \operatorname{Dini}_{m+1}$.
\end{enumerate}
\end{enumerate}
\end{thm}
See again the main text for all of the definitions and for some additional results.
These Bloom-style two-weight estimates have recently been one of the main lines of development concerning commutators, see e.g.
 \cite{EA, ALMV, HLW, HPW, LOR1, LOR2, LMV:Bloom, LMV:Bloom2} for a non-exhaustive list.
\subsection*{Acknowledgements}
We thank Tuomas Hytönen and Kangwei Li for useful discussions.

\section{Basic notation and fundamental estimates}\label{sec:def}

Throughout this paper $A\lesssim B$ means that $A\le CB$ with some constant $C$ that we deem unimportant to track at that point.
We write $A\sim B$ if $A\lesssim B\lesssim A$.

\subsection*{Dyadic notation}
Given a dyadic grid $\calD$, $I \in \calD$ and $k \in \Z$, $k \ge 0$, we use the following notation:
\begin{enumerate}
\item $\ell(I)$ is the side length of $I$.
\item $I^{(k)} \in \calD$ is the $k$th parent of $I$, i.e., $I \subset I^{(k)}$ and $\ell(I^{(k)}) = 2^k \ell(I)$.
\item $\ch(I)$ is the collection of the children of $I$, i.e., $\ch(I) = \{J \in \calD \colon J^{(1)} = I\}$.
\item $E_I f=\langle f \rangle_I 1_I$ is the averaging operator, where $\langle f \rangle_I = \fint_{I} f = \frac{1}{|I|} \int _I f$.
\item $E_{I, k}f$ is defined via
$$
E_{I,k}f = \sum_{\substack{J \in \calD \\ J^{(k)}=I}}E_J f.
$$ 
\item $\Delta_If$ is the martingale difference $\Delta_I f= \sum_{J \in \ch (I)} E_{J} f - E_{I} f$.
\item $\Delta_{I,k} f$ is the martingale difference block
$$
\Delta_{I,k} f=\sum_{\substack{J \in \calD \\ J^{(k)}=I}} \Delta_{J} f.
$$
\item $P_{I,k}f$ is the following sum of martingale difference blocks
$$
P_{I,k}f = \sum_{j=0}^{k} \Delta_{I,j} f
=\sum_{\substack{J \in \calD \\ J \subset I \\ \ell(J) \ge 2^{-k}\ell(I)}} \Delta_{J} f.
$$
\end{enumerate}
A fundamental fact is that we have the square function estimate
\begin{equation}\label{eq:SF}
\|S_{\calD} f\|_{L^p} \sim \|f\|_{L^p}, \qquad p \in (1,\infty), \,\, S_{\calD}f := \Big( \sum_{I \in \calD} |\Delta_I f|^2 \Big)^{1/2}.
\end{equation}
See e.g. \cite{CMP, CWW} for even weighted $\|S_{\calD} f\|_{L^p(w)} \sim \|f\|_{L^p(w)}$, $w \in A_p$, square function estimates and their history.
A weight $w$ (i.e. a locally integrable a.e. positive function) belongs to the weight class $A_p(\R^d)$, $1 < p < \infty$, if
$$
[w]_{A_p(\R^d)} := \sup_{Q} \frac 1{|Q|}\int_Q w   \Bigg( \frac 1{|Q|}\int_Q w^{1-p'}\Bigg)^{p-1} < \infty,
$$
where the supremum is taken over all cubes $Q \subset \R^d$.

We will also have use for the Fefferman--Stein inequality
$$
\Big\| \Big( \sum_{k} |Mf_k|^2 \Big)^{1/2} \Big\|_{L^p} \lesssim \Big\| \Big( \sum_{k} |f_k|^2 \Big)^{1/2} \Big\|_{L^p}, \qquad p \in (1,\infty),
$$
where $M$ is the Hardy--Littlewood maximal function.
However, most of the time we can make do with the lighter Stein's inequality
$$
\Big\| \Big( \sum_{I \in \calD} |E_I f_I|^2 \Big)^{1/2} \Big\|_{L^p} \lesssim \Big\| \Big( \sum_{I \in \calD} |f_I|^2 \Big)^{1/2} \Big\|_{L^p}, \qquad p \in (1,\infty).
$$
The distinction is relevant, for example, in $\UMD$-valued analysis. We will introduce the required vector-valued machinery later.

For an interval $J \subset \R$ we denote by $J_{l}$ and $J_{r}$ the left and right
halves of $J$, respectively. We define $h_{J}^0 = |J|^{-1/2}1_{J}$ and $h_{J}^1 = |J|^{-1/2}(1_{J_{l}} - 1_{J_{r}})$.
Let now $I = I_1 \times \cdots \times I_d \subset \R^d$ be a cube, and define the Haar function $h_I^{\eta}$, $\eta = (\eta_1, \ldots, \eta_d) \in \{0,1\}^d$, by setting
\begin{displaymath}
h_I^{\eta} = h_{I_1}^{\eta_1} \otimes \cdots \otimes h_{I_d}^{\eta_d}.
\end{displaymath}
If $\eta \ne 0$ the Haar function is cancellative: $\int h_I^{\eta} = 0$. We exploit notation by suppressing the presence of $\eta$, and write $h_I$ for some $h_I^{\eta}$, $\eta \ne 0$. Notice that for $I \in \calD$ we have $\Delta_I f = \langle f, h_I \rangle h_I$ (where the finite $\eta$ summation is suppressed), $\langle f, h_I\rangle := \int fh_I$.

\section{One-parameter singular integrals}\label{sec:1par}

Let $\omega$ be a modulus of continuity: an increasing and subadditive function with $\omega(0) = 0$. 
A relevant quantity is
the modified Dini condition
\begin{equation}\label{eq:Dini}
\|\omega\|_{\operatorname{Dini}_{\alpha}} := \int_0^1 \omega(t) \Big( 1 + \log \frac{1}{t} \Big)^{\alpha} \frac{dt}{t}, \qquad \alpha \ge 0.
\end{equation}
In practice, the quantity \eqref{eq:Dini} arises as follows:
\begin{equation}\label{eq:diniuse}
\sum_{k=1}^{\infty} \omega(2^{-k}) k^{\alpha} = \sum_{k=1}^{\infty} \frac{1}{\log 2} \int_{2^{-k}}^{2^{-k+1}} \omega(2^{-k}) k^{\alpha} \frac{dt}{t} \lesssim \int_0^1 \omega(t) \Big( 1 + \log \frac{1}{t} \Big)^{\alpha} \frac{dt}{t}.
\end{equation}

For many standard arguments $\alpha = 0$ is enough. For the $T1$ type arguments we will -- at the minimum -- always need $\alpha = 1/2$. When we do $\UMD$-extensions beyond function lattices, we will need a bit higher $\alpha$ depending on the so-called type and cotype constants of the underlying $\UMD$ space $X$.
\subsection*{Multilinear singular integrals}
A function
$$
K \colon \R^{(n+1)d} \setminus \Delta \to \C, \qquad \Delta=\{x = (x_1, \dots, x_{n+1}) \in\R^{(n+1)d} \colon x_1=\dots =x_{n+1}\},
$$
is called an $n$-linear $\omega$-Calder\'on--Zygmund kernel if it holds that
$$
|K(x)| \le \frac{C_{K}}{\Big(\sum_{m=1}^{n} |x_{n+1}-x_m|\Big)^{dn}},
$$
and for all $j \in \{1, \dots, n+1\}$ it holds that
$$
|K(x)-K(x')| \le \omega\Big( \frac{|x_j-x_j'| }{ \sum_{m=1}^{n} |x_{n+1}-x_m|} \Big) 
\frac{1}{\Big(\sum_{m=1}^{n} |x_{n+1}-x_m|\Big)^{dn}}
$$
whenever $x=(x_1, \dots, x_{n+1}) \in\R^{(n+1)d} \setminus \Delta$ and 
$x'=(x_1, \dots, x_{j-1},x_j',x_{j+1},\dots x_{n+1}) \in\R^{(n+1)d}$ satisfy
$$
|x_j-x_j'| \le 2^{-1} \max_{1 \le m \le n} |x_{n+1}-x_m|.
$$

\begin{defn}
An $n$-linear operator $T$ defined on a suitable class of functions -- e.g. on the linear combinations of cubes --
is an $n$-linear $\omega$-SIO with an associated kernel $K$, if we have
$$
\langle T(f_1, \dots, f_n), f_{n+1} \rangle=
\int_{\R^{(n+1)d}}  K(x_{n+1},x_1, \dots, x_n)\prod_{j=1}^{n+1} f_j(x_j) \ud x
$$
whenever $\supp f_i \cap \supp f_j = \emptyset$ for some $i \not= j$.
\end{defn}

\begin{defn}
We say that $T$ is an $n$-linear $\omega$-CZO if the following conditions hold:
\begin{itemize}
\item $T$ is an $n$-linear $\omega$-SIO.
\item We have that
$$
\| T^{m*}(1, \ldots, 1) \|_{\BMO} := \sup_{\calD}
\sup_{I \in \calD} \Big( \frac{1}{|I|} \sum_{\substack{J \in \calD \\ J \subset I}} |\langle T^{m*}(1, \ldots, 1), h_J\rangle|^2 \Big)^{1/2} < \infty
$$
for all $m \in \{0, \ldots, n\}$. Here $T^{0*} := T$, $T^{m*}$ denotes the $m$th, $m \in \{1,\ldots,n\}$, adjoint
$$
\langle T(f_1, \dots, f_n), f_{n+1} \rangle
= \langle T^{m*}(f_1, \dots, f_{m-1}, f_{n+1}, f_{m+1}, \dots, f_n), f_m \rangle
$$
of $T$, and the pairings $\langle T^{m*}(1, \ldots, 1), h_J\rangle$ have a standard $T1$ type definition with the aid of the kernel $K$.
\item We have that
$$
\|T\|_{\WBP} := \sup_{\calD} \sup_{I \in \calD}
|I|^{-1} |\langle T(1_I, \ldots, 1_I), 1_I\rangle| < \infty.
$$
\end{itemize}
\end{defn}

\subsection*{Model operators}
Let $i=(i_1, \dots, i_{n+1})$, $i_j \in \{0,1,\ldots\}$, and let $\calD$ be a dyadic lattice in $\R^d$.
An operator $S_i$ is called an $n$-linear dyadic shift if it has the form
\begin{equation} \label{e:shift}
\langle S_i(f_1,\dots,f_n), f_{n+1}\rangle
=\sum_{K \in \calD} \langle A_K(f_1, \ldots, f_n), f_{n+1}\rangle,
\end{equation}
where
$$
\langle A_K(f_1, \ldots, f_n), f_{n+1}\rangle= 
\sum_{\substack{I_1, \dots, I_{n+1} \in \calD \\ I_j^{(i_j)}=K}}
a_{K,(I_j)}\prod_{j=1}^{n+1} \langle f_j, \wt h_{I_j} \rangle.
$$
Here $a_{K,(I_j)} = a_{K, I_1, \ldots ,I_{n+1}}$ is a scalar satisfying the normalization
$$
|a_{K,(I_j)}| \le \frac{\prod_{j=1}^{n+1} |I_j|^{1/2}}{|K|^{n}},
$$
and there exist two indices $j_0,j_1 \in \{1, \ldots, n+1\}$, $j_0 \not =j_1$, so that $\wt h_{I_{j_0}}=h_{I_{j_0}}$, $\wt h_{I_{j_1}}=h_{I_{j_1}}$ and for the remaining indices $j \not \in \{j_0, j_1\}$ we have $\wt h_{I_j} \in \{h_{I_j}^0, h_{I_j}\}$.

A modified $n$-linear shift $Q_k$, $k \in \{1, 2, \ldots\}$, has the form
\begin{equation*} 
\langle Q_k(f_1,\dots,f_n), f_{n+1} \rangle
=\sum_{K \in \calD} \langle B_K(f_1, \ldots, f_n), f_{n+1} \rangle,
\end{equation*}
where
\begin{equation}\label{e:modnshift}
\begin{split}
\langle B_K(f_1, \ldots, &f_n), f_{n+1} \rangle \\
&= 
\sum_{I_1^{(k)} =  \cdots = I_{n+1}^{(k)} =K} a_{K,(I_j)} \Big[ \prod_{j=1}^n \langle f_j, h_{I_j}^0 \rangle - 
\prod_{j=1}^n \langle f_j, h_{I_{n+1}}^0 \rangle \Big] \langle f_{n+1}, h_{I_{n+1}} \rangle,
\end{split}
\end{equation}
or $B_K$ has one of the other symmetric forms, where the role of $f_{n+1}$ is replaced by some other $f_j$. The coefficients satisfy
the same (but now $|I_1| = \ldots = |I_{n+1}|$) normalization 
$$
|a_{K,(I_j)}| \le \frac{|I_1|^{(n+1)/2}}{|K|^{n}}.
$$

An $n$-linear dyadic paraproduct $\pi = \pi_{\calD}$ also has $n+1$ possible forms, but there is no complexity associated
to them. One of the forms is
$$
\langle \pi(f_1,\ldots,f_n), f_{n+1} \rangle
=\sum_{I \in \calD} a_I \prod_{j=1}^n \langle f_j \rangle_I \langle f_{n+1}, h_I \rangle,
$$
where the coefficients satisfy the usual BMO condition
\begin{equation}\label{eq:BMO}
\sup_{I_0 \in \calD} \Big( \frac{1}{|I_0|} \sum_{I \subset I_0} |a_I|^2 \Big)^{1/2} \le 1.
\end{equation}
In the remaining $n$ alternative forms the cancellative Haar function $h_I$ is in a different position.

When we represent a CZO we will have modified dyadic shifts $Q_k$, standard dyadic shifts
of the very special form $S_{k, \ldots, k}$ and paraproducts $\pi$. Dyadic shifts $S_{k, \ldots, k}$ are simply easier versions of the operators $Q_k$.
Paraproducts do not involve a complexity parameter and are thus inherently not even relevant for the kernel regularity considerations
(we just need their boundedness).

At least in the linear situation, we can easily unify the study of shifts $S_{k, \ldots, k}$ and modified shifts $Q_k$. This viewpoint could work in the multilinear generality also (with some tensor product formalism), but we did not pursue it. We can understand a modified linear shift to have the more general form 
$Q_k$, $k = 0, 1, \ldots$, where
\begin{equation}\label{eq:Q_kFORM1}
\langle Q_k f, g \rangle = \sum_{K \in \calD} \sum_{I^{(k)} = J^{(k)} = K} a_{IJK} \langle f, h_I \rangle \langle g, H_{I, J} \rangle
\end{equation}
or
\begin{equation}\label{eq:Q_kFORM2}
\langle Q_k f, g \rangle = \sum_{K \in \calD} \sum_{I^{(k)} = J^{(k)} = K} a_{IJK} \langle f, H_{I,J} \rangle \langle g, h_J \rangle,
\end{equation}
the constants $a_{IJK}$ satisfy the usual normalization and
and the functions $H_{I,J}$ satisfy
\begin{enumerate}
\item $H_{I,J}$ is supported on $I \cup J$ and constant on the children of $I$ and $J$, i.e., we have
$$
H_{I,J} = \sum_{L \in \ch(I) \cup \ch(J)} b_L 1_L,\,\, b_L \in \R,
$$
\item $|H_{I,J}| \le |I|^{-1/2}$ and
\item $\int H_{I,J} = 0$.
\end{enumerate}
In practice we have $H_{I,J} \in \{h_J^0 - h_I^0, h_I^0 - h_J^0, h_I, h_J\}$,
but this abstract form contains enough information to bound the operators.

An important property of the functions $H_{I,J}$ is the following. Let $I^{(k)} = J^{(k)} = K$. Since $H_{I,J}$ is constant on the children of $I$ and $J$
there holds that $\langle g, H_{I,J} \rangle=\langle E_{K,k+1}g, H_{I,J} \rangle$. We have the expansion
$$
E_{K,k+1}g= E_K g
+P_{K,k} g
$$ 
and the zero average of $H_{I,J}$ over $K$ implies that $\langle E_K g,H_{I,J} \rangle =0$. 
Thus, we have the key property
\begin{equation}\label{eq:gAndH}
\langle g, H_{I,J} \rangle=\langle P_{K,k}g, H_{I,J} \rangle.
\end{equation}
Also, there clearly holds that $\langle f, h_I \rangle= \langle\Delta_{K,k}f, h_I \rangle$. Analogous steps will be taken in the general
multilinear situation as well, even though the functions $H$ do not explicitly appear.

\subsection{The threshold $\alpha=1/2$}\label{sec:scalarval}
We quickly explain the role of the regularity threshold $\alpha = 1/2$, which appears naturally in the fundamental scalar-valued theory.
The estimate of the next lemma and identities like \eqref{eq:gAndH} are at the heart of the matter.
\begin{lem}\label{lem:PEst}
Let $p \in (1, \infty)$. There holds that 
$$
\Big\| \Big( \sum_{K \in \calD} |P_{K,k}f|^2 \Big)^{1/2}  \Big\|_{L^p} \sim \sqrt{k+1} \| f \|_{L^p}, \qquad k \in \{0,1,2, \ldots\}.
$$
\end{lem}

\begin{proof}
If $f_i \in L^p$ then
\begin{equation}\label{eq:SeqSF}
\Big \| \Big( \sum_{i=0}^\infty \sum_{I \in \calD} |\Delta_I f_i |^2 \Big)^{1/2} \Big\|_{L^p}
\sim \Big \| \Big( \sum_{i=0}^\infty | f_i |^2 \Big)^{1/2} \Big\|_{L^p}.
\end{equation}
This can be proved by using random signs and the Kahane-Khinchine inequality (scalar- and $\ell^2$-valued)
or by  extrapolating the corresponding weighted $L^2$ version of \eqref{eq:SeqSF}, which just follows from $\|S_{\calD} f\|_{L^2(w)} \sim \|f\|_{L^2(w)}$, $w \in A_2$.
Recall that the classical extrapolation theorem of Rubio de Francia says that if $\|h\|_{L^{p_0}(w)} \lesssim \|g\|_{L^{p_0}(w)}$ for some
$p_0 \in (1,\infty)$ and all $w \in A_{p_0}$, then $\|h\|_{L^{p}(w)} \lesssim \|g\|_{L^{p}(w)}$ for all $p \in (1,\infty)$ and all $w \in A_{p}$.

Let $K \in \calD$. We have that
$$
\sum_{I \in \calD} | \Delta_I P_{K,k} f |^2
= \sum_{j=0}^k |\Delta_{K,j} f |^2.
$$
Thus, \eqref{eq:SeqSF} gives that
\begin{equation*}
\begin{split}
\Big\| \Big( \sum_{K \in \calD} |P_{K,k}f|^2 \Big)^{1/2}  \Big\|_{L^p}
& \sim \Big\| \Big( \sum_{K \in \calD} \sum_{j=0}^k |\Delta_{K,j} f |^2 \Big)^{1/2}  \Big\|_{L^p} \\
&=\Big\| \Big( \sum_{j=0}^k \sum_{I \in \calD}  |\Delta_{I} f |^2 \Big)^{1/2}  \Big\|_{L^p} 
\sim \sqrt{k+1} \| f \|_{L^p}.
\end{split}
\end{equation*}
\end{proof}
\begin{prop}\label{prop:Qscalar}
Suppose that $Q_k$ is an $n$-linear modified shift. Let $1 < p_j < \infty$ 
with $\sum_{j=1}^{n+1} 1/p_j = 1$.
Then we have
\begin{equation}\label{eq:eq2}
|\langle Q_k(f_1, \ldots, f_n), f_{n+1}\rangle| \lesssim (k+1)^{1/2} \prod_{j=1}^{n+1} \| f_j \|_{L^{p_j}}.
\end{equation}
\end{prop}
\begin{proof}
We may assume that $Q_k$ has the form \eqref{e:modnshift}. Notice that if $I^{(k)} = K$ then we have
$$
\langle f, h_I^0 \rangle = \langle E_{K, k} f, h_I^0 \rangle = \langle E_K f + P_{K, k-1} f, h_I^0 \rangle. 
$$
Using this we have for $I_1^{(k)} =  \cdots = I_{n+1}^{(k)} =K$ that
\begin{align*}
\prod_{j=1}^n \langle f_j, h_{I_j}^0 \rangle = \langle P_{K, k-1} f_1, &h_{I_1}^0 \rangle \prod_{j=2}^n \langle f_j, h_{I_j}^0 \rangle
+ \langle E_K f_1, h_{I_1}^0 \rangle \langle P_{K, k-1} f_2, h_{I_2}^0 \rangle \prod_{j=3}^n \langle f_j, h_{I_j}^0 \rangle \\
&+\cdots +   \prod_{j=1}^{n-1} \langle E_K f_j, h_{I_j}^0 \rangle  \langle P_{K, k-1} f_{n}, h_{I_n}^0 \rangle + |I_{n+1}|^{n/2} \prod_{j=1}^{n} \langle f_j \rangle_K
\end{align*}
and
\begin{align*}
\prod_{j=1}^n \langle f_j, h_{I_{n+1}}^0 \rangle =& \langle P_{K, k-1} f_1, h_{I_{n+1}}^0 \rangle \prod_{j=2}^n \langle f_j, h_{I_{n+1}}^0 \rangle \\
&+ \langle E_K f_1, h_{I_{n+1}}^0 \rangle \langle P_{K, k-1} f_2, h_{I_{n+1}}^0 \rangle \prod_{j=3}^n \langle f_j, h_{I_{n+1}}^0 \rangle \\
&+\cdots +   \prod_{j=1}^{n-1} \langle E_K f_j, h_{I_{n+1}}^0 \rangle  \langle P_{K, k-1} f_{n}, h_{I_{n+1}}^0 \rangle + |I_{n+1}|^{n/2} \prod_{j=1}^{n} \langle f_j \rangle_K.
\end{align*}
We see that the last terms of these expansions cancel out in the difference $\prod_{j=1}^n \langle f_j, h_{I_j}^0 \rangle  - \prod_{j=1}^n \langle f_j, h_{I_{n+1}}^0 \rangle$.
It remains to estimate the others terms one by one. We pick the concrete (but completely representative) term $\langle E_K f_1, h_{I_1}^0 \rangle \langle P_{K, k-1} f_2, h_{I_2}^0 \rangle \prod_{j=3}^n \langle f_j, h_{I_j}^0 \rangle$ from the expansion of $\prod_{j=1}^n \langle f_j, h_{I_j}^0 \rangle$ and look at
\begin{equation}\label{eq:StaCom1}
\begin{split}
&\sum_K \sum_{I_1^{(k)} =  \cdots = I_{n+1}^{(k)} =K} \Big|a_{K,(I_j)} \langle E_K f_1, h_{I_1}^0 \rangle \langle P_{K, k-1} f_2, h_{I_2}^0 \rangle \prod_{j=3}^n \langle f_j, h_{I_j}^0 \rangle  \langle f_{n+1}, h_{I_{n+1}} \rangle\Big| \\
&\le \sum_K \sum_{I_1^{(k)} =  \cdots = I_{n+1}^{(k)} =K} \frac{1}{|K|^n} \int_{I_1} |E_K f_1|  \int_{I_2} |P_{K, k-1} f_2|  \prod_{j=3}^n \int_{I_j} |f_j|
\int_{I_{n+1}} |\Delta_{K,k} f_{n+1}| \\
&\le \sum_K \int_K \prod_{\substack{j=1 \\ j \ne 2}}^n \langle |f_j| \rangle_K \langle |P_{K, k-1} f_2| \rangle_K  |\Delta_{K,k} f_{n+1}| \\
&\le \int \prod_{\substack{j=1 \\ j \ne 2}}^n Mf_j \Big( \sum_K |M P_{K, k-1} f_2|^2 \Big)^{1/2} \Big(\sum_K |\Delta_{K,k} f_{n+1}|^2 \Big)^{1/2}.
\end{split}
\end{equation}
It remains to use H\"older's inequality, maximal function and square function estimates and Lemma \ref{lem:PEst}. We remark that
the estimate for $f_{n+1}$ is, indeed, just the usual square function estimate, since
$$
\sum_K |\Delta_{K,k} f_{n+1}|^2= \sum_I |\Delta_I f_{n+1}|^2.
$$

We now pick the corresponding term $\langle E_K f_1, h_{I_{n+1}}^0 \rangle \langle P_{K, k-1} f_2, h_{I_{n+1}}^0 \rangle \prod_{j=3}^n \langle f_j, h_{I_{n+1}}^0 \rangle$ from the expansion of $\prod_{j=1}^n \langle f_j, h_{I_{n+1}}^0 \rangle$ and look at
\begin{equation*}
\begin{split}
&\sum_K \sum_{I_1^{(k)} =  \cdots = I_{n+1}^{(k)} =K} \Big|a_{K,(I_j)} \langle E_K f_1, h_{I_{n+1}}^0 \rangle \langle P_{K, k-1} f_2, h_{I_{n+1}}^0 \rangle \prod_{j=3}^n \langle f_j, h_{I_{n+1}}^0 \rangle \langle f_{n+1}, h_{I_{n+1}} \rangle\Big| \\
&\le \sum_K \sum_{I_1^{(k)} =  \cdots = I_{n+1}^{(k)} =K} \frac{1}{|K|^n} \int_{I_{n+1}} |E_K f_1|  \int_{I_{n+1}} |P_{K, k-1} f_2|  \prod_{j=3}^n \int_{I_{n+1}} |f_j|
\int_{I_{n+1}} |\Delta_{K,k} f_{n+1}| .
\end{split}
\end{equation*}
Notice that
$$
\sum_{I_j^{(k)} = K} 1 =  \frac{1}{|I_{n+1}|} \sum_{I_j^{(k)} = K} |I_j| = \frac{|K|}{|I_{n+1}|}.
$$
We are thus left with
\begin{equation}\label{eq:StaCom2}
\begin{split}
& \sum_K \langle |f_1| \rangle_{K} \sum_{I_{n+1}^{(k)} = K} \int_{I_{n+1}}  \langle |P_{K, k-1} f_2|\rangle_{I_{n+1}}   \prod_{j=3}^n \langle |f_j| \rangle_{I_{n+1}}
|\Delta_{K,k} f_{n+1}| \\
&\le \sum_K \langle |f_1| \rangle_{K} \int_{K}  M P_{K, k-1} f_2 \prod_{j=3}^n M f_j \cdot
|\Delta_{K,k} f_{n+1}| \\
&\le \int \prod_{\substack{j=1 \\ j \ne 2}}^n Mf_j \Big( \sum_K |M P_{K, k-1} f_2|^2 \Big)^{1/2} \Big(\sum_K |\Delta_{K,k} f_{n+1}|^2 \Big)^{1/2}.
\end{split}
\end{equation}
This is the same upper bound as in the first case, and thus handled with in the same way. We are done.
\end{proof}
\begin{rem}
Proposition \ref{prop:Qscalar} considers only the Banach range boundedness of $Q_k$.
We can, in any case, upgrade the boundedness to the full range with standard methods when we consider CZOs.
This issue is discussed multiple times in what follows.
\end{rem}
Our representation of $T$ will involve $\sum_{k=0}^{\infty} \omega(2^{-k}) Q_k(f_1, \ldots, f_n)$, and thus by \eqref{eq:diniuse}
and Proposition \ref{prop:Qscalar} we will always need $\operatorname{Dini}_{1/2}$. The above proof readily generalises to so-called $\UMD$ function lattices.
In Section \ref{sec:UMD} we tackle the much deeper case of general $\UMD$ spaces.

\subsection*{Modified shifts are sums of standard shifts}
The standard linear shifts satisfy the \emph{complexity free} bound
$$
\|S_{i_1,i_2} f\|_{L^p} \lesssim \|f\|_{L^p}, \qquad p \in (1,\infty).
$$
Similar estimates hold in the multilinear generality -- for example, the following complexity free bilinear estimate is true
\[
\begin{split}
&  \quad \| S_{i_1,i_2,i_3}(f_1,f_2) \|_{L^{q_{3}}}
\lesssim   \| f_1\|_{L^{p_1}}\| f_2\|_{L^{p_2}},\\ &\forall 1<p_1,p_2\leq \infty, \quad \textstyle \frac{1}{2}<q_3<\infty, \quad  \frac{1}{p_1}+ \frac{1}{p_2}= \frac{1}{q_3}.
\end{split}
\]
Here we stated the estimates in the full range.
The quasi-Banach estimates cannot be proved via weak $(1,1)$ type estimates (or sparse domination), as those estimates are not complexity free. If desired, they can be obtained
with direct $L^p$ estimates or by extrapolation \cite{GM}.
In any case, we need roughly $nk$ shifts to represent an $n$-linear modified shift $Q_k$ as a sum of standard shifts, see Lemma \ref{lem:CompModStand} below. Therefore, these estimates lose to estimates like \eqref{eq:eq2}, and would lead to  $\operatorname{Dini}_{1}$.

On the other hand, more involved estimates,
such as some commutator estimates,
can be difficult to carry out directly with the operators $Q_k$. Bounds via the
route of representing $Q_k$ using ordinary shifts still lead to the modified Dini condition \eqref{eq:Dini} with some $\alpha$,
and this is still quite efficient. Therefore, Lemma \ref{lem:CompModStand} is of practical and philosophical use, but should not be resorted to
when more efficient estimates can be obtained by the direct study of the operators $Q_k$.

\begin{lem}\label{lem:CompModStand}
Let $Q_k$, $k \in \{1,2, \ldots\}$, be a modified $n$-linear shift of the form \eqref{e:modnshift}. Then for some $C\lesssim 1$ we have
$$
Q_k = \sum_{m=1}^n \sum_{i=0}^{k-1} S_{0, \ldots, 0, i, k, \ldots, k} -  C\sum_{m=1}^n \sum_{i=0}^{k-1} S_{0, \ldots, 0, 1, \ldots, 1, i},
$$
where in the first sum there are $m-1$ zeroes and in the second sum $m$ zeroes in the complexity of the shift.
\end{lem}
\begin{rem}
In the proof below we decompose various martingale differences using Haar functions, which strictly speaking
leads to the fact that there is an implicit dimensional summation in the above decomposition.
\end{rem}
\begin{proof}[Proof of Lemma \ref{lem:CompModStand}]
The underlying decomposition is, in part, more sophisticated than the one in the beginning of the proof of Proposition \ref{prop:Qscalar}. See
the multilinear collapse \eqref{eq:MultilinCollapse} and \eqref{eq:DefUm}.
This feels necessary
for this result -- moreover, we will later use this decomposition strategy when we do general $\UMD$-valued estimates.

Write $b_{K, (I_j)} = |I_1|^{n/2} a_{K, (I_j)}$ so that
$$ 
a_{K, (I_j)} \Big[ \prod_{j=1}^n \langle f_j, h_{I_j}^0 \rangle - \prod_{j=1}^n \langle f_j, h_{I_{n+1}}^0 \rangle \Big]
= b_{K, (I_j)} \Big[ \prod_{j=1}^n \langle f_j \rangle_{I_j} - \prod_{j=1}^n \langle f_j \rangle_{I_{n+1}} \Big].
$$
We then write
$$
 \prod_{j=1}^n \langle f_j \rangle_{I_j} - \prod_{j=1}^n \langle f_j \rangle_{I_{n+1}}
  = \Big[ \prod_{j=1}^n \langle f_j \rangle_{I_j} - \prod_{j=1}^n \langle f_j \rangle_{K} \Big]
  + \Big[ \prod_{j=1}^n \langle f_j \rangle_{K} - \prod_{j=1}^n \langle f_j \rangle_{I_{n+1}}\Big].
$$

We start working with the first term. Notice that
$$
\langle f_j \rangle_{I_j} = \langle E_{K, k} f_j \rangle_{I_j} 
= \langle P_{K, k-1} f_j \rangle_{I_j} + \langle f_j \rangle_K.
$$
Using this we can write
\begin{equation}\label{eq:EPSplitting}
\begin{split}
\prod_{j=1}^n \langle f_j \rangle_{I_j} - \prod_{j=1}^n \langle f_j \rangle_{K} &= 
\langle P_{K, k-1} f_1 \rangle_{I_1} \prod_{j=2}^n \langle f_j \rangle_{I_j}
+ \langle f_1 \rangle_K \langle P_{K, k-1} f_2 \rangle_{I_2} \prod_{j=3}^n \langle f_j \rangle_{I_j} \\
& \qquad +\cdots +   \prod_{j=1}^{n-1} \langle f_j \rangle_{K}  \langle P_{K, k-1} f_{n} \rangle_{I_n}.
\end{split}
\end{equation}
Consider now, for $m \in \{ 1, \ldots, n\}$, the following part of the modified shift
\begin{equation}\label{eq:DefAm}
\begin{split}
\langle A_m(f_1, \ldots, f_n), f_{n+1}\rangle := &\sum_K \sum_{I_1^{(k)} =  \cdots = I_{n+1}^{(k)} =K} b_{K, (I_j)} \\
&\times \prod_{j=1}^{m-1} \langle f_j \rangle_K \cdot \langle P_{K, k-1} f_m \rangle_{I_m}
\cdot \prod_{j=m+1}^n \langle f_j \rangle_{I_j} \cdot \langle f_{n+1}, h_{I_{n+1}} \rangle.
\end{split}
\end{equation}
Next, write
$$
\langle P_{K, k-1} f_m \rangle_{I_m} = \sum_{i=0}^{k-1} \sum_{L^{(i)} = K} \langle \Delta_L f_m \rangle_{I_m} =
\sum_{i=0}^{k-1} \sum_{L^{(i)} = K} \langle f_m, h_L \rangle \langle h_L \rangle_{I_m}.
$$
We can write $\langle A_m(f_1, \ldots, f_n), f_{n+1}\rangle$ in the form
\begin{align*}
\sum_{i=0}^{k-1}  \sum_{L^{(i)} = K}  \sum_{I_{m+1}^{(k)} =  \cdots = I_{n+1}^{(k)} =K} \Big(&\sum_{I_{1}^{(k)} =  \cdots = I_{m}^{(k)} =K}   b_{K, (I_j)} |K|^{-(m-1)/2}
\langle h_L \rangle_{I_m} |I_{n+1}|^{-(n-m)/2} \Big) \\
&\times   \prod_{j=1}^{m-1} \langle f_j, h_K^0 \rangle \cdot \langle f_m, h_L \rangle \cdot \prod_{j=m+1}^n \langle f_j, h_{I_j}^0 \rangle \cdot \langle f_{n+1}, h_{I_{n+1}} \rangle.
\end{align*}
Notice the normalization estimate
\begin{align*}
\sum_{\substack{ I_{1}^{(k)} =  \cdots = I_{m}^{(k)} =K \\ I_m \subset L}} &  |b_{K, (I_j)}| |K|^{-(m-1)/2}
|L|^{-1/2} |I_{n+1}|^{-(n-m)/2} \\ 
& \le \frac{  \prod_{j=1}^{m-1} |K|^{1/2} \cdot |L|^{1/2} \cdot \prod_{j=m+1}^{n+1} |I_j|^{1/2}}{|K|^n}.
\end{align*}
We also have two cancellative Haar functions, so for every $i \in \{0, \ldots, k-1\}$, the inner sum in $A_m$
is a standard $n$-linear shift of complexity $(0, \ldots, 0, i, k, \ldots, k)$, where the $i$ is in the $m$th slot:
$$
\langle A_m(f_1, \ldots, f_n), f_{n+1}\rangle = \sum_{i=0}^{k-1} \langle S_{0, \ldots, 0, i, k, \ldots, k}(f_1, \ldots, f_m), f_{n+1}\rangle.
$$

We now turn to the part of the modified shift associated with
\begin{equation}\label{eq:MultilinCollapse}
 \prod_{j=1}^n \langle f_j \rangle_{I_{n+1}} - \prod_{j=1}^n \langle f_j \rangle_{K} = 
 \sum_{i=0}^{k-1} \Big(  \prod_{j=1}^n \langle f_j \rangle_{I_{n+1}^{(i)}}
-\prod_{j=1}^n \langle f_j \rangle_{I_{n+1}^{(i+1)}}\Big).
 \end{equation}
Further, we write
\begin{equation}\label{eq:nCollapseRelated}
\begin{split}
\prod_{j=1}^n \langle f_j \rangle_{I_{n+1}^{(i)}} -\prod_{j=1}^n \langle f_j \rangle_{I_{n+1}^{(i+1)}}
&= \langle \Delta_{I_{n+1}^{(i+1)}} f_1 \rangle_{I_{n+1}^{(i)}}  \prod_{j=2}^n \langle f_j \rangle_{I_{n+1}^{(i)}} \\
&+ \langle f_1 \rangle_{I_{n+1}^{(i+1)}} \langle \Delta_{I_{n+1}^{(i+1)}} f_2  \rangle_{I_{n+1}^{(i)}} \prod_{j=3}^n \langle f_j \rangle_{I_{n+1}^{(i)}} \\
&+ \ldots + \prod_{j=1}^{n-1} \langle f_j \rangle_{I_{n+1}^{(i+1)}} \cdot \langle \Delta_{I_{n+1}^{(i+1)}} f_n  \rangle_{I_{n+1}^{(i)}}.
\end{split}
\end{equation}
Consider now, for $m \in \{ 1, \ldots, n\}$, the following part of the modified shift
\begin{equation}\label{eq:DefUm}
\begin{split}
\langle U_m(f_1, \ldots, f_n), f_{n+1}\rangle = 
&\sum_{i=0}^{k-1} \sum_K \sum_{I_1^{(k)} =  \cdots = I_{n+1}^{(k)} =K} b_{K, (I_j)}  \prod_{j=1}^{m-1}  \langle f_j \rangle_{I_{n+1}^{(i+1)}}  \\
& \times \langle \Delta_{I_{n+1}^{(i+1)}} f_m  \rangle_{I_{n+1}^{(i)}} \cdot \prod_{j=m+1}^n \langle f_j \rangle_{I_{n+1}^{(i)}}
\cdot \langle f_{n+1}, h_{I_{n+1}} \rangle.
\end{split}
\end{equation}
We can write $\langle U_m(f_1, \ldots, f_n), f_{n+1}\rangle$ in the form
\begin{align*}
\sum_{i=0}^{k-1} \sum_K \sum_{L^{(k-i)} = K} \sum_{I_{n+1}^{(i)} = L} &\Big( \sum_{I_1^{(k)} =  \cdots = I_{n}^{(k)} =K} b_{K, (I_j)} |L^{(1)}|^{-(m-1)/2} 
\langle h_{L^{(1)}} \rangle_{L} |L|^{-(n-m)/2} \Big) \\
&\times \prod_{j=1}^{m-1}  \langle f_j, h_{L^{(1)}}^0 \rangle \cdot \langle f_m, h_{L^{(1)}} \rangle \cdot \prod_{j=m+1}^n \langle f_j, h_L^0 \rangle \cdot \langle f_{n+1}, h_{I_{n+1}} \rangle.
\end{align*}
Notice the normalization estimate
\begin{align*}
\sum_{I_1^{(k)} =  \cdots = I_{n}^{(k)} =K} &|b_{K, (I_j)}| |L^{(1)}|^{-(m-1)/2} 
|L^{(1)}|^{-1/2} |L|^{-(n-m)/2} \\
&\lesssim \frac{ |L^{(1)}|^{m/2} |L|^{(n-m)/2} |I_{n+1}|^{1/2}}{|L^{(1)}|^n}.
\end{align*} 
Therefore, for some constant $C \lesssim 1$ we get that
$$
\langle U_m(f_1, \ldots, f_n), f_{n+1}\rangle = C \sum_{i=0}^{k-1} \langle S_{0, \ldots, 0, 1, \ldots, 1, i}(f_1, \ldots, f_m), f_{n+1}\rangle,
$$
where there are $m$ zeroes in $S_{0, \ldots, 0, 1, \ldots, 1, i}$.
\end{proof}
 
\subsection*{The representation theorem}
Let $\sigma = (\sigma^i)_{i \in \Z}$, where $\sigma^i \in \{0,1\}^d$. Let $\mathcal{D}_0$ be the standard dyadic grid on $\R^d$,
$$
\calD_0 := \{2^{-k}([0,1)^d+m) \colon k \in \Z, m \in \Z^d\}.
$$
We define the new dyadic grid 
$$
\mathcal{D}_{\sigma} = \Big\{I + \sum_{i:\, 2^{-i} < \ell(I)} 2^{-i}\sigma^i: \, I \in \mathcal{D}_0\Big\} = \{I + \sigma: \, I \in \mathcal{D}_0\},
$$
where we simply have defined $$
I + \sigma := I + \sum_{i:\, 2^{-i} < \ell(I)} 2^{-i}\sigma^i.
$$
It is straightforward that $\calD_{\sigma}$ inherits the key nestedness property of $\calD_0$: if $I,J \in \calD_{\sigma}$, then $I \cap J \in \{I,J,\emptyset\}$.
Moreover, there is a natural product probability measure $\mathbb{P}_{\sigma} = \mathbb{P}$ on $(\{0,1\}^d)^{\Z}$ -- this gives us the notion
of random dyadic grids $\sigma \mapsto \mathcal{D}_{\sigma}$ over which we take the expectation $\E_{\sigma}$ below.

\begin{rem}
The assumption $\omega \in \operatorname{Dini}_{1/2}$ in the theorem below is only needed to have a converging series.
The regularity is not explicitly used in the proof of the representation.
It is required due to the estimates 
of the model operators briefly discussed above. We will state the $T1$ type corollaries, including the $\UMD$-extensions, carefully later.
\end{rem}

\begin{thm}\label{thm:rep1par}
Suppose that $T$ is an $n$-linear $\omega$-CZO, where $\omega \in \operatorname{Dini}_{1/2}$. Then we have
$$
\langle T(f_1,\ldots,f_n), f_{n+1} \rangle = C_T \E_{\sigma} \sum_{k=0}^{\infty} \sum_{u=0}^{c_{d,n}} \omega(2^{-k}) \langle V_{k,u,\sigma}(f_1,\ldots,f_n), f_{n+1} \rangle,
$$
where $V_{k,u, \sigma}$ is always either a standard $n$-linear shift $S_{k, \ldots, k}$, a modified $n$-linear shift $Q_{k}$ or an $n$-linear paraproduct (this requires $k= 0$) in the grid $\calD_{\sigma}$. Moreover, we have
$$
|C_T| \lesssim \sum_{m=0}^n \| T^{m*}(1, \ldots, 1) \|_{\BMO} + \|T\|_{\WBP} + C_K + 1.
$$
\end{thm}
\begin{proof}
We begin with the decomposition
\begin{equation*}
\begin{split}
\langle T(f_1, \ldots, f_n),f_{n+1} \rangle
&= \E_{\sigma} \sum_{I_1, \ldots, I_{n+1} } \langle T(\Delta_{I_1}f_1, \ldots, \Delta_{I_n}f_n),\Delta_{I_{n+1}}f_{n+1} \rangle \\
&= \sum_{j=1}^{n+1}  \E_{\sigma} \sum_{ \substack{ I_1, \ldots, I_{n+1} \\ \ell(I_i) > \ell(I_j) \textup{ for } i \ne j}} \langle T(\Delta_{I_1}f_1, \ldots, \Delta_{I_n}f_n),\Delta_{I_{n+1}}f_{n+1} \rangle
+  \E_{\sigma} R_{\sigma},
\end{split}
\end{equation*}
where $I_1, \ldots, I_{n+1} \in \calD_\sigma$ for some $\sigma \in (\{0,1\}^d)^{\Z}$.
We deal with the remainder term $R_{\sigma}$ later, and now focus on dealing with one of the main terms
$$
\Sigma_{j,\sigma} = \sum_{ \substack{ I_1, \ldots, I_{n+1} \\ \ell(I_i) > \ell(I_j) \textup{ for } i \ne j}} \langle T(\Delta_{I_1}f_1, \ldots, \Delta_{I_n}f_n),\Delta_{I_{n+1}}f_{n+1} \rangle,
$$
where $j \in \{1, \ldots, n+1\}$.

The main terms are symmetric, and we choose to handle $\Sigma_{\sigma} := \Sigma_{n+1,\sigma}$. 
After collapsing the sums
\begin{equation}\label{eq:collapse}
\sum_{I_i \colon\ell(I_i) >  \ell(I_{n+1})}  \Delta_{I_i} f_i  = \sum_{I_i \colon \ell(I_i) = \ell(I_{n+1})} E_{I_i} f_i,
\end{equation}
we have
$$
\Sigma_{\sigma} = \sum_{\ell(I_1) = \cdots = \ell(I_{n+1})} T(E_{I_1}f_1, \ldots, E_{I_n}f_n),\Delta_{I_{n+1}}f_{n+1} \rangle.
$$
Further, we write
\begin{align*}
\langle T(E_{I_1}&f_1, \ldots, E_{I_n}f_n),\Delta_{I_{n+1}}f_{n+1} \rangle \\
&= \langle T( h_{I_1}^0, \ldots, h_{I_n}^0), h_{I_{n+1}} \rangle
 \prod_{j=1}^n \langle f_j, h_{I_j}^0 \rangle  \langle f_{n+1}, h_{I_{n+1}} \rangle \\
 &= \langle T( h_{I_1}^0, \ldots, h_{I_n}^0), h_{I_{n+1}} \rangle \Big[ \prod_{j=1}^n \langle f_j, h_{I_j}^0 \rangle - 
\prod_{j=1}^n \langle f_j, h_{I_{n+1}}^0 \rangle \Big] \langle f_{n+1}, h_{I_{n+1}} \rangle \\
&+ \langle T( 1_{I_1}, \ldots, 1_{I_n}), h_{I_{n+1}} \rangle \prod_{j=1}^n  \langle f_j \rangle_{I_{n+1}} \langle f_{n+1}, h_{I_{n+1}} \rangle.
\end{align*}
We define the abbreviation
$$
\varphi_{I_1,\ldots,I_{n+1}}
:=\langle T( h_{I_1}^0, \ldots, h_{I_n}^0), h_{I_{n+1}} \rangle \Big[ \prod_{j=1}^n \langle f_j, h_{I_j}^0 \rangle - 
\prod_{j=1}^n \langle f_j, h_{I_{n+1}}^0 \rangle \Big] \langle f_{n+1}, h_{I_{n+1}} \rangle.
$$ 
If we now sum over $I_1, \ldots, I_{n+1}$ we may express $\Sigma_{\sigma}$ in the form
$$
\Sigma_{\sigma} = \sum_{\ell(I_1) = \cdots = \ell(I_{n+1})} \varphi_{I_1,\ldots,I_{n+1}} + \sum_I
\langle T( 1, \ldots, 1), h_I \rangle \prod_{j=1}^n  \langle f_j \rangle_I \langle f_{n+1}, h_I \rangle = \Sigma_{\sigma}^1 + \Sigma_{\sigma}^2,
$$
where we recognize that the second term $\Sigma_{\sigma}^2$ is a paraproduct. Thus, we only need to continue working with $\Sigma_{\sigma}^1$.

Since $\varphi_{I, \ldots, I}=0$, we have that
\begin{align*}
\Sigma_{\sigma}^1 &=  \sum_{ \substack{m_1, \ldots, m_n \in \Z^d \\(m_1, \ldots, m_n) \not = (0,\ldots, 0)}} \sum_I
\varphi_{I+m_1 \ell(I),\ldots, I+m_n\ell(I),I} \\
&=  \sum_{k=2}^\infty \sum_{ \substack{\max |m_j| \\ \in (2^{k-3}, 2^{k-2}]}} 
\sum_I \varphi_{I+m_1 \ell(I),\ldots, I+m_n\ell(I),I}.
\end{align*}
As in \cite{GH} we say that $I$ is $k$-good for $k \ge 2$ -- and denote this by $I \in \calD_{\sigma, \good}(k)$ -- if $I \in \calD_{\sigma}$ satisfies
\begin{equation}\label{eq:DefkGood}
d(I, \partial I^{(k)}) \ge \frac{\ell(I^{(k)})}{4} = 2^{k-2} \ell(I).
\end{equation}
Notice that for all $I \in \calD_0$ we have
$$
\mathbb{P}( \{ \sigma\colon I + \sigma \in \calD_{\sigma, \good}(k) \})  = 2^{-d}.
$$
Thus, by the independence of the position of $I$ and the $k$-goodness of $I$ we have
\begin{equation}\label{eq:AddGoodness}
\begin{split}
\E_{\sigma} \Sigma_{\sigma}^1
&=2^d \E_{\sigma}  \sum_{k=2}^\infty \sum_{ \substack{\max |m_j| \\ \in (2^{k-3}, 2^{k-2}]}} 
\sum_{I \in  \calD_{\sigma, \good}(k)} \varphi_{I+m_1 \ell(I),\ldots, I+m_n\ell(I),I} \\
&= C2^d \E_{\sigma}  \sum_{k=2}^\infty  \omega(2^{-k}) \langle  Q_k (f_1,\ldots, f_n),f_{n+1} \rangle,
\end{split}
\end{equation}
where 
$$
\langle Q_k (f_1,\ldots, f_n),f_{n+1} \rangle :=  \frac{1}{C \omega(2^{-k})} \sum_{ \substack{\max |m_j| \\ \in (2^{k-3}, 2^{k-2}]}}
\sum_{I \in  \calD_{\sigma, \good}(k)} \varphi_{I+m_1 \ell(I),\ldots, I+m_n\ell(I),I}
$$
and $C$ is large enough.

Next, the key implication of the $k$-goodness is that
\begin{equation}\label{eq:GoodnessImplies}
(I + m\ell(I))^{(k)} = I^{(k)} =: K
\end{equation}
if $|m| \le 2^{k-2}$ and $I \in \calD_{\sigma, \good}(k)$.
Indeed, notice that e.g. $c_I + m\ell(I) \in [I + m\ell(I)] \cap K$ (so that $ [I + m\ell(I)] \cap K \ne \emptyset$ which is enough) as
$$
d(c_I + m\ell(I), K^c) \ge d(c_I, K^c) - |m|\ell(I) > d(I, \partial K) - |m|\ell(I) \ge 2^{k-2}\ell(I) - 2^{k-2}\ell(I) = 0.
$$
Therefore, to conclude that $Q_k$ is a modified $n$-linear shift it only remains to prove the normalization
\begin{equation}\label{eq:normalization}
\frac{|\langle T( h_{I + m_1\ell(I)}^0, \ldots, h_{I + m_n\ell(I)}^0), h_I \rangle |}{\omega(2^{-k})}
\lesssim \frac{|I|^{(n+1)/2}}{|K|^{n}}.
\end{equation}

Suppose first that $k \sim 1$. Recall that $(m_1, \ldots, m_n) \not=(0,0)$ and  assume for example that $m_1 \not=0$. 
We have using the size estimate of the kernel that
\begin{equation}\label{eq:NearbyEst}
\begin{split}
|\langle T( h_{I + m_1\ell(I)}^0&, \ldots, h_{I + m_n\ell(I)}^0), h_I \rangle | \\
&\lesssim \int_{\R^{(n+1)d}} \frac{h^0_{I+m_1\ell(I)}(x_1) \cdots h^0_{I+m_n\ell(I)}(x_n) |h_I(x_{n+1})|}{\Big(\sum_{m=1}^{n} |x_{n+1}-x_m|\Big)^{dn}}  \ud x \\
&\lesssim \frac{1}{|I|^{(n+1)/2}} \int_{CI \setminus I} \int_{I} \frac{\ud x_1 \ud x_{n+1}}{|x_{n+1}-x_1|^{d}}
\lesssim \frac{1}{|I|^{(n-1)/2}},
\end{split}
\end{equation}
where in the second step we repeatedly used estimates of the form
\begin{equation}\label{eq:lesssimInt}
\int_{\R^d} \frac{\ud z}{(r+|z_0-z|)^{d+\alpha}}  \lesssim \frac{1}{r^{\alpha}}.
\end{equation}
Notice that this is the right upper bound \eqref{eq:normalization} in the case $k  \sim 1 $.

Suppose then that $k$ is large enough so that we can use the continuity assumption of the kernel.
In this case we have that if $x_{n+1} \in I$ and $x_1 \in I+m_1 \ell(I), \ldots, x_n \in I+m_n \ell(I)$, then
$\sum_{m=1}^{n} |x_{n+1}-x_m| \sim 2^k\ell(I)=\ell(K)$. Thus, there holds that
\begin{equation}\label{eq:FarEst}
\begin{split}
&|\langle T( h_{I + m_1\ell(I)}^0, \ldots, h_{I + m_n\ell(I)}^0), h_I \rangle |\\
&= \Big |  \int_{\R^{(n+1)d}} (K(x_{n+1}, x_1, \ldots, x_n)-K(c_I,x_1, \ldots, x_n)) \prod_{j=1}^n h^0_{I+m_j\ell(I)}(x_j) h_I(x_{n+1}) \ud x \Big | \\
&\lesssim  \int_{\R^{(n+1)d}} \omega ( 2^{-k}) \frac{1}{|K|^n}  \prod_{j=1}^n h^0_{I+m_j\ell(I)}(x_j) |h_I(x_{n+1})| \ud x \\
&= \omega ( 2^{-k}) \frac{1}{|K|^n} |I|^{n+1} |I|^{-(n+1)/2}
=\omega ( 2^{-k}) \frac{|I|^{(n+1)/2}}{|K|^n}.
\end{split}
\end{equation}
We have proved \eqref{eq:normalization}.
This ends our treatment of $\E_{\sigma} \Sigma_{\sigma}$.

We now only need to deal with the remainder term $\E_{\sigma} R_{\sigma}$. Write
$$
R_{\sigma} = \sum_{ (I_1, \ldots, I_{n+1}) \in \calI_{\sigma} }  \langle T(\Delta_{I_1}f_1, \ldots, \Delta_{I_n}f_n),\Delta_{I_{n+1}}f_{n+1} \rangle,
$$
where each $(I_1, \ldots, I_{n+1}) \in \calI_{\sigma}$ satisfies that 
if $j \in \{1, \ldots, n+1\}$ is such that $\ell(I_j) \le \ell(I_i)$ for all $i \in \{1, \ldots, n+1\}$, then $\ell(I_j) = \ell(I_{i_0})$ for at least one $i_0  \in \{1, \ldots, n+1\} \setminus \{j\}$.
The point why the remainder is simpler than the main terms
is that we can split this summation so that there are always at least two sums which we do not need to collapse -- that means we will readily have
two cancellative Haar functions. To give the idea, it makes sense to explain the bilinear case $n=2$.
In this case we can, in a natural way, decompose
\begin{align*}
 \sum_{ (I_1, I_2, I_3) \in \calI_{\sigma} } 
=  \sum_{\ell(I_2) = \ell(I_3) < \ell(I_1)} +  \sum_{\ell(I_1) = \ell(I_3) < \ell(I_2)} +  \sum_{\ell(I_1) = \ell(I_2) < \ell(I_3)} 
+  \sum_{\ell(I_1) = \ell(I_2) = \ell(I_3)},
\end{align*}
which -- after collapsing the relevant sums -- gives that
\begin{align*}
R_{\sigma} =  &\sum_{\ell(I_1) = \ell(I_2) = \ell(I_3)} \big(\langle T(E_{I_1}f_1,\Delta_{I_2}f_2),\Delta_{I_3}f_3 \rangle
+ \langle T(\Delta_{I_1}f_1,E_{I_2}f_2),\Delta_{I_3}f_3 \rangle \\ &+ \langle T(\Delta_{I_1}f_1,\Delta_{I_2}f_2),E_{I_3}f_3 \rangle
+ \langle T(\Delta_{I_1}f_1,\Delta_{I_2}f_2),\Delta_{I_3}f_3 \rangle \big) = \sum_{i=1}^4 R_{\sigma}^i.
\end{align*}
These are all handled similarly (the point is that there are at least two martingale differences remaining in all of them) so we look for example at 
\begin{equation*}
\begin{split}
R_{\sigma}^2
&=\sum_{\ell(I_1)=\ell(I_2)=\ell(I_3)}\langle T(\Delta_{I_1}f_1,E_{I_2}f_2),\Delta_{I_3}f_3 \rangle \\
&=\sum_{\substack{\ell(I_1)=\ell(I_2)=\ell(I_3) \\ I_1 \not= I_3 \text{ or } I_2 \not= I_3}}\langle T(\Delta_{I_1}f_1,E_{I_2}f_2),\Delta_{I_3}f_3 \rangle 
+\sum_{I}\langle T(\Delta_{I}f_1,E_{I}f_2),\Delta_{I}f_3 \rangle.
\end{split}
\end{equation*}
We can represent these terms as sums of standard bilinear shifts of the form $S_{k,k,k}$.
The first term is handled exactly like $\Sigma_{\sigma}^1$ above.   

The second term  is readily a zero complexity shift.  To prove the estimate for the coefficient we write
$\langle T(h_I, h^0_I), h_I \rangle $ as the sum of
\begin{equation}\label{eq:DiagSplit1}
\sum_{\substack{I_1,I_2,I_3 \subset I \\ \ell(I_i)=2^{-1} \ell(I) \\ I_1 \not= I_3 \text{ or } I_2 \not=I_3}}
\langle T(1_{I_1}h_I, 1_{I_2}h^0_I), 1_{I_3}h_I \rangle 
\end{equation}
and 
\begin{equation}\label{eq:DiagSplit2}
 \sum_{\substack{I' \subset I \\ \ell(I')=2^{-1} \ell(I)}}
\langle T(1_{I'}h_I, 1_{I'}h^0_I), 1_{I'}h_I \rangle.
\end{equation}
There are $\lesssim 1$ terms in both sums. In \eqref{eq:DiagSplit1} we use the size of the kernel of $T$
and in \eqref{eq:DiagSplit2} the weak boundedness property
$$
| \langle T(1_I,1_I), 1_I \rangle | \lesssim | I |.
$$

The general $n$-linear remainder term $R_{\sigma}$ is analogous and only yields standard $n$-linear shifts $S_{k, \ldots, k}$. We are done.
\end{proof}

We record the following $T1$ type corollary. See the proof of Theorem \ref{thm:multiCZOUMD} below for an explanation on how to obtain the full range of boundedness.
\begin{cor}\label{cor:multiCZO}
Suppose that $T$ is an $n$-linear $\omega$-CZO with $\omega \in \operatorname{Dini}_{1/2}$.
Then for all exponents $1 < p_1, \ldots, p_n \le \infty$ and $1 / q_{n+1} = \sum_{j=1}^n 1/p_j > 0$
we have
$$
\|T(f_1, \ldots, f_n) \|_{L^{q_{n+1}}} \lesssim \prod_{j=1}^{n} \| f_j \|_{L^{p_j}}.
$$
\end{cor}

\section{$\UMD$-valued extensions of singular integrals}\label{sec:UMD}
\subsection*{Preliminaries of Banach space theory}
An extensive treatment of Banach space theory is given in the books \cite{HNVW1, HNVW2} by Hyt\"onen, van Neerven, Veraar and Weis.

We say that $\{\varepsilon_i\}_i$ is a \emph{collection of independent random signs}, where $i$ runs over some index set, if there exists a probability space $(\calM, \mu)$ so that $\varepsilon_i \colon \calM \to \{-1,1\}$, 
$\{\varepsilon_i\}_i$ is independent and
$\mu(\{\varepsilon_i=1\})=\mu(\{\varepsilon_i=-1\})=1/2$. 
Below, $\{\varepsilon_i\}_i$ will always denote a collection of independent random signs.

Suppose $X$ is a Banach space. We denote the underlying norm by $| \cdot |_X$.
The Kahane-Khintchine inequality says that
for all $x_1,\dots, x_N \in X$ and $p,q \in (0,\infty)$ there holds that
\begin{equation}\label{eq:KK}
\Big(\E \Big | \sum_{i=1}^N \varepsilon_i x_i \Big |_X^p \Big)^{1/p}
\sim \Big(\E \Big | \sum_{i=1}^N \varepsilon_i x_i \Big |_X^q \Big)^{1/q}.
\end{equation}
Definitions related to Banach spaces often involve such random sums
and the definition may involve some fixed choice of the exponent -- but the choice is irrelevant by the Kahane-Khintchine inequality.

The Kahane contraction principle says that if $(a_i)_{i=1}^N$ is a sequence of scalars, $x_1,\dots, x_N \in X$ and $p \in (0, \infty]$, then
\begin{equation*}\label{eq:KCont}
\Big( \E \Big | \sum_{i=1}^N \varepsilon_i a_i x_i \Big |_{X}^p \Big)^{1/p}
\lesssim \max |a_i| \Big( \E \Big | \sum_{i=1}^N \varepsilon_i x_i \Big |_{X}^p \Big)^{1/p}.
\end{equation*}

\begin{defn}
Let $X$ be a Banach space, let $r \in [1,2]$ and $q \in [2,\infty]$.
\begin{enumerate}
\item The space $X$ has type $r$ if there exists a finite constant $\tau \ge 0$ such that for all finite sequences $x_1,\ldots,x_N \in X$ we have
$$
\Big( \E \Big| \sum_{i=1}^N \varepsilon_i x_i \Big|_X^r \Big)^{1/r} \le \tau \Big( \sum_{i=1}^N |x_i|_X^r \Big)^{1/r}.
$$
\item  The space $X$ has cotype $q$ if there exists a finite constant $c \ge 0$ such
that for all finite sequences $x_1,\ldots,x_N \in X$ we have
$$
\Big( \sum_{i=1}^N |x_i|_X^q \Big)^{1/q} \le 
c \Big( \E \Big| \sum_{i=1}^N \varepsilon_i x_i \Big|_X^q \Big)^{1/q}.
$$
For $q = \infty$ the usual modification is used. 
\end{enumerate}
The least admissible constants are denoted by $\tau_{r,X}$ and $c_{q, X}$ -- they are the type $r$ constant and cotype $q$ constant of $X$.
\end{defn}
In \cite[Section 7]{HNVW2} the reader can find the basic theory of types and cotypes. We only need a few basic facts, however.

If $X$ has type $r$ (cotype $q$), then it also has type $u$ for all $u \in [1,r]$ (cotype $v$ for all $v \in [q, \infty]$), and we have
$
\tau_{u, X} \le \tau_{p, X} \qquad (c_{v,X} \le c_{q,X}).
$
It is also trivial that always $\tau_{1, X} = c_{\infty, X} = 1$.
We say that $X$ has non-trivial type if $X$ has type $r$ for some $r \in (1, 2]$ and finite cotype if it has cotype $q$ for some $q \in [2, \infty)$.

For the types and cotypes of $L^p$ spaces we have the following:  if $X$ has type $r$, then $L^p(X)$ has type $\min(r,p)$, and if $X$ has cotype $q$, then $L^p(X)$ has cotype $\max(q,p)$.

The $\UMD$ property is a necessary and sufficient condition for the boundedness of various singular integral operators on $L^p(\R^d;X) = L^p(X)$, see \cite[Sec. 5.2.c and the Notes to Sec. 5.2]{HNVW1}. 
\begin{defn}
A Banach space $X$ is said to be a $\UMD$ \emph{space}, where $\UMD$ stands for unconditional martingale differences, if for all $p \in (1,\infty)$, all
$X$-valued $L^p$-martingale difference sequences $(d_i)_{i=1}^N$ and all choices of fixed signs $\epsilon_i \in \{-1,1\}$
we have
\begin{equation}\label{eq:UMDDef}
\Big\| \sum_{i=1}^N \epsilon_i d_i \Big \|_{L^p(X)}
\lesssim \Big\| \sum_{i=1}^N d_i \Big \|_{L^p(X)}.
\end{equation}
The $L^p(X)$-norm is with respect to the measure space where the martingale differences are defined.
\end{defn}
A standard property of $\UMD$ spaces is that if \eqref{eq:UMDDef} holds for one $p_0\in (1,\infty)$ it
holds for all $p \in (1, \infty)$ \cite[Theorem 4.2.7]{HNVW1}. Moreover, if $X$ is $\UMD$ then so is the dual space $X^*$ \cite[Prop. 4.2.17]{HNVW1}.
Importantly, $\UMD$ spaces have non-trivial type and a finite cotype.

Stein's inequality says that for a $\UMD$ space $X$ we have
$$
 \E \Big\| \sum_{I \in \calD} \varepsilon_I \langle f_I \rangle_I 1_I \Big\|_{L^p(X)} \lesssim  \E \Big\| \sum_{I \in \calD} \varepsilon_I f_I \Big\|_{L^p(X)}, \qquad p \in (1,\infty).
$$
This $\UMD$-valued version of Stein's inequality is by Bourgain, for a proof see e.g. Theorem 4.2.23 in the book \cite{HNVW1}.

We now introduce some definitions related to the so called decoupling estimate.
For $K \in \calD$ 
denote by $Y_{K}$ the measure space $(K, \text{Leb}(K), \nu_{K})$. Here $\text{Leb}(K)$ is the collection of 
Lebesgue measurable subsets of $K$ and $ \nu_K=\ud x \lfloor K/|K|$, where $\ud x \lfloor K$ is the $d$-dimensional Lebesgue measure restricted to $K$.
We then define the product probability space
$$
(Y, \scrA, \nu):= \prod_{K \in \calD} Y_{K}.
$$
If $y \in Y$ and $K \in \calD$, we denote by $y_K$ the coordinate related to $Y_K$. 

In our upcoming estimates, it will be important to separate scales using the following subgrids. 
For $k \in \{0,1,\dots\}$ and $l \in \{0, \dots, k\}$ define
\begin{equation}\label{eq:SubLattice}
\calD_{k,l}:=\{ K \in \calD \colon \ell(K)=2^{m(k+1)+l} \text{ for some } m \in \Z\}.
\end{equation}
The following proposition concerning decoupling is a special case of Theorem 3.1 in \cite{HH}. It is a result that can
be stated in the generality of suitable filtrations, but we prefer to only state the following dyadic version.
\begin{prop}\label{prop:decoupling}
Let $X$ be a $\UMD$ space, $p \in (1, \infty)$, $k \in \{0,1,\dots\}$ and $l \in \{0, \dots, k\}$. Suppose $f_K$, $K \in \calD_{k,l}$, are functions such that
\begin{enumerate}
\item $f_K = 1_K f_K$,
\item $\int f_K = 0$ and
\item $f_K$ is constant on those $K' \in \calD_{k,l}$ for which $K' \subsetneq K$.
\end{enumerate}
Then we have
\begin{equation}\label{eq:decoupling}
\int_{\R^d} \Big| \sum_{K \in \calD_{k,l}} f_K(x) \Big|_X^p \ud x
\sim \E \int_{Y} \int_{\R^d} 
\Big| \sum_{K \in \calD_{k,l}} \varepsilon_K 1_K(x) f_K(y_K) \Big|_X^p \ud x \ud \nu (y),
\end{equation}
where the implicit constant is independent of $k,l$.
\end{prop}

Hilbert spaces are the only Banach spaces with both type $2$ and cotype $2$. Below we will prove estimates for modified shifts like
$$
\| Q_k f \|_{L^2(X)} \lesssim (k+1)^{1/\min(r, q')} \| f \|_{L^2(X)},
$$
where the $\UMD$ space $X$ has type $r$ and cotype $q$. Therefore, in the Hilbert space case -- and thus in the scalar-valued case -- these estimates recover
the best possible regularity $\alpha=1/2$. The presented estimates are efficient in completely general $\UMD$ spaces -- however, in $\UMD$ function lattices
it is more efficient to mimic the scalar-valued theory (Proposition \ref{prop:Qscalar})  and use square function estimates instead.

In this paper we are mainly interested on these deeper general $\UMD$-valued estimates. For the same reason
we will not later pursue $\UMD$-valued theory in the bi-parameter setting, as it is known that then the additional ``property $(\alpha)$'' is required, see \cite[Sec. 8.3.e]{HNVW2}.
The function lattice assumption is formally stronger, but for main concrete examples practically the same as the ``property $(\alpha)$''  assumption.

\subsection*{The linear case}
We feel that it is too difficult to jump directly into the multilinear estimates, as they are quite involved. Thus, we first study the linear case.
We show that the framework of modified dyadic shifts gives a modern and convenient proof of the results of Figiel
\cite{Figiel1, Figiel2} concerning $\UMD$-extensions of CZOs with mild kernel regularity. 

Before moving to the main $X$-valued estimate for $Q_k$, we state the following result for paraproducts. We have
\begin{equation}\label{eq:XLinPar}
\| \pi f\|_{L^p(X)} \lesssim \|f\|_{L^p(X)}
\end{equation}
whenever $p \in (1,\infty)$ and $X$ is $\UMD$. We understand that this is usually attributed to Bourgain -- in any case, a simple proof can now be found in \cite{HH}.

\begin{rem}
The estimate in Proposition \ref{prop:ModShiftUMD} below is best used for $p=2$, since then e.g. $\min(r,p) = r$, if $r \in (1, 2]$ is an exponent such that $X$ has type $r$. Indeed, it is efficient to only move the $p=2$ estimate
for the CZO $T$, and then interpolate to get the $L^p$ boundedness under a modified Dini type assuption that is independent of $p$.
 On the $Q_k$ level improving an $L^2$ estimate into an $L^p$ estimate with good dependency on the complexity does not seem so simple. Interpolation would introduce some additional complexity dependency, since
the weak $(1,1)$ inequality of $Q_k$ is not complexity free.
\end{rem}

\begin{prop}\label{prop:ModShiftUMD}
Let $p \in (1,\infty)$ and $X$ be a $\UMD$ space. If $Q_k$ is a modified shift of the form \eqref{eq:Q_kFORM2}, 
then
$$
\| Q_k f \|_{L^p(X)} \lesssim (k+1)^{1/\min(r, p)} \| f \|_{L^p(X)},
$$
where $r \in (1, 2]$ is an exponent such that $X$ has type $r$. If $Q_k$ is a modified shift of the form \eqref{eq:Q_kFORM1}, we have
$$
\| Q_k f \|_{L^p(X)} \lesssim (k+1)^{1/\min(q', p')} \| f \|_{L^p(X)},
$$
where $q \in [2, \infty)$ is an exponent such that $X$ has cotype $q$.
\end{prop}

\begin{proof}
We assume that $Q_k$ has the form \eqref{eq:Q_kFORM2} -- the other result follows by duality. This uses that if the $\UMD$ space $X$
has cotype $q$, then the dual space $X^*$ has type $q'$ -- see \cite[Proposition 7.4.10]{HNVW2}.

If $K \in \calD$ we define
$$
B_Kf:= \sum_{I^{(k)}=J^{(k)}=K} a_{IJK} \langle f, H_{I,J} \rangle h_J.
$$
Recall the lattices $\calD_{k,l}$ from \eqref{eq:SubLattice} and write $Q_k f = \sum_{l=0}^k \sum_{K \in \calD_{k,l}} B_K f$. By using the $\UMD$ property of $X$ and the Kahane--Khintchine inequality we have
for all $s \in (0,\infty)$ that
$$
\| Q_k f \|_{L^p(X)} \sim
 \Big( \E \Big \| \sum_{l=0}^k \varepsilon_l \sum_{K \in \calD_{k,l}}
B_K f \Big \|_{L^p(X)}^s \Big)^{1/s}.
$$
We use this with the choice $s := \min(r,p)$, since $L^p(X)$ has type $s$. Using this we have
$$
 \Big( \E \Big \| \sum_{l=0}^k \varepsilon_l \sum_{K \in \calD_{k,l}} B_K f \Big \|_{L^p(X)}^s \Big)^{1/s}
  \lesssim \Big (\sum_{l=0}^k \Big\| \sum_{K \in \calD_{k,l}} B_K f \Big \|_{L^p(X)}^s \Big)^{1/s}.
$$
To end the proof, it remains to show that
\begin{equation}\label{eq:PartOfQk}
\Big \| \sum_{K \in \calD_{k,l}} B_K f \Big \|_{L^p(X)}
\lesssim \|f\|_{L^p(X)}
\end{equation}
uniformly on $l$.

Recall that $\langle f, H_{I,J} \rangle = \langle P_{K,k} f, H_{I,J} \rangle$.
We then see that
\begin{equation*}
\begin{split}
B_K f = 
\sum_{I^{(k)}=J^{(k)}=K} 
a_{IJK} \langle P_{K,k}f, H_{I,J} \rangle  h_J
&=\sum_{I^{(k)}=J^{(k)}=K} 
a_{IJK} \langle P_{K,k}f, 1_J H_{I,J} \rangle  h_J \\
&+\sum_{\substack{ I^{(k)}=J^{(k)}=K \\ I \not= J}} 
a_{IJK} \langle P_{K,k}f, 1_I H_{I,J} \rangle  h_J.
\end{split}
\end{equation*}
Accordingly, this splits the estimate of \eqref{eq:PartOfQk} into two parts.

We consider first the part related to $\langle P_{K,k}f, 1_I H_{I,J} \rangle$.
By the $\UMD$ property and the Kahane--Khintchine inequality we have
\begin{equation}\label{eq:LinearStandStart}
\begin{split}
\Big \| \sum_{K \in \calD_{k,l}}& \sum_{\substack{I^{(k)}=J^{(k)}=K \\ I \not= J}} 
a_{IJK} \langle P_{K,k}f, 1_I H_{I,J} \rangle  h_J
 \Big \|_{L^p(X)}  \\
 &\sim \Big( \E \Big \| \sum_{K \in \calD_{k,l}} \epsilon_K \sum_{\substack{I^{(k)}=J^{(k)}=K \\ I \not= J}} 
a_{IJK} \langle P_{K,k}f, 1_I H_{I,J} \rangle  h_J
 \Big \|_{L^p(X)}^p \Big)^{1/p}.
\end{split}
\end{equation}
Notice then that for
$$
a_K(x,y):= |K| \sum_{\substack{I^{(k)}=J^{(k)}=K \\ I \not= J}}  a_{IJK} 1_I(y)H_{I,J}(y)h_J(x)
$$
we have
\begin{equation*}
\begin{split}
\sum_{\substack{I^{(k)}=J^{(k)}=K \\ I \not= J}} 
a_{IJK} \langle P_{K,k}f, 1_I H_{I,J} \rangle  h_J(x)
&=\frac{1}{|K|} \int_K a_K(x,y) P_{K,k}f(y) \ud y \\
&=\int_{Y_K} a_K(x,y_K) P_{K,k}f(y_K) \ud \nu_K (y_K) \\
&=\int_{Y} a_K(x,y_K) P_{K,k}f(y_K) \ud \nu (y).
\end{split}
\end{equation*}
Using this we have by H\"older's inequality (recalling that $\nu$ is a probability measure) that
\begin{equation*}
\begin{split}
\E \Big \|& \sum_{K \in \calD_{k,l}} \epsilon_K \sum_{\substack{I^{(k)}=J^{(k)}=K \\ I \not= J}} 
a_{IJK} \langle P_{K,k}f, 1_I H_{I,J} \rangle  h_J
 \Big \|_{L^p(X)}^p\\
& \le \E \int_{\R^d} \int_Y \Big|\sum_{K \in \calD_{k,l}} \varepsilon_K 
a_K(x,y_K) P_{K,k}f(y_K) \Big |_{X}^{p}  \ud \nu(y) \ud x. 
\end{split}
\end{equation*}

Notice now that $|a_K(x,y)| \le 1_K(x)$. Thus, the Kahane contraction principle implies that for fixed $x$ and $y$ there holds that
$$
\E \Big|\sum_{K \in \calD_{k,l}} \varepsilon_K 
a_K(x,y_K) P_{K,k}f(y_K) \Big |_{X}^{p}
\lesssim \E \Big|\sum_{K \in \calD_{k,l}} \varepsilon_K 
1_K(x) P_{K,k}f(y_K) \Big |_{X}^{p}.
$$
Using this we are left with
\begin{equation*}
\begin{split}
\E \int_{\R^d} \int_Y \Big|\sum_{K \in \calD_{k,l}} \varepsilon_K 
1_K(x) P_{K,k}f(y_K) \Big |_{X}^{p}  \ud \nu(y) \ud x 
& \sim  \int_{\R^d} \Big|\sum_{K \in \calD_{k,l}} 
P_{K,k}f(x) \Big |_{X}^{p}  \ud x \\
&= \| f \|_{L^{p}(X)}^{p},
\end{split}
\end{equation*}
where we used the decoupling estimate \eqref{eq:decoupling} and noticed that $\sum_{K \in \calD_{k,l}} 
P_{K,k}f = f$.

We turn to the part  related to the terms $\langle P_{K,k}f, 1_J H_{I,J} \rangle$. We begin with
\begin{equation}\label{eq:ReadyForStein}
\begin{split}
 \Big\| & \sum_{K \in \calD_{k,l}} \sum_{I^{(k)}=J^{(k)}=K} 
a_{IJK} \langle P_{K,k}f, 1_J H_{I,J} \rangle  h_J \Big \|_{L^{p}(X)}^{p} \\
&\sim \E \Big\|  \sum_{K \in \calD_{k,l}}  \sum_{I^{(k)}=J^{(k)}=K} 
\varepsilon_J a_{IJK} \langle P_{K,k}f, 1_J H_{I,J} \rangle  \frac{1_J}{|J|^{1/2}} \Big \|_{L^{p}(X)}^{p},
\end{split}
\end{equation}
where we used the $\UMD$ property to introduce the random signs and then the fact that we can clearly replace $h_J$ by $|h_J| = 1_J / |J|^{1/2}$ due to
the random signs (the relevant random variables are identically distributed).
We write the inner sum as
$$
\sum_{I^{(k)}=J^{(k)}=K} 
\varepsilon_J a_{IJK} \langle P_{K,k}f, 1_J H_{I,J} \rangle  \frac{1_J}{|J|^{1/2}}
=\sum_{J^{(k)}=K} \varepsilon_J
 \Big\langle P_{K,k}f \sum_{I^{(k)}=K} a_{IJK}|J|^{1/2} H_{I,J} \Big\rangle_J  1_J.
$$
Notice also that
$$
\Big|\sum_{I^{(k)}=K} a_{IJK}|J|^{1/2} H_{I,J}\Big|
\le \frac{1}{|K|}\sum_{I^{(k)}=K} |I| =1.
$$
We continue from \eqref{eq:ReadyForStein}. Applying Stein's inequality  we get that
the last term in \eqref{eq:ReadyForStein} is dominated by
\begin{equation*}
\begin{split}
\E \Big\|   \sum_{K \in \calD_{k,l}}  \sum_{J^{(k)}=K} & \varepsilon_J
 P_{K,k}f \sum_{I^{(k)}=K} a_{IJK}|J|^{1/2} H_{I,J}   1_J \Big \|_{L^{p}(X)}^{p} \\
 & \lesssim \E \Big\|   \sum_{K \in \calD_{k,l}}  \sum_{J^{(k)}=K} \varepsilon_J
 P_{K,k}f   1_J \Big \|_{L^{p}(X)}^{p},
 \end{split}
 \end{equation*}
 where we used the Kahane contraction principle. From here the estimate is easily concluded by
 \begin{equation*}
 \begin{split}
 \E \Big\|   \sum_{K \in \calD_{k,l}}  \sum_{J^{(k)}=K} \varepsilon_J
 P_{K,k}f   1_J \Big \|_{L^{p}(X)}^{p}
 &= \E \Big\|   \sum_{K \in \calD_{k,l}}  \varepsilon_K \sum_{J^{(k)}=K} 
 P_{K,k}f   1_J \Big \|_{L^{p}(X)}^{p} \\
 &=\E \Big\|   \sum_{K \in \calD_{k,l}} \varepsilon_K 
 P_{K,k}f   \Big \|_{L^{p}(X)}^{p}
 \sim \| f \|_{L^{p}(X)}^{p}.
\end{split}
\end{equation*}
Here we first changed the indexing of the random signs (using that for a fixed $x$ for every $K$ there is at most one $J$ as in the sum for which $1_J(x) \ne 0$)
and then applied the $\UMD$ property. This finishes the proof.
\end{proof}

\begin{thm}
Let $T$ be a linear $\omega$-CZO and $X$ be a $\UMD$ space with type $r \in (1,2]$ and cotype $q \in [2, \infty)$.
If $\omega \in \operatorname{Dini}_{1 / \min(r, q')}$, we have
$$
\|Tf \|_{L^p(X)} \lesssim \|f\|_{L^p(X)}
$$
for all $p \in (1,\infty)$.
\end{thm}
\begin{proof}
Apply Theorem \ref{thm:rep1par} to simple vector-valued functions. By Proposition \ref{prop:ModShiftUMD} and the $X$-valued boundedness of the paraproducts \eqref{eq:XLinPar}
we conclude that
$$
\|Tf \|_{L^2(X)} \lesssim \|f\|_{L^2(X)}.
$$
As the weak type $(1,1)$ follows from this even with just the assumption $\omega \in \operatorname{Dini}_{0}$, we can conclude the proof by the standard interpolation and duality method.
\end{proof}

\subsection*{The multilinear case}
Let $(X_1, \ldots, X_{n+1})$ be $\UMD$ spaces and $Y_{n+1}^* = X_{n+1}$.
Assume that there is an $n$-linear mapping $X_1 \times \cdots \times X_n \to Y_{n+1}$, which we denote with the product notation
$(x_1, \ldots, x_n) \mapsto \prod_{j=1}^n x_j$, so that
$$
\Big| \prod_{j=1}^n x_j \Big|_{Y_{n+1}} \le \prod_{j=1}^n |x_j|_{X_j}.
$$
With just this setup it makes sense to extend an $n$-linear SIO (or some other suitable $n$-linear operator) $T$ using the formula
\begin{equation}
\label{e:formal1}
\begin{split}
&T(f_1, \ldots, f_n)(x) = \sum_{j_1, \ldots, j_n}  T(f_{1,j_1}, \ldots,f_{n,j_n}) (x) \prod_{k=1}^n e_{k,j_k}, \qquad x\in \R^d, 
\\ &f_{k}= \sum_{j_k}  e_{k,j_k} f_{k,j_k}, \qquad f_{k,j_k}\in L^\infty_c,\,e_{k,j_k}\in X_k.
\end{split}
\end{equation}
In the bilinear case $n=2$ the existence of such a product is the only assumption that we will need. The bilinear case is somewhat harder than the linear case,
but the $n \ge 3$ case is by far the most subtle. Indeed, for $n \ge 3$ we will
need a more complicated setting for the tuple of spaces $(X_1, \ldots, X_{n+1})$  -- the idea is to model the H\"older type structure typical of concrete examples of Banach $n$-tuples,  such as that of non-commutative $L^p$ spaces with the exponents $p$ satisfying the natural H\"older relation.
We will borrow this setting from \cite{DLMV1}. In \cite{DLMV1} it is shown in detail how natural tuples of non-commutative $L^p$ spaces fit to this abstract framework.
While we borrow the setting, the proof is significantly different. First, we have to deal with the more complicated modified shifts. Second, even in the standard shift
case the proof in \cite{DLMV1} is -- by its very design -- extremely costly on its complexity dependency. To circumvent this we need a new strategy.

For $m \in \N$ we write $\calJ_m:= \{1, \dots, m\}$ and denote the set of permutations of $\mathcal J\subset \mathcal J_m$ by $\Sigma(\mathcal J)$.   We write $\Sigma(m) = \Sigma(\calJ_m)$. 

Next, we fix an associative algebra $\mathcal A$ over $\mathbb C$, and denote the associative operation
$\mathcal A \times \mathcal A \to \mathcal A$ by $ (e,f)\mapsto ef$.
We  assume that there exists a subspace $\mathcal L^1 $ of $\mathcal A$ and a linear functional $\tau:\mathcal L^1\to \C$, which we refer to as \emph{trace}.
Given  an $m$-tuple $(X_1,\ldots, X_m)$ of Banach subspaces $X_j \subset \mathcal A$, we construct the seminorm  
\begin{equation}
\label{e:YJnorm2}
|e|_{  Y(X_1,\ldots, X_m)} = \sup\Big\{ \Big|\tau\Big( e \prod_{\ell=1}^{m} e_{\sigma(\ell)}\Big)\Big|: \sigma \in\Sigma(m), |e_j|_{  X_j}=1, j =1,\ldots,m \Big\}
\end{equation}
on the subspace 
\begin{equation}
\label{e:subspace}
Y(X_1,\ldots, X_m)=
\Big\{e\in \mathcal A: e\prod_{\ell=1}^{m} e_{\sigma(\ell)} \in \mathcal L^1 \, \forall \sigma \in \Sigma(m), \, e_j \in X_j,\, j=1,\ldots, m  \Big\}.
\end{equation}

For a Banach subspace $X \subset \mathcal A$ and $y \in Y(X)$ we can define the mapping $\Lambda_y \in X^*$ by the formula $\Lambda_y(x) := \tau(yx)$, since by the definition $yx \in \mathcal L^1$ and
$|\tau(yx)| \le |y|_{Y(X)} |x|_X$.
We say that a Banach subspace $X$ of $\mathcal A $ is admissible if the following holds.
\begin{enumerate}
\item $Y(X)$ is a Banach space with respect to $| \cdot|_{Y(X)}$.
\item The mapping $y \mapsto \Lambda_y$ from $Y(X)$ to $X^*$ is surjective.
\item For each $x \in X, \  y\in Y(X)$, we also have $xy\in \mathcal L^1$ and 
\begin{equation}\label{e:commtrace}
\tau(xy)=\tau(yx).
\end{equation}
\end{enumerate}
If $X$ is admissible, then the map $y \mapsto \Lambda_y$ is an isometric bijection from $Y(X)$ onto $X^*$, and we identify   
$Y(X)$ with $X^*$.
The following is \cite[Lemma 3.10]{DLMV1}.
\begin{lem} \label{lem:YJ}
Let $X$ be admissible and reflexive (for instance, $X$ is admissible and $\UMD$). If $Y(X)$ is also admissible, then $Y(Y(X))=X$ as sets and $|x|_{Y(Y(X))}= |x |_X$ for all $x \in X$.
\end{lem}
If $X, X_1, \dots, X_m$ are Banach spaces we write $X=Y(X_1, \dots, X_m)$ to mean that $X$ and $Y(X_1, \dots, X_m)$ coincide as sets,
$Y(X_1, \dots, X_m)$ is a Banach space with the norm $| \cdot |_{Y(X_1, \dots, X_m)}$, and that the norms are equivalent, that is,
$|x|_X \sim |x|_{Y(X_1, \dots, X_m)}$ for all $x \in X$.

\begin{defn}[$\UMD$ H\"older pair] \label{defn:productsys0} 
Let $X_1$,  $X_2$ be admissible spaces. We say that $\{X_1,X_2\}$ is a \emph{$\UMD$ H\"older pair}  
if $X_1$ is a $\UMD$ space and $X_2={Y(}X_1)$. 
\end{defn}

\begin{defn}[{$\UMD$ H\"older $m$-tuple, $m\geq 3$}] \label{defn:productsys1} 
Let $X_1,\ldots, X_m$ be admissible spaces. We say that $\{X_1,\ldots, X_{m}\}$ is a \emph{$\UMD$ H\"older $m$-tuple} if
  the following  properties hold.
\vskip2mm \noindent  \textbf{P1.} For all $j_0\in \mathcal J_m$ there holds
\[
X_{j_0}=Y\left(\left\{X_{j}: j\in \mathcal J_{m}\setminus\{j_0\}\right\}\right).
\]
\vskip2mm \noindent \textbf{P2.}
If $1\leq k \leq m-2$ and $\mathcal J=\{j_1<j_2<\cdots <j_k\}\subset \mathcal J_m$, then $Y(X_{j_1}, \dots, X_{j_k})$ is an admissible 
Banach space with the norm \eqref{e:YJnorm2}  and
\begin{equation}\label{eq:SubHolder}
\{X_{j_1}, \dots, X_{j_k}, Y(X_{j_1}, \dots, X_{j_k})\}
\end{equation}
is a $\UMD$ H\"older $(k+1)$-tuple.
\end{defn}

The following is a key consequence of the definition. 
Let  $m \ge 3$ and $\{X_1,\ldots, X_{m}\}$   be a $\UMD$ H\"older $m$-tuple.  
Then according to P2 the pair $\{X_{j_0},{Y(}X_{j_0})\}$ is a  $\UMD$ H\"older pair, which by Definition \ref{defn:productsys0}
implies that $X_{j_0}$ and $Y(X_{j_0})$ are $\UMD$ spaces. The inductive nature of the definition then ensures that  
each $Y(X_{j_1}, \dots, X_{j_k})$ appearing in \eqref{eq:SubHolder}  is a $\UMD$ space.

Notice also the following.
Let $m \ge 2$ and   $\{X_1, \dots, X_m\}$ be  a $\UMD$ $m$-H\"older tuple.  
Let $e_j \in X_j$ for $j \in \calJ_{m}$. For each $\si \in \Sigma(m)$,
as $X_{\si(1)}=Y(X_{\si(2)}, \dots , X_{\si(m)})$, we necessarily have
$\prod_{j=1}^{m} e_{\si(j)} \in \calL^1$ and 
$$
|\tau(e_{\si(1)} \cdots e_{\si(m)}) |
\le |e_{\si(1)}|_{Y(X_{\si(2)}, \dotsm, X_{\si(m)})}\prod_{j=2}^{m} | e_{\si(j)} |_{X_{\si(j)}}
\sim \prod_{j=1}^{m} | e_{j} |_{X_{j}}.
$$
Moreover, by \eqref{e:commtrace} we have
\begin{equation}\label{eq:sonice}
\tau( e_1 \cdots e_m) = \tau (e_m e_1 \cdots e_{m-1}) =  \tau (e_{m-1} e_m e_1 \cdots e_{m-2}) = \cdots .
\end{equation}
We have the following H\"older type inequality.
\begin{lem}\label{lem:nice}
Let $\{X_1,\ldots, X_{m}\}$ be a $\UMD$ H\"older tuple. Then we have
\begin{align*}
| e v|_{Y(X_{m})}
\lesssim
|e|_{Y(Y(X_1, \ldots, X_k))}  |v|_{Y(Y(X_{k+1}, \ldots, X_{m-1}))}. 
\end{align*}
\end{lem}
\begin{proof}
To estimate $|ev|_{Y(X_{m})}$ we need to estimate
$$
|\tau(  eve_m)|
$$
with an arbitrary $e_m$ with $|e_m|_{X_m} = 1$. First, we bound this with
$$
|e|_{Y(Y(X_1, \ldots, X_k))} | ve_m |_{Y(X_1, \ldots, X_k)}.
$$
To estimate $ |ve_m |_{Y(X_1, \ldots, X_k)}$ we need to estimate
$$
|\tau( ve_m u_{\sigma_1(1)} \cdots u_{\sigma_1(k)})|
$$
with arbitrary $|u_j|_{X_j} = 1$ and $\sigma_1 \in \Sigma(k)$. We estimate this with
$$
|v|_{Y(Y(X_{k+1}, \ldots, X_{m-1}))} | e_m u_{\sigma_1(1)} \cdots u_{\sigma_1(k)})|_{Y(X_{k+1}, \ldots, X_{m-1})}.
$$
To estimate $| e_m u_{\sigma_1(1)} \cdots u_{\sigma_1(k)})|_{Y(X_{k+1}, \ldots, X_{m-1})}$ we need to estimate
$$
|\tau( e_m u_{\sigma_1(1)} \cdots u_{\sigma_1(k)} u_{\sigma_2(k+1)} \cdots u_{\sigma_2(m-1)} )|
$$
with an arbitrary permutation $\sigma_2$ of $\{k+1, \ldots, m-1\}$ and $|u_j|_{X_j} = 1$. Finally, we estimate this with
$$
|e_m|_{Y(X_1, \ldots, X_{m-1})} \sim |e_m|_{X_m} = 1.
$$
\end{proof}

In all of the statements below an arbitrary $\UMD$ H\"older tuple $\{X_1,\ldots, X_{n+1}\}$ is given.

\begin{prop}\label{prop:MultliModUMD}
Suppose that $Q_k$ is an $n$-linear modified shift and $f_j \colon \R^d \to X_j$. Let $1 < p_j < \infty$ 
with $\sum_{j=1}^{n+1} 1/p_j = 1$.
Then we have
$$
|\langle Q_k(f_1, \ldots, f_n), f_{n+1}\rangle| \lesssim (k+1)^{\alpha} \prod_{j=1}^{n+1} \| f_j \|_{L^{p_j}(X_j)},
$$
where
$$
\alpha = \frac{1}{\min( p_1', \ldots, p_{n+1}', s_{1}', \ldots, s_{n+1}')}
$$
and $X_j$ has cotype $s_j$.
\end{prop}

\begin{proof}
We will assume that $Q_k$ is of the form \eqref{e:modnshift} -- the other cases follow by duality using the property \eqref{eq:sonice}.
We follow the ideas of the decomposition from the proof of Lemma \ref{lem:CompModStand}: we will estimate the terms $A_m$
and $U_m$, $m \in \{1, \ldots, n\}$, from there separately.

First,
we estimate the part  $A_m$ defined in \eqref{eq:DefAm}. We have that
$$
A_m(f_1, \dots, f_n)
= \sum_K A_{m,K}(f_1, \dots, f_n),
$$ 
where
\begin{equation*}
A_{m,K}(f_1, \dots, f_n):=\sum_{I_m^{(k)} =  \cdots = I_{n+1}^{(k)} =K} b_{m,K, (I_j)} \prod_{j=1}^{m-1} \langle f_j \rangle_K \cdot \langle P_{K, k-1} f_m \rangle_{I_m}
\cdot \prod_{j=m+1}^n \langle f_j \rangle_{I_j}  \cdot h_{I_{n+1}},
\end{equation*}
$$
b_{m,K, (I_j)}= b_{m,K, I_m, \dots, I_{n+1}} := \sum_{I_1^{(k)} =  \cdots = I_{m-1}^{(k)} =K}b_{K,(I_j)}
$$
and
$$
b_{K, (I_j)} = b_{K, I_1, \ldots, I_{n+1}} = |I_{n+1}|^{n/2} a_{K, (I_j)}.
$$
Here we have the normalization
\begin{equation}\label{eq:CoefbmNorm}
|b_{m,K, (I_j)}| \le \frac{|I_{n+1}|^{n+1/2}}{|K|^n} \sum_{I_1^{(k)} =  \cdots = I_{m-1}^{(k)} =K} 1 = \frac{|I_{n+1}|^{n-m+3/2}}{|K|^{n-m+1}}.
\end{equation}

Recall the grids $\calD_{k,l}$, $l \in \{0,1, \dots, k\}$, from \eqref{eq:SubLattice}. 
Let $$Z_{1, \ldots, n} := Y(X_{n+1}) = Y(Y(X_1, \ldots ,X_n)).$$ Recalling that $X_{n+1}^*$ is identified with $Y(X_{n+1})$,
we have that $L^{p_{n+1}'}(Z_{1, \ldots, n})$ has type $s:=\min(p_{n+1}', s_{n+1}')$.
Thus, there holds that
\begin{equation}\label{eq:UseOfType}
\begin{split}
 \| A_m(f_1, \dots, f_n) \|_{L^{p_{n+1}'}(Z_{1, \ldots, n})}
&\sim \Big( \E \Big \| \sum_{l=0}^{k} \varepsilon_l \sum_{K \in \calD_{k,l}} A_{m,K}(f_1, \dots, f_n) \Big \|_{L^{p_{n+1}'}(Z_{1, \ldots, n})}^s \Big)^{1/s} \\
& \lesssim \Big( \sum_{l=0}^{k} 
\Big \|  \sum_{K \in \calD_{k,l}} A_{m,K}(f_1, \dots, f_n) \Big \|_{L^{p_{n+1}'}(Z_{1, \ldots, n})}^s \Big)^{1/s}.
\end{split}
\end{equation}
In the first step we used the $\UMD$ property of $Z_{1, \ldots, n}$ and the Kahane--Khintchine inequality.
We see that to prove the claim it suffices to show the uniform bound
\begin{equation}\label{eq:EstSubAm}
\Big \|  \sum_{K \in \calD_{k,l}} A_{m,K}(f_1, \dots, f_n) \Big \|_{L^{p_{n+1}'}(Z_{1, \ldots, n})}
\lesssim \prod_{j=1}^n \| f_j \|_{L^{p_j}(X_j)}.
\end{equation}

We turn to prove \eqref{eq:EstSubAm}. To avoid confusion with the various $Y$ spaces, we denote
the decoupling space by $(W, \nu)$ in this proof.
The decoupling estimate \eqref{eq:decoupling} gives that the left hand side of \eqref{eq:EstSubAm}
is comparable to
\begin{equation*}
\begin{split}
\Big( &\E \int_W \int_{\R^d} 
\Big | \sum_{K \in \calD_{k,l}} \varepsilon_K 1_K(x) \sum_{I_m^{(k)} =  \cdots = I_{n+1}^{(k)} =K}  \\
&  b_{m,K, (I_j)} \prod_{j=1}^{m-1} \langle f_j \rangle_K \cdot \langle P_{K, k-1} f_m \rangle_{I_m}
\cdot \prod_{j=m+1}^n \langle f_j \rangle_{I_j}  \cdot h_{I_{n+1}}(w_K) \Big|_{Z_{1, \ldots, n}}^{p_{n+1}'}
\ud x \ud \nu(w) \Big)^{1/p_{n+1}'}.
\end{split}
\end{equation*}
Suppose that $m>1$; if $m=1$ one can start directly from \eqref{eq:AfterSteinLeft} below. 
Now, with a fixed $w \in W$ can use Stein's inequality with respect to the function $f_1$ to have that the previous term is dominated by 
\begin{equation}\label{eq:eq5}
\begin{split}
\Big( &\E \int_W \int_{\R^d} 
\Big | f_1(x) \sum_{K \in \calD_{k,l}} \varepsilon_K 1_K(x) \sum_{I_m^{(k)} =  \cdots = I_{n+1}^{(k)} =K} b_{m,K, (I_j)}  \\
& \prod_{j=2}^{m-1} \langle f_j \rangle_K \cdot \langle P_{K, k-1} f_m \rangle_{I_m}
\cdot \prod_{j=m+1}^n \langle f_j \rangle_{I_j}  \cdot h_{I_{n+1}}(w_K) \Big|_{Z_{1, \ldots, n}}^{p_{n+1}'}
\ud x \ud \nu(w) \Big)^{1/p_{n+1}'}.
\end{split}
\end{equation}
To move forward, notice that by Lemma \ref{lem:nice} we have
\begin{equation}\label{eq:key1}
\Big| e_1 \sum_k \prod_{j=2}^n e_{j,k} \Big|_{Z_{1,\ldots,n}} = 
\Big| e_1 \sum_k \prod_{j=2}^n e_{j,k} \Big|_{Y(X_{n+1})} \lesssim |e_1|_{X_1} \Big| \sum_k \prod_{j=2}^n e_{j,k} \Big|_{Z_{2, \ldots, n}},
\end{equation}
where $Z_{2, \ldots, n} = Y(Y(X_2, \ldots, X_n))$.
Having established \eqref{eq:key1} we can now dominate \eqref{eq:eq5} with
\begin{equation*}
\begin{split}
 \| &f_1\|_{L^{p_1}(X_1)} \Big(\E \int_W \int_{\R^d} 
\Big | \sum_{K \in \calD_{k,l}} \varepsilon_K 1_K(x) \sum_{I_m^{(k)} =  \cdots = I_{n+1}^{(k)} =K} \\
& b_{m,K, (I_j)} \prod_{j=2}^{m-1} \langle f_j \rangle_K \cdot \langle P_{K, k-1} f_m \rangle_{I_m}
\cdot \prod_{j=m+1}^n \langle f_j \rangle_{I_j}  \cdot h_{I_{n+1}}(w_K) \Big|_{Z_{2, \dots, n}}^{q_{2, \ldots, n}}
\ud x \ud \nu(w) \Big)^{1/q_{2, \ldots, n}},
\end{split}
\end{equation*}
where $q_{2, \ldots n}$ is defined by $1/q_{2, \ldots, n}= \sum_{j=2}^{n} 1/p_j$. 
We can continue this process. In the next step we argue as above but using the H\"older tuple $(X_2, \ldots, X_n, Y(X_2, \ldots, X_n))$.
Below we write $Z_{k_1, \ldots, k_2} = Y(Y(X_{k_1}, \ldots, X_{k_2}))$ and
$1/q_{k_1, \ldots, k_2} =  \sum_{j=k_1}^{k_2} 1/p_j$.
Iterating this we arrive at $\prod_{j=1}^{m-1}  \| f_j\|_{L^{p_j}(X_j)}$ multiplied by
\begin{equation}\label{eq:AfterSteinLeft}
\begin{split}
\Big(\E & \int_W \int_{\R^d} 
\Big | \sum_{K \in \calD_{k,l}} \varepsilon_K 1_K(x) \sum_{I_m^{(k)} =  \cdots = I_{n+1}^{(k)} =K} \\
& b_{m,K, (I_j)} \langle P_{K, k-1} f_m \rangle_{I_m}
\cdot \prod_{j=m+1}^n \langle f_j \rangle_{I_j}  \cdot h_{I_{n+1}}(w_K) \Big|_{Z_{m, \dots, n}}^{q_{m, \ldots, n}}
\ud x \ud \nu(w) \Big)^{1/q_{m, \ldots, n}}.
\end{split}
\end{equation}

We assume that $m<n$; if $m=n$ one is already at \eqref{eq:AfterSteinRight} below. 
Let $K \in \calD_{k,l}$. We define the kernel
\begin{equation}\label{eq:Defbm}
b_{m,K}(x, y_m, \dots, y_n)
:=\sum_{I_m^{(k)} =  \cdots = I_{n+1}^{(k)} =K} |K|^{n-m+1}b_{m,K, (I_j)}
\prod_{j=m}^n \frac{1_{I_j}(y_j)}{|I_j|}h_{n+1}(x)
\end{equation}
so that
\begin{equation*}
\begin{split}
\sum_{I_m^{(k)} =  \cdots = I_{n+1}^{(k)} =K} 
&b_{m,K, (I_j)} \langle P_{K, k-1} f_m \rangle_{I_m}
\cdot \prod_{j=m+1}^n \langle f_j \rangle_{I_j}  \cdot h_{I_{n+1}}(w_K) \\
&= \frac{1}{|K|^{n-m+1}} 
\int_{K^{n-m+1}}b_{m,K}(w_K, y_m, \dots, y_n)P_{K, k-1} f_m(y_m) 
\prod_{j=m+1}^n  f_j(y_j) \ud y.
\end{split}
\end{equation*}
Using this representation in \eqref{eq:AfterSteinLeft} we can use Stein's inequality with respect to the function $f_n$ to have that
the term in \eqref{eq:AfterSteinLeft} is dominated by 
\begin{equation*}
\begin{split}
\Big(\E & \int_W \int_{\R^d} 
\Big | \Big( \sum_{K \in \calD_{k,l}} \varepsilon_K   \frac{1}{|K|^{n-m}}  \int_{K^{n-m}} \\ 
&b_{m, K}(w_K, y_m, \dots,y_{n-1}, x)P_{K, k-1} f_m(y_m) 
\prod_{j=m+1}^{n-1}  f_j(y_j)  \ud y \Big) f_n(x) \Big|_{Z_{m, \dots, n}}^{q_{m, \ldots, n}}
\ud x \ud \nu(w) \Big)^{1/q_{m, \ldots, n}},
\end{split}
\end{equation*}
which is (again by Lemma \ref{lem:nice}) dominated by $\| f_n \|_{L^{p_n}(X_n)}$ multiplied with
\begin{equation*}
\begin{split}
\Big(\E & \int_W \int_{\R^d} 
\Big | \sum_{K \in \calD_{k,l}} \varepsilon_K   \frac{1}{|K|^{n-m}}  \int_{K^{n-m}} \\ 
&b_{m,K}(w_K, y_m, \dots, y_{n-1}, x)P_{K, k-1} f_m(y_m) 
\prod_{j=m+1}^{n-1}  f_j(y_j)  \ud y \Big|_{Z_{m, \dots, n-1}}^{q_{m, \ldots, n-1}}
\ud x \ud \nu(w) \Big)^{1/q_{m, \ldots, n-1}}.
\end{split}
\end{equation*}

Next, we fix the point $w \in W$ and consider the term
\begin{equation}\label{eq:Fixedw}
\begin{split}
\E  \int_{\R^d} 
 &\Big | \sum_{K \in \calD_{k,l}} \varepsilon_K   \frac{1}{|K|^{n-m}}  \int_{K^{n-m}} \\ 
&b_{m,K}(w_K, y_m, \dots, y_{n-1}, x)P_{K, k-1} f_m(y_m) 
\prod_{j=m+1}^{n-1}  f_j(y_j)  \ud y \Big|_{Z_{m, \dots, n-1}}^{q_{m, \ldots, n-1}}
\ud x.
\end{split}
\end{equation}
To finish the estimate of \eqref{eq:AfterSteinLeft} it is enough to dominate \eqref{eq:Fixedw} by $\prod_{j=m}^{n-1} \| f_j \|_{L^{p_j}(X_j)}^{q_{m, \ldots, n-1}}$ uniformly in $w \in W$.
Recalling \eqref{eq:Defbm}  we can write \eqref{eq:Fixedw} as
\begin{equation*}
\begin{split}
\E  \int_{\R^d} 
 &\Big | \sum_{K \in \calD_{k,l}} \varepsilon_K 
 \sum_{I_m^{(k)} =  \cdots = I_{n+1}^{(k)} =K} \\
&|K| b_{m,K, (I_j)} \langle P_{K, k-1} f_m \rangle_{I_m}
\prod_{j=m+1}^{n-1} \langle f_j \rangle_{I_j} \frac{1_{I_{n}(x)}}{|I_n|}  h_{I_{n+1}}(w_K) \Big|_{Z_{m, \dots, n-1}}^{q_{m, \ldots, n-1}} \ud x,
\end{split}
\end{equation*}
which is further comparable to
\begin{equation}\label{eq:AfterFirstSteinRight}
\begin{split}
 \int_{\R^d} 
 &\Big | \sum_{K \in \calD_{k,l}} 
 \sum_{I_m^{(k)} =  \cdots = I_{n}^{(k)} =K} 
 \Big(\sum_{I_{n+1}^{(k)}=K} \frac{|K|}{|I_n|^{1/2}} b_{m,K, (I_j)} h_{I_{n+1}}(w_K)\Big) \\
& \langle P_{K, k-1} f_m \rangle_{I_m}
\prod_{j=m+1}^{n-1} \langle f_j \rangle_{I_j} h_{I_{n}}(x)   \Big|_{Z_{m, \dots, n-1}}^{q_{m, \ldots, n-1}} \ud x.
\end{split}
\end{equation}
In the last step we were able to replace $1_{I_n} / |I_n|^{1/2}$ with $h_{I_n}$ because of the random signs, after which we removed
the signs using $\UMD$.
Recalling the size of the coefficients $b_{m,K, (I_j)}$ from \eqref{eq:CoefbmNorm} we see that
$$
\Big|\sum_{I_{n+1}^{(k)}=K} \frac{|K|}{|I_n|^{1/2}} b_{m,K, (I_j)} h_{I_{n+1}}(w_K)\Big| 
\le \frac{|I_{n}|^{n-m+1/2}}{|K|^{n-m}},
$$
since there is only one $I_{n+1}$ such that $h_{I_{n+1}}(w_K) \not=0$.
It is seen that after applying decoupling \eqref{eq:AfterFirstSteinRight} is like \eqref{eq:AfterSteinLeft} but the degree of linearity is one less.
Therefore, iterating this we see that \eqref{eq:Fixedw} satisfies the desired bound if we can estimate 
\begin{equation}\label{eq:AfterSteinRight}
\begin{split}
\E  \int_{\R^d} 
 &\Big | \sum_{K \in \calD_{k,l}} \varepsilon_K 
 \sum_{I_m^{(k)}=I_{m+1}^{(k)} =K} b_{K, I_m, I_{m+1}}
\langle P_{K, k-1} f_m \rangle_{I_m}
 h_{I_{m+1}}(x)   \Big|_{X_m}^{p_m} \ud x,
\end{split}
\end{equation}
where
$$
|b_{K, I_m, I_{m+1}}| \le \frac{|I_m|^{3/2}}{|K|},
$$
by $\|f_m\|_{L^{p_m}(X_m)}^{p_m}$.
This is a linear estimate and bounded exactly like the right hand side of
\eqref{eq:LinearStandStart}. Therefore, we get the desired bound $\|f_m\|_{L^{p_m}(X_m)}^{p_m}$.
This finally finishes our estimate for the term $A_m$.

We turn to estimate the parts $U_m$ from \eqref{eq:DefUm}. 
Recall that $U_m(f_1,\dots, f_n)$  is by definition
\begin{equation*}
\begin{split}
 \sum_K \sum_{I_1^{(k)} =  \cdots = I_{n+1}^{(k)} =K} b_{K, (I_j)}
\sum_{i=0}^{k-1}\Big(\prod_{j=1}^{m-1}  \langle f_j \rangle_{I_{n+1}^{(i+1)}} \cdot \langle \Delta_{I_{n+1}^{(i+1)}} f_m  \rangle_{I_{n+1}^{(i)}} \cdot \prod_{j=m+1}^n \langle f_j \rangle_{I_{n+1}^{(i)}}
\Big)  h_{I_{n+1}}.
\end{split}
\end{equation*}
Similarly as with the operators $A_m$, we use the fact that 
$L^{p_{n+1}'}(Z_{1, \ldots, n})$ has type $s =\min(p_{n+1}', s_{n+1}')$ to reduce to controlling the term
\begin{equation}\label{eq:UmScaSep}
\begin{split}
 \sum_{K \in \calD_{k,l}}& \sum_{I_1^{(k)} =  \cdots = I_{n+1}^{(k)} =K}\\
 & b_{K, (I_j)}
\sum_{i=0}^{k-1}\Big(\prod_{j=1}^{m-1}  \langle f_j \rangle_{I_{n+1}^{(i+1)}} \cdot \langle \Delta_{I_{n+1}^{(i+1)}} f_m  \rangle_{I_{n+1}^{(i)}} \cdot \prod_{j=m+1}^n \langle f_j \rangle_{I_{n+1}^{(i)}}
\Big)  h_{I_{n+1}}
\end{split}
\end{equation}
uniformly on $l$.
For every $I_{n+1}$ such that  $I_{n+1}^{(k)}=K$ there holds that
$$
\Big|\sum_{I_1^{(k)} =  \cdots = I_{n}^{(k)} =K} b_{K, (I_j)}\Big| \le |I_{n+1}|^{1/2}.
$$
Therefore, using the $\UMD$ property and the Kahane contraction principle the $L^{p_{n+1}'}(Z_{1, \ldots, n})$-norm of 
\eqref{eq:UmScaSep}
is dominated by
\begin{equation*}
\begin{split}
\E\Big \|\sum_{K \in \calD_{k,l}} \varepsilon_{K}  \sum_{I^{(k)} =K}\sum_{i=0}^{k-1}\prod_{j=1}^{m-1}  \langle f_j \rangle_{I^{(i+1)}} \cdot \langle \Delta_{I^{(i+1)}} f_m  \rangle_{I^{(i)}} \cdot \prod_{j=m+1}^n \langle f_j \rangle_{I^{(i)}} 1_{I} \Big \|
_{L^{p_{n+1}'}(Z_{1, \ldots, n})}.
\end{split}
\end{equation*}

If all the averages were on the ``level $i+1$'' we could estimate this directly. Since there are the averages on the ``level $i$'' we need to further
split this. There holds that
\begin{equation*}
\begin{split}
\prod_{j=m+1}^n \langle f_j \rangle_{I^{(i)}}
=\langle f_{m+1} \rangle_{I^{(i+1)}}\prod_{j=m+2}^n \langle f_j \rangle_{I^{(i)}}
+\langle \Delta_{I^{(i+1)}} f_m  \rangle_{I^{(i)}}\prod_{j=m+2}^n \langle f_j \rangle_{I^{(i)}}.
\end{split}
\end{equation*}
Then, both of these are expanded in the same way related to $f_{m+2}$ and so on. This gives that
$$
\prod_{j=m+1}^n \langle f_j \rangle_{I^{(i)}}1_{I}
= \sum_{\varphi} \prod_{j=m+1}^{n}   D_{I^{(i+1)}} ^{\varphi(j)}f_j1_{I}.
$$
Here the summation is over functions $\varphi \colon \{m+1, \dots, n\} \to \{0,1\}$
and for a cube $I \in \calD$ we defined $D_I^0=E_I$ and $D^1_I=\Delta_I$. We also used the fact that
$\langle \Delta_{I^{(i+1)}} f_j \rangle_{I^{(i)}}1_{I}= \Delta_{I^{(i+1)}} f_j 1_{I}$.

Finally, we take one $\varphi$ and estimate the related term. It can be written as
\begin{equation}\label{eq:OnePhi}
\begin{split}
\E\Big \|\sum_{K \in \calD_{k,l}}\varepsilon_K \sum_{\substack{L \subset K \\ \ell(L) >2^{-k} \ell(K)}} \prod_{j=1}^{m-1}  \langle f_j \rangle_{L} 
\cdot  \Delta_{L} f_m  
\cdot \prod_{j=m+1}^{n}   D_{L} ^{\varphi(j)}f_j \Big \|
_{L^{p_{n+1}'}(Z_{1, \ldots, n})}.
\end{split}
\end{equation}
If $\varphi(j)=0$ for all $j=m+1, \dots, n$, 
then this can be estimated by a repeated use of Stein's inequality (similarly as above) related to the functions $f_j$, $j \not=m$.

Suppose that $\varphi(j) \not=0$ for some $j$. Notice that
\begin{equation*}
\begin{split}
\sum_{K \in \calD_{k,l}}&\varepsilon_K \sum_{\substack{L \subset K \\ \ell(L) >2^{-k} \ell(K)}} \prod_{j=1}^{m-1}  \langle f_j \rangle_{L} 
\cdot  \Delta_{L} f_m  
\cdot \prod_{j=m+1}^{n}   D_{L} ^{\varphi(j)}f_j \\
& =\E'\Big(\sum_{K \in \calD_{k,l}}\varepsilon_K \sum_{\substack{L \subset K \\ \ell(L) >2^{-k} \ell(K)}}\varepsilon_L' \prod_{j=1}^{m-1}  \langle f_j \rangle_{L} 
\cdot  \Delta_{L} f_m\Big)\Big(\sum_L \varepsilon_L'\prod_{j=m+1}^{n}   D_{L} ^{\varphi(j)}f_j \Big).
\end{split}
\end{equation*}
Therefore, by Lemma \ref{lem:nice} the term in \eqref{eq:OnePhi} is dominated by
\begin{equation}\label{eq:SeparatedLeft}
\begin{split}
\E\Big(\E'\Big  \| &\Big(\sum_{K \in \calD_{k,l}}\varepsilon_K 
\sum_{\substack{L \subset K \\ \ell(L) >2^{-k} \ell(K)}}\varepsilon_L' \prod_{j=1}^{m-1}  \langle f_j \rangle_{L} 
\cdot  \Delta_{L} f_m\Big) \Big \|^{q_{1, \dots, m}}_{L^{q_{1, \dots, m}}(Z_{1,\dots, m})}\Big)^{1/q_{1, \dots, m}}
\end{split}
\end{equation}
multiplied by
\begin{equation}\label{eq:SeparatedRight}
\Big(\E'\Big \|\sum_L \varepsilon_L'\prod_{j=m+1}^{n}   D_{L} ^{\varphi(j)}f_j  
\Big \|^{q_{m+1, \dots, n}}_{L^{q_{m+1, \dots, n}}(Z_{m+1,\dots, n})}\Big)^{1/q_{m+1, \dots, n}}.
\end{equation}

The term in \eqref{eq:SeparatedLeft} can be estimated by a repeated use of Stein's inequality. 
The term in \eqref{eq:SeparatedRight} is like the term in \eqref{eq:OnePhi}. If there is only one martingale difference, then we can again estimate 
directly with Stein's inequality. If there is at least two martingale differences, 
then one can split into two as we did when we arrived at \eqref{eq:SeparatedLeft} and \eqref{eq:SeparatedRight}. 
This process is continued until one ends up with terms that contain only one martingale difference, and such terms we can estimate.

The proof of Proposition \ref{prop:MultliModUMD} is finished.
\end{proof}

With the essentially same proof as for the terms $A_m$ above we also have the following result.
\begin{prop}\label{prop:MultliStandUMD}
Suppose that $S_{k, \ldots, k}$ is an $n$-linear shift of complexity $(k, \ldots, k)$ and $f_j \colon \R^d \to X_j$. Let $1 < p_j < \infty$ 
with $\sum_{j=1}^{n+1} 1/p_j = 1$.
Then we have
$$
|\langle S_{k, \ldots, k}(f_1, \ldots, f_n), f_{n+1}\rangle| \lesssim (k+1)^{\alpha} \prod_{j=1}^{n+1} \| f_j \|_{L^{p_j}(X_j)},
$$
where
$$
\alpha = \frac{1}{\min( p_1', \ldots, p_{n+1}', s_{1}', \ldots, s_{n+1}')}
$$
and $X_j$ has cotype $s_j$.
\end{prop}
\begin{rem}
Ordinary shifts obey a complexity free bound in the scalar-valued setting. However, we do not know how to achieve this with general $\UMD$ spaces.
It seems that a somewhat better dependency could be obtained, but it would not have any practical use for us.
\end{rem}

The following is \cite[Theorem 5.3]{DLMV1}. This
is a significantly simpler argument than the shift proof and consists of repeated use of Stein's inequality until
one is reduced to the linear case \eqref{eq:XLinPar}.
\begin{prop}\label{prop:MultliParUMD}
Suppose that $\pi$ is an $n$-linear paraproduct and $f_j \colon \R^d \to X_j$. Let $1 < p_j < \infty$ 
with $\sum_{j=1}^{n+1} 1/p_j = 1$.
Then we have
$$
|\langle \pi(f_1, \ldots, f_n), f_{n+1}\rangle| \lesssim \prod_{j=1}^{n+1} \| f_j \|_{L^{p_j}(X_j)}.
$$
\end{prop}

Finally, we are ready to state our main result concerning the $\UMD$ extensions of $n$-linear $\omega$-CZOs.

\begin{thm}\label{thm:multiCZOUMD}
Suppose that $T$ is an $n$-linear $\omega$-CZO. Suppose 
$\omega \in \operatorname{Dini}_{\alpha}$,
where
$$
\alpha = \frac{1}{\min( (n+1)/n, s_{1}', \ldots, s_{n+1}')}
$$
and $X_j$ has cotype $s_j$.
Then for all exponents $1 < p_1, \ldots, p_n \le \infty$ and $1 / q_{n+1} = \sum_{j=1}^n 1/p_j > 0$
we have
$$
\|T(f_1, \ldots, f_n) \|_{L^{q_{n+1}}(X_{n+1}^*)} \lesssim \prod_{j=1}^{n} \| f_j \|_{L^{p_j}(X_j)}.
$$
\end{thm}
\begin{proof}
The important part is to establish the boundedness with a single tuple of exponents. We may e.g. conclude from the boundedness of the model operators
and Theorem \ref{thm:rep1par} that
$$
|\langle T(f_1, \ldots, f_n), f_{n+1}\rangle| \lesssim \prod_{j=1}^{n+1} \| f_j \|_{L^{n+1}(X_j)}
$$
if we choose $\alpha$ as in the statement of the theorem.
It is completely standard how to improve this to cover the full range: we can e.g. prove the end point estimate
$T\colon L^1(X_1) \times \cdots \times L^1(X_n) \to L^{1/n, \infty}(X_{n+1}^*)$, see \cite{LZ}, and then use interpolation or good-$\lambda$ methods. See e.g. \cite{GT, MV}.
For such arguments the spaces $X_j$ no longer play any role (the scalar-valued proofs can readily me mimicked).
\end{proof}
\begin{rem}
The exponent $(n+1)/n$ in the definition of $\alpha$ is slightly annoying, since now the
exponent $\alpha = 1/2$ valid in the scalar-valued case $X_1 = \cdots = X_{n+1} = \C$ (Corollary \ref{cor:multiCZO}) does  
not follow from this result, even though then $s_{1}' = \cdots = s_{n+1}' = 2$.
 Of course, it is way more simple to prove scalar-valued estimates directly with other methods anyway (see Section \ref{sec:scalarval}).

Notice that it is also clear that $\operatorname{Dini}_{1/2}$ suffices in suitable tuples $(X_1, \ldots, X_{n+1})$ of $\UMD$ function lattices. See e.g. \cite[Sections 2.10--2.12]{HMV}
for an account of the well-known square function and maximal function estimates valid in lattices. In lattices the simple approach of Section \ref{sec:scalarval} is much better, as then in addition to the factor $(n+1)/n$ we often have $s_j' < 2$.

In some interesting non-trivial situations the presence of $(n+1)/n$ is not an additional restriction.
Suppose each space $X_j$ is a non-commutative $L^p$ space $L^{p_j}(M)$ and $\sum_{j=1}^{n+1} 1/p_j = 1$, $1 < p_j < \infty$. Then
the cotype of $X_j$ is $s_j = \max(2, p_j) \ge p_j$ so that
$$
1 = \sum_{j=1}^{n+1} \frac{1}{p_j} \ge \sum_{j=1}^{n+1} \frac{1}{s_j} = n+1 - \sum_{j=1}^{n+1} \frac{1}{s_j'},
$$
and so there has to be an index $j$ so that $s_j' \le (n+1)/n$ anyway -- thus $\min( (n+1)/n, s_{1}', \ldots, s_{n+1}')
= \min(s_{1}', \ldots, s_{n+1}')$.
\end{rem}

\section{Bi-parameter singular integrals}\label{sec:bipar}

\subsection*{Bi-parameter SIOs}
Let $\R^d = \R^{d_1} \times \R^{d_2}$ and consider an $n$-linear operator $T$ on $\R^d$. We define what it means
for $T$ to be an $n$-linear bi-parameter SIO.
Let $\omega_i$ be a modulus of continuity on $\R^{d_i}$.
Let $f_j = f_j^1 \otimes f_j^2$, $j = 1, \ldots, n+1$.

First, we set up notation for the adjoints of $T$. We let $T^{j*}$, $j \in \{0, \ldots, n\}$, denote the full adjoints, i.e., 
$T^{0*} = T$ and otherwise
$$
\langle T(f_1, \dots, f_n), f_{n+1} \rangle
= \langle T^{j*}(f_1, \dots, f_{j-1}, f_{n+1}, f_{j+1}, \dots, f_n), f_j \rangle.
$$
A subscript $1$ or $2$ denotes a partial adjoint in the given parameter -- for example, we define
$$
\langle T(f_1, \dots, f_n), f_{n+1} \rangle
= \langle T^{j*}_1(f_1, \dots, f_{j-1}, f_{n+1}^1 \otimes f_j^2, f_{j+1}, \dots, f_n), f_j^1 \otimes f_{n+1}^2 \rangle.
$$
Finally, we can take partial adjoints with respect to different parameters in different slots also -- in that case we denote the adjoint by $T^{j_1*, j_2*}_{1,2}$. It simply interchanges
the functions $f_{j_1}^1$ and $f_{n+1}^1$ and the functions $f_{j_2}^2$ and $f_{n+1}^2$. Of course, we e.g. have $T^{j^*, j^*}_{1,2} = T^{j*}$ and $T^{0*, j^*}_{1,2} = T^{j*}_{2}$,
so everything can be obtained, if desired, with the most general notation $T^{j_1*, j_2*}_{1,2}$.
In any case, there are $(n+1)^2$ adjoints (including $T$ itself). Similarly, the dyadic model operators that we later define always have $(n+1)^2$ different forms.

\subsubsection*{Full kernel representation}
Here we assume that given $m \in \{1,2\}$ there exists $j_1, j_2 \in \{1, \ldots, n+1\}$ so that
$\operatorname{spt} f_{j_1}^m \cap \operatorname{spt} f_{j_2}^m = \emptyset$.
In this case we demand that
$$
\langle T(f_1, \ldots, f_n), f_{n+1}\rangle = \int_{\R^{(n+1)d}}  K(x_{n+1},x_1, \dots, x_n)\prod_{j=1}^{n+1} f_j(x_j) \ud x,
$$
where
$$
K \colon \R^{(n+1)d} \setminus \{ (x_1, \ldots, x_{n+1}) \in \R^{(n+1)d}\colon x_1^1 = \cdots =  x_{n+1}^1 \textup{ or }  x_1^2 = \cdots =  x_{n+1}^2\} \to \C
$$
is a kernel satisfying a set of estimates which we specify next. 

The kernel $K$ is assumed to satisfy the size estimate
\begin{displaymath}
|K(x_{n+1},x_1, \dots, x_n)| \lesssim \prod_{m=1}^2 \frac{1}{\Big(\sum_{j=1}^{n} |x_{n+1}^m-x_j^m|\Big)^{d_mn}}.
\end{displaymath}

We also require the following continuity estimates -- to which we continue to refer to as H\"older estimates despite the general continuity moduli. For example, we require that we have
\begin{align*}
|K(x_{n+1}, x_1, \ldots, x_n)-&K(x_{n+1},x_1, \dots, x_{n-1}, (c^1,x^2_n))\\
&-K((x_{n+1}^1,c^2),x_1, \dots, x_n)+K((x_{n+1}^1,c^2),x_1, \dots, x_{n-1},  (c^1,x^2_n))| \\
&\qquad \lesssim \omega_1 \Big( \frac{|x_{n}^1-c^1| }{ \sum_{j=1}^{n} |x_{n+1}^1-x_j^1|} \Big) 
\frac{1}{\Big(\sum_{j=1}^{n} |x_{n+1}^1-x_j^1|\Big)^{d_1n}} \\
&\qquad\times
\omega_2 \Big( \frac{|x_{n+1}^2-c^2| }{ \sum_{j=1}^{n} |x_{n+1}^2-x_j^2|} \Big) 
\frac{1}{\Big(\sum_{j=1}^{n} |x_{n+1}^2-x_j^2|\Big)^{d_2n}}
\end{align*}
whenever $|x_n^1-c^1| \le 2^{-1} \max_{1 \le i \le n} |x_{n+1}^1-x_i^1|$
and $|x_{n+1}^2-c^2| \le 2^{-1} \max_{1 \le i \le n} |x_{n+1}^2-x_i^2|$.
Of course, we also require all the other natural symmetric estimates, where $c^1$ can be in any of the given $n+1$ slots and similarly for $c^2$. There
are, of course, $(n+1)^2$ different estimates.

Finally, we require the following mixed H\"older and size estimates. For example, we ask that
\begin{align*}
|K(x_{n+1}&, x_1, \ldots, x_n)-K(x_{n+1},x_1, \dots, x_{n-1}, (c^1,x^2_n))| \\
& \lesssim \omega_1 \Big( \frac{|x_{n}^1-c^1| }{ \sum_{j=1}^{n} |x_{n+1}^1-x_j^1|} \Big) 
\frac{1}{\Big(\sum_{j=1}^{n} |x_{n+1}^1-x_j^1|\Big)^{d_1n}} \cdot  \frac{1}{\Big(\sum_{j=1}^{n} |x_{n+1}^2-x_j^2|\Big)^{d_2n}}
\end{align*}
whenever $|x_n^1-c^1| \le 2^{-1} \max_{1 \le i \le n} |x_{n+1}^1-x_i^1|$. Again, we also require all the other natural symmetric estimates.

\subsubsection*{Partial kernel representations}
Suppose now only that there exists $j_1, j_2 \in \{1, \ldots, n+1\}$ so that
$\operatorname{spt} f_{j_1}^1 \cap \operatorname{spt} f_{j_2}^1 = \emptyset$.
 Then we assume that
$$
\langle T(f_1, \ldots, f_n), f_{n+1}\rangle = \int_{\R^{(n+1)d_1}} K_{(f_j^2)}(x_{n+1}^1, x_1^1, \ldots, x_n^1) \prod_{j=1}^{n+1} f_j^1(x^1_j) \ud x^1,
$$
where $K_{(f_j^2)}$ is a one-parameter $\omega_1$-Calder\'on--Zygmund kernel as in Section \ref{sec:1par} but with a constant depending on the fixed functions $f_1^2, \ldots, f_{n+1}^2$.
For example, this means that the size estimate takes the form
$$
|K_{(f_j^2)}(x_{n+1}^1, x_1^1, \ldots, x_n^1)| \le C(f_1^2, \ldots, f_{n+1}^2) \frac{1}{\Big(\sum_{j=1}^{n} |x_{n+1}^1-x_j^1|\Big)^{d_1n}}.
$$
The continuity estimates are analogous.

We assume the following $T1$ type control on the constant $C(f_1^2, \ldots, f_{n+1}^2)$. We have
\begin{equation}\label{eq:PKWBP}
C(1_{I^2}, \ldots, 1_{I^2}) \lesssim |I^2|
\end{equation}
and
$$
C(a_{I^2}, 1_{I^2}, \ldots, 1_{I^2}) + C(1_{I^2}, a_{I^2}, 1_{I^2}, \ldots, 1_{I^2}) + \cdots + C(1_{I^2}, \ldots, 1_{I^2}, a_{I^2}) \lesssim |I^2|
$$
for all cubes $I^2 \subset \R^{d_2}$
and all functions $a_{I^2}$ satisfying $a_{I^2} = 1_{I^2}a_{I^2}$, $|a_{I^2}| \le 1$ and $\int a_{I^2} = 0$.

Analogous partial kernel representation on the second parameter is assumed when $\operatorname{spt} f_{j_1}^2 \cap \operatorname{spt} f_{j_2}^2 = \emptyset$
for some $j_1, j_2$.

\begin{defn}
If $T$ is an $n$-linear operator with full and partial kernel representations as defined above, we call $T$ an $n$-linear bi-parameter $(\omega_1, \omega_2)$-SIO.
\end{defn}

\subsection*{Bi-parameter CZOs}
We say that $T$ satisfies the weak boundedness property if
\begin{equation}\label{eq:2ParWBP}
|\langle T(1_R, \ldots, 1_R), 1_R \rangle| \lesssim |R|
\end{equation}
for all rectangles $R = I^1 \times I^2 \subset \R^{d} = \R^{d_1} \times \R^{d_2}$.

An SIO $T$ satisfies the diagonal BMO assumption if the following holds. For all rectangles $R = I^1 \times I^2 \subset \R^{d} = \R^{d_1} \times \R^{d_2}$
and functions $a_{I^i}$ with $a_{I^i} = 1_{I^i}a_{I^i}$, $|a_{I^i}| \le 1$ and $\int a_{I^i} = 0$ we have
\begin{equation}\label{eq:DiagBMO}
|\langle T(a_{I^1} \otimes 1_{I^2}, 1_R, \ldots, 1_R), 1_R \rangle| + \cdots +  |\langle T(1_R, \ldots, 1_R), a_{I^1} \otimes 1_{I^2} \rangle| \lesssim |R|
\end{equation}
and
$$
|\langle T(1_{I^1} \otimes a_{I^2}, 1_R, \ldots, 1_R), 1_R \rangle| + \cdots +  |\langle T(1_R, \ldots, 1_R), 1_{I^1} \otimes a_{I^2} \rangle| \lesssim |R|.
$$

The product $\BMO$ space is originally by Chang and Fefferman \cite{CF1, CF2}, and
it is the right bi-parameter $\BMO$ space for many considerations. 
An SIO $T$ satisfies the product BMO assumption if it holds
$$S(1,1) \in \BMO_{\textup{prod}}$$ for all the $(n+1)^2$ adjoints $S = T^{j_1*, j_2*}_{1,2}$.
This can be interpreted in the sense that
$$
\| S1 \|_{\BMO_{\operatorname{prod}}} = \sup_{\calD = \calD^1 \times \calD^2} \sup_{\Omega} \Big(\frac{1}{|\Omega|} \sum_{ \substack{ R = I^1 \times I^2 \in \calD \\
R \subset \Omega}} |\langle S1, h_R \rangle|^2 \Big)^{1/2} < \infty, \qquad h_R = h_{I^1} \otimes h_{I^2},
$$
where the supremum is over all dyadic grids $\calD^i$ on $\R^{d_i}$ and
open sets $\Omega \subset \R^d = \R^{d_1} \times \R^{d_2}$ with $0 < |\Omega| < \infty$, and the pairings
$\langle S1, h_R\rangle$ can be defined, in a natural way, using the kernel representations.

\begin{defn}\label{defn:CZO}
An $n$-linear  bi-parameter $(\omega_1, \omega_2)$-SIO $T$
satisfying the weak boundedness property, the diagonal BMO assumption and the product BMO assumption is called an $n$-linear bi-parameter
$(\omega_1, \omega_2)$-Calder\'on--Zygmund operator ($(\omega_1, \omega_2)$-CZO). 
\end{defn}

\subsection*{General bi-parameter notation and basic operators}
A weight $w(x_1, x_2)$ (i.e. a locally integrable a.e. positive function) belongs to the bi-parameter weight class $A_p(\R^{d_1} \times \R^{d_2})$, $1 < p < \infty$, if
$$
[w]_{A_p(\R^{d_1} \times \R^{d_2})} := \sup_{R} \frac 1{|R|}\int_R w   \Bigg( \frac 1{|R|}\int_R w^{1-p'}\Bigg)^{p-1} < \infty,
$$
where the supremum is taken over $R = I^1 \times I^2$ and each $I^i \subset \R^{d_i}$ is a cube. Thus, this is the one-parameter definition but cubes are replaced by rectangles.

We have
$$
[w]_{A_p(\R^{d_1} \times \R^{d_2})} < \infty \textup { iff } \max\big( \esssup_{x_1 \in \R^{d_1}} \,[w(x_1, \cdot)]_{A_p(\R^{d_2})}, \esssup_{x_2 \in \R^{d_2}}\, [w(\cdot, x_2)]_{A_p(\R^{d_1})} \big) < \infty,
$$
and that $$
\max\big( \esssup_{x_1 \in \R^{d_1}} \,[w(x_1, \cdot)]_{A_p(\R^{d_2})}, \esssup_{x_2 \in \R^{d_2}}\, [w(\cdot, x_2)]_{A_p(\R^{d_1})} \big) \le [w]_{A_p(\R^{d_1}\times \R^{d_2})},
$$
while the constant $[w]_{A_p}$ is dominated by the maximum to some power. For basic bi-parameter weighted theory see e.g. \cite{HPW}.
We say $w\in A_\infty(\R^{d_1}\times \R^{d_2})$ if
\[
[w]_{A_\infty(\R^{d_1}\times \R^{d_2})}:=\sup_R \frac 1{|R|}\int_R w  \exp\Bigg( \frac1{|R|}\int_R \log (w^{-1})  \Bigg)<\infty.
\]
It is well-known that
$$A_\infty(\R^{d_1}\times \R^{d_2})=\bigcup_{1<p<\infty}A_p(\R^{d_1}\times \R^{d_2}).$$
We do not have any important use for the $A_{\infty}$ constant. The $w \in A_{\infty}$ assumption
can always be replaced with the explicit assumption $w \in A_s$ for some $s \in (1,\infty)$, and
then estimating everything with a dependence on $[w]_{A_s}$.

We denote a general dyadic grid in $\R^{d_i}$ by $\calD^i$. We denote cubes in $\calD^i$ by $I^i, J^i, K^i$, etc.
Thus, our dyadic rectangles take the forms $I^1 \times I^2$, $J^1 \times J^2$, $K^1 \times K^2$ etc.

If $A$ is an operator acting on $\R^{d_1}$, we can always let it act on the product space $\R^d = \R^{d_1} \times \R^{d_2}$ by setting $A^1f(x) = A(f(\cdot, x_2))(x_1)$. Similarly, we use
the notation $A^2 f(x) = A(f(x_1, \cdot))(x_2)$ if $A$ is originally an operator acting on $\R^{d_2}$. Our basic bi-parameter dyadic operators -- martingale differences and averaging operators -- are obtained by simply chaining together relevant one-parameter operators. For instance, a bi-parameter martingale difference is
$\Delta_R f = \Delta_{I^1}^1 \Delta_{I^2}^2 f$, $R = I^1 \times I^2$. Bi-parameter estimates, such as the square function bound
$$
\Big\| \Big( \sum_{R \in \calD^1 \times \calD^2} |\Delta_R f|^2 \Big)^{1/2} \Big\|_{L^p(w)} = 
\Big\| \Big( \sum_{I^i \in \calD^i} |\Delta_{I^1}^1 \Delta_{I^2}^2 f|^2 \Big)^{1/2} \Big\|_{L^p(w)} \sim \|f\|_{L^p(w)},
$$
where $p \in (1,\infty)$ and $w$ is a bi-parameter $A_p$ weight,
are easily obtained using vector-valued versions of the corresponding one-parameter estimates. The required vector-valued estimates, on the other hand, follow simply
by extrapolating the obvious weighted $L^2(w)$ estimates.

We systematically collect maximal function and square function bounds now. First, some notation.
When we integrate with respect to only one of the parameters we may e.g. write
\[
\langle f, h_{I_1} \rangle_1(x_2):=\int_{\R^{d_1}} f(x_1, x_2)h_{I_1}(x_1) \ud x_1.
\]
If $\calD = \calD^1 \times \calD^2$ we define the dyadic bi-parameter maximal function
$$
M_{\calD} f:= \sup_{R \in \calD}    1_R \bla |f|\bra_R.
$$
Now define the square functions
$$
S_{\calD} f = \Big( \sum_{R \in \calD}  |\Delta_R f|^2 \Big)^{1/2}, \,\, S_{\calD^1}^1 f =  \Big( \sum_{I^1 \in \calD^1}  |\Delta_{I^1}^1 f|^2 \Big)^{1/2}
$$
and define $S_{\calD^2}^2 f$ analogously.
Define also
$$
S_{\calD, M}^1 f = \Big( \sum_{I^1 \in \calD^1} \frac{1_{I^1}}{|I^1|} \otimes \big[M_{\calD^2} \bla f, h_{I^1} \bra_1\big]^2 \Big)^{1/2}, \,\, S_{\calD, M}^2 f = \Big( \sum_{I^2 \in \calD^2} \big[M_{\calD^1} \bla f, h_{I^2} \bra_2\big]^2 \otimes \frac{1_{I^2}}{|I^2|}\Big)^{1/2}.
$$

Let $k=(k_1,k_2)$, where $k_i \in \{0,1,2, \dots,\}$, and $K=K^1 \times K^2 \in \calD$. 
We set
$$
P^1_{K^1,k_1}f=\sum_{\substack{I^1 \in \calD^1 \\ I^1 \subset K^1 \\ \ell(I^1) \ge 2^{-k_1}\ell(K^1)}} \Delta^1_{I^1} f
$$
and define similarly $P^2_{K^2,k_2}$. Then, we define $P_{K,k}:= P^1_{K^1,k_1}P^2_{K^2,k_2}$.

\begin{lem}\label{lem:standardEst}
For $p \in (1,\infty)$ and a bi-parameter weight $w \in A_p$ we have
$$
\| f \|_{L^p(w)}
 \sim  \| S_{\calD} f\|_{L^p(w)}
\sim   \| S_{\calD^1}^1 f  \|_{L^p(w)}
\sim  \| S_{\calD^2}^2 f  \|_{L^p(w)}.
$$
For $k=(k_1,k_2)$, $k_i \in \{0,1, \dots, \}$,  we have the estimates
$$
\Big\| \Big( \sum_{K \in \calD} | P_{K,k}f|^2 \Big)^{1/2} \Big \|_{L^p(w)}
\lesssim \sqrt{k_1+1} \sqrt{k_2+1} \| f \|_{L^p(w)},
$$
$$
\Big\| \Big( \sum_{K^1 \in \calD^1} | P^1_{K^1,k_1}f|^2 \Big)^{1/2} \Big \|_{L^p(w)}
\lesssim \sqrt{k_1+1} \| f \|_{L^p(w)}
$$
and the analogous estimate with $P^2_{K^2,k_2}$.

Moreover, for $p, s \in (1,\infty)$ we have the Fefferman--Stein inequality
$$
\Big\| \Big( \sum_j |M f_j |^s \Big)^{1/s} \Big\|_{L^p(w)} \lesssim  \Big\| \Big( \sum_{j} | f_j |^s \Big)^{1/s} \Big\|_{L^p(w)}.
$$
Here $M$ can e.g. be $M_{\calD^1}^1$ or $M_{\calD}$. Finally, we have
$$
\| S_{\calD, M}^1 f\|_{L^p(w)} + \| S_{\calD, M}^2 f\|_{L^p(w)} \lesssim  \|f\|_{L^p(w)}.
$$
\end{lem}
The following lower square function estimate valid for $A_{\infty}$ weights is important for us.
\begin{lem}\label{lem:LowerFS} There holds
$$
\|f\|_{L^p(w)} \lesssim \|S_{\calD^1}^1 f\|_{L^p(w)}
$$
and
$$
\|f\|_{L^p(w)} \lesssim \|S_{\calD} f\|_{L^p(w)}
$$
for all $p \in (0, \infty)$ and bi-parameter weights $w \in A_{\infty}$.
\end{lem}
For a proof of the one-parameter estimate, which improves to the bi-parameter estimate, 
see \cite[Theorem 2.5]{Wi}. This is important to us
as the product weights $w := \prod_{j=1}^n w_j^{r/p_j}$, where $p_j \in (1, \infty)$, $1/r:= \sum_{j=1}^n 1/p_j$ and $w_j \in A_{p_j}$ are bi-parameter weights, are
at least $A_{\infty}$ -- in fact they belong to $A_{2r}$. However, they do not need to belong to $A_r$ making some duality based proofs problematic. This issue
is not present in the linear situation accounting for the fact that weighted estimates are much easier in the linear situation.

\subsection*{Bi-parameter model operators}
As the bi-parameter CZOs are modelled after tensor products, their representation by model operators
will e.g. involve generalisations of $Q_{k_1} \otimes Q_{k_2}$ (a modified bi-parameter shift), of $Q_{k_1} \otimes \pi$
and $\pi \otimes Q_{k_2}$ (a modified partial paraproduct) and $\pi \otimes \pi$ (a bi-parameter full paraproduct). All possible combinations
of the model operators appearing in the one-parameter result, Theorem \ref{thm:rep1par}, will appear. It may seem complicated, but
the reader should recall that while the operators $S_{k, \ldots, k}$ are formally different than the operators $Q_k$, they are always simpler, and
that in the representation we will not need standard shifts of arbitrary complexities. 
Moreover, at least in the linear situation, it can be convenient to define the operators $Q_k$ with the functions $H_{I,J}$ so that they incorporate
the linear standard shifts $S_{k, k}$.

We will use
suggestive notation, such as, $(S\pi)_i$ to denote a bi-parameter operator that behaves like an ordinary $n$-linear
shift $S_i$ on the first parameter and like an $n$-linear paraproduct on the second -- but this is just notation and our operators are not of tensor product form.
When we e.g. have an operator of the type $(QQ)_{k_1, k_2}$ -- an operator that behaves like a modified shift $Q_{k_m}$ on the parameter $m$, $m = 1,2$,
we simply write $Q_{k_1,k_2}$.

\subsubsection*{Shifts}
Let $i=(i_1, \dots, i_{n+1})$, where $i_j = (i_j^1, i_j^2) \in \{0,1,\ldots\}^2$. 
An $n$-linear bi-parameter shift $S_i$ takes the form
\begin{equation*}\label{eq:S2par}
\langle S_i(f_1, \ldots, f_n), f_{n+1}\rangle = \sum_{K} \sum_{\substack{R_1, \ldots, R_{n+1} \\ R_j^{(i_j)} = K }}
a_{K, (R_j)} \prod_{j=1}^{n+1} \langle f_j, \wt h_{R_j} \rangle.
\end{equation*}
Here $K, R_1, \ldots, R_{n+1} \in \calD = \calD^1 \times \calD^2$, $R_j = I_j^1 \times I_j^2$, $R_j^{(i_j)} := (I_j^1)^{(i_j^1)} \times (I_j^2)^{(i_j^2)}$ and 
$\wt h_{R_j} = \wt h_{I_j^1} \otimes \wt h_{I_j^2}$. Here we assume that for $m \in \{1,2\}$
there exist two indices $j_0,j_1 \in \{1, \ldots, n+1\}$, $j_0 \not =j_1$, so that $\wt h_{I_{j_0}^m}=h_{I_{j_0}^m}$, $\wt h_{I_{j_1}^m}=h_{I_{j_1}^m}$ and for the remaining indices $j \not \in \{j_0, j_1\}$ we have $\wt h_{I_j^m} \in \{h_{I_j^m}^0, h_{I_j^m}\}$.
Moreover, $a_{K,(R_j)} = a_{K, R_1, \ldots ,R_{n+1}}$ is a scalar satisfying the normalization
\begin{equation}\label{eq:Snorm2par}
|a_{K,(R_j)}| \le \frac{\prod_{j=1}^{n+1} |R_j|^{1/2}}{|K|^{n}}.
\end{equation}

We continue to define modified shifts.
Let
\begin{equation}\label{eq:DefAR}
A_{R_1, \ldots, R_{n+1}}^{j_1, j_2}(f_1, \ldots, f_{n+1}) = A_{R_1, \ldots, R_{n+1}}^{j_1, j_2} := \prod_{j=1}^{n+1} \langle f_j, \wt h_{R_j} \rangle,
\end{equation}
where $\wt h_{R_j} = \wt h_{I_j^1} \otimes \wt h_{I_j^2}$, $\wt h_{I_{j_1}^1}  = h_{I_{j_1}^1}$, $\wt h_{I_{j}^1}  = h_{I_{j}^1}^0$, $j \ne j_1$,
$\wt h_{I_{j_2}^2}  = h_{I_{j_2}^2}$, $\wt h_{I_{j}^2}  = h_{I_{j}^2}^0$, $j \ne j_2$.
A modified $n$-linear bi-parameter shift $Q_k$, $k = (k_1, k_2)$, takes the form
\begin{equation*}\label{eq:Q2par}
\begin{split}
\langle Q_{k}(f_1, \ldots, f_n), f_{n+1}\rangle = \sum_{K} \sum_{\substack{R_1, \ldots, R_{n+1} \\ R_j^{(k)} = K }}
a_{K, (R_j)} \big[&A_{R_1, \ldots, R_{n+1}}^{j_1, j_2} - A_{I_{j_1}^1 \times I_1^2, \ldots, I_{j_1}^1 \times I_{n+1}^2}^{j_1, j_2}  \\
&-A_{I_1^1 \times I_{j_2}^2, \ldots, I_{n+1}^1 \times I_{j_2}^2}^{j_1, j_2} + A_{I_{j_1}^1 \times I_{j_2}^2, \ldots, I_{j_1}^1 \times I_{j_2}^2}^{j_1, j_2}\big]
\end{split}
\end{equation*}
for some $j_1, j_2$. Moreover, $a_{K,(R_j)} = a_{K, R_1, \ldots ,R_{n+1}}$ is a scalar satisfying the usual normalization \eqref{eq:Snorm2par}.

We now define the hybrid operators that behave like a modified shift in one of the parameters and like a standard shift in the other. 
A modified/standard $n$-linear bi-parameter shift $(QS)_{k, i}$, $i = (i_1, \ldots, i_{n+1})$, $k, i_j \in \{0, 1, \ldots\}$, takes the form
\begin{equation*}\label{eq:QS}
\begin{split}
\langle (QS)_{k,i}&(f_1, \ldots, f_n), f_{n+1}\rangle  \\
&= \sum_{K} \sum_{\substack{R_1, \ldots, R_{n+1} \\ R_j^{(k, i_j)} = K }}
a_{K, (R_j)} 
\Big[ \prod_{j=1}^{n+1} \langle f_j, \wt h_{R_j} \rangle - \prod_{j=1}^{n+1} \langle f_j, \wt h_{I_{j_0}^1 \times I_j^2} \rangle \Big]
\end{split}
\end{equation*}
for some $j_0$. Here we assume that $\wt h_{I_{j_0}^1}  = h_{I_{j_0}^1}$, $\wt h_{I_{j}^1}  = h_{I_{j}^1}^0$ for $j \ne j_0$, and that
there exist two indices $j_1,j_2 \in \{1, \ldots, n+1\}$, $j_1 \not =j_2$, so that $\wt h_{I_{j_1}^2}=h_{I_{j_1}^2}$, $\wt h_{I_{j_2}^2}=h_{I_{j_2}^2}$ and for the remaining indices $j \not \in \{j_1, j_2\}$ we have $\wt h_{I_j^2} \in \{h_{I_j^2}^0, h_{I_j^2}\}$
Moreover, $a_{K,(R_j)} = a_{K, R_1, \ldots ,R_{n+1}}$ is a scalar satisfying the usual normalization \eqref{eq:Snorm2par}.
Of course, $(SQ)_{i,k}$ is defined symmetrically.

\begin{rem}\label{rem:BiParH}
Similarly as in the one-parameter situation, in the representation theorem we only encounter the particular standard shifts $S_{((k_1,k_2), \ldots, (k_1,k_2))}$ and standard/modifed shifts
$(QS)_{k_1, (k_2, \ldots, k_2)}$. In the linear situation all of these can be incorporated, if desired, by defining the modifed shifts in the following higher generality.
A linear modified bi-parameter shift $Q_{k_1, k_2}$ takes the form
\begin{equation*}\label{eq:biparQform1}
\langle Q_{k_1, k_2}f, g\rangle = \sum_{K} \sum_{\substack{ R_1, R_2 \\ R_j^{(k_1, k_2)} = K }}
a_{K, R_1, R_2} \langle f, h_{R_1} \rangle \langle g, H_{I_1^1, I_2^1} \otimes H_{I_1^2, I_2^2} \rangle
\end{equation*}
or one of the three other possible forms, where the functions $h$ and $H$ can be interchanged in each parameter. 
Here the coefficients satisfy the usual normalization and the functions $H$ are like in the one-parameter situation.
The corresponding definition can be made also with the modified partial paraproducts.
\end{rem}

\subsubsection*{Partial paraproducts}
Partial paraproducts are hybrids of $\pi$ and $S$ or $\pi$ and $Q$.

Let $i=(i_1, \dots, i_{n+1})$, where $i_j \in \{0,1,\ldots\}$.
An $n$-linear bi-parameter partial paraproduct $(S\pi)_i$ with the paraproduct component on $\R^{d_2}$ takes the form
\begin{equation}\label{eq:Spi}
\langle (S\pi)_i(f_1, \ldots, f_n), f_{n+1} \rangle = 
\sum_{K = K^1 \times K^2} \sum_{\substack{ I^1_1, \ldots, I_{n+1}^1 \\ (I_j^1)^{(i_j)} = K^1}} a_{K, (I_j^1)} \prod_{j=1}^{n+1} \langle f_j, \wt h_{I_j^1} \otimes u_{j, K^2} \rangle,
\end{equation}
where the functions $\wt h_{I_j^1}$ and $u_{j, K^2}$ satisfy the following.
There are $j_0,j_1 \in \{1, \ldots, n+1\}$, $j_0 \not =j_1$, so that $\wt h_{I_{j_0}^1}=h_{I_{j_0}^1}$, $\wt h_{I_{j_1}^1}=h_{I_{j_1}^1}$ and for the remaining indices $j \not \in \{j_0, j_1\}$ we have $\wt h_{I_j^1} \in \{h_{I_j^1}^0, h_{I_j^1}\}$. There is $j_2 \in \{1, \ldots, n+1\}$ so that $u_{j_2, K^2} = h_{K^2}$ and for the remaining indices $j \ne j_2$ we have
$u_{j, K^2} = \frac{1_{K^2}}{|K^2|}$.
Moreover, the coefficients are assumed to satisfy
$$
\| (a_{K, (I_j^1)})_{K_2} \|_{\BMO} \le \frac{\prod_{j=1}^{n+1} |I_j^1|^{1/2}}{|K^1|^{n}}.
$$
Of course, $(\pi S)_i$ is defined symmetrically.

A modified $n$-linear partial paraproduct $(Q\pi)_{k}$ with the paraproduct component on $\R^{d_2}$ takes the form
\begin{equation*}\label{eq:Qpi}
\begin{split}
\langle (Q\pi)_k&(f_1, \ldots, f_n), f_{n+1} \rangle \\ &= 
\sum_{K = K^1 \times K^2} \sum_{\substack{ I^1_1, \ldots, I_{n+1}^1 \\ (I_j^1)^{(k)} = K^1}} a_{K, (I_j^1)}
\Big[ \prod_{j=1}^{n+1} \langle f_j, \wt h_{I_j^1} \otimes u_{j, K^2} \rangle - \prod_{j=1}^{n+1} \langle f_j, \wt h_{I_{j_0}^1} \otimes u_{j, K^2} \rangle \Big]
\end{split}
\end{equation*}
for some $j_0$ -- here $\wt h_{I_{j_0}^1} = h_{I_{j_0}^1}$, $\wt h_{I_{j}^1} = h_{I_{j}^1}^0$ for $j \ne j_0$ and $u_{j, K^2}$ are like in \eqref{eq:Spi}.
The constants satisfy the same normalization.

\subsubsection*{Full paraproducts}
An $n$-linear bi-parameter full paraproduct $\Pi$ takes the form
\begin{equation*}\label{eq:pi2bar}
\langle \Pi(f_1, \ldots, f_n) , f_{n+1} \rangle = \sum_{K = K^1 \times K^2} a_{K} \prod_{j=1}^{n+1} \langle f_j, u_{j, K^1} \otimes u_{j, K^2} \rangle,
\end{equation*}
where the functions $u_{j, K^1}$ and $u_{j, K^2}$ are like in \eqref{eq:Spi}.
The coefficients are assumed to satisfy
$$
\| (a_{K} ) \|_{\BMO_{\operatorname{prod}}} = \sup_{\Omega} \Big(\frac{1}{|\Omega|} \sum_{K\subset \Omega} |a_{K}|^2 \Big)^{1/2} \le 1,
$$
where the supremum is over open sets $\Omega \subset \R^d = \R^{d_1} \times \R^{d_2}$ with $0 < |\Omega| < \infty$.

\subsection*{Comparison to the usual model operators}

As in the one-parameter case, Lemma \ref{lem:CompModStand}, the modified model operators can be written as suitable sums of the standard operators. 
The exact formulas can be guessed by e.g. decomposing a tensor product $Q_{k_1} \otimes Q_{k_2}$ using Lemma \ref{lem:CompModStand}. The general bi-parameter case
requires some work, and we only give the following formulation.

\begin{lem}\label{lem:QasSBiPar}
Let $Q_k$, $k = (k_1, k_2)$, be a modified $n$-linear bi-parameter shift. 
Then
$$
Q_{k}
=C\sum_{u=1}^{c} \sum_{i_1=0}^{k_1-1} \sum_{i_2=0}^{k_2-1} 
S^{u,i_1,i_2},
$$
where each $S = S^{u,i_1,i_2}$ is a standard $n$-linear bi-parameter shift of complexity $i^m_{S, j}$, $j \in \{1, \ldots, n+1\}$, $m \in \{1,2\}$,
satisfying 
$$
i^{m}_{S, j} \le k_m.
$$

Similarly, a modified/standard shift can be represented using standard shifts and a modified partial paraproduct can be represented using standard partial paraproducts.
\end{lem}

\begin{proof}
For notational convenience we consider a shift $Q_k$ of the particular form
\begin{equation}\label{eq:ParticularShift}
\begin{split}
\langle Q_{k}(f_1, \ldots, f_n), f_{n+1}\rangle = \sum_{K} & \sum_{\substack{R_1, \ldots, R_{n+1} \\ R_j^{(k)} = K }}
a_{K, (R_j)} \big[A_{R_1, \ldots, R_{n+1}}^{n+1, n+1} - A_{I_{n+1}^1 \times I_1^2, \ldots, I_{n+1}^1 \times I_{n+1}^2}^{n+1,n+1}  \\
&-A_{I_1^1 \times I_{n+1}^2, \ldots, I_{n+1}^1 \times I_{n+1}^2}^{n+1,n+1} + 
A_{I_{n+1}^1 \times I_{n+1}^2, \ldots, I_{n+1}^1 \times I_{n+1}^2}^{n+1,n+1}\big].
\end{split}
\end{equation}
There is no essential difference in the general case.

We define
$$
b_{K, (R_j)}
=|R_1|^{n/2}a_{K, (R_j)}
$$
and
$$
B_{R_1, \ldots, R_{n+1}}^{n+1, n+1}
=\prod_{j=1}^n \langle f_j \rangle_{R_j} \langle f_{n+1}, h_{R_{n+1}} \rangle.
$$
We can write the shift with these similarly as in \eqref{eq:ParticularShift} just by replacing $a$ with $b$ and $A$ with $B$.

Comparing to the one-parameter case the following decompositions correspond to \eqref{eq:EPSplitting} and later 
we do the steps corresponding to \eqref{eq:MultilinCollapse} and \eqref{eq:nCollapseRelated}.

For the moment we define the following shorthand. For a cube $I$ and integers $l,j_0 \in \{1,2, \dots \}$ we define
\begin{equation}\label{eq:Dnot}
D_{I,l}(j,j_0)=
\begin{cases}
E_I, \quad &\text{if } j \in \{1, \dots, j_0-1\}, \\ 
P_{I,l-1}, \quad &\text{if } j=j_0, \\
\Id, \quad &\text{if } j \in \{j_0+1,j_0+2, \dots\}, 
\end{cases}
\end{equation}
where $\Id$ denotes the identity operator.

Let $R_1, \dots, R_{n+1}$ be as in the summation of $Q_k$. We use the above notation in both parameters, and we denote this, as usual, with superscripts 
$D^1_{I,l}(j,j_0)$ and $D^2_{I,l}(j,j_0)$.
Using \eqref{eq:EPSplitting} in both parameters separately
we have that
\begin{equation*}
\begin{split}
B_{R_1, \ldots, R_{n+1}}^{n+1, n+1}
&=\sum_{m_1,m_2=1}^{n+1}
\prod_{j=1}^{n} \langle D^1_{K^1,k_1}(j,m_1)D^2_{K^2,k_2}(j,m_2)f_j \rangle_{R_j} \langle f_{n+1}, h_{R_{n+1}} \rangle,
\end{split}
\end{equation*}
which gives that
$$
\sum_{K} \sum_{\substack{R_1, \ldots, R_{n+1} \\ R_j^{(k)} = K }}B_{R_1, \ldots, R_{n+1}}^{n+1, n+1}
=:\sum_{m_1,m_2=1}^{n+1}\Sigma_{m_1,m_2}^1.
$$
Also, we have that 
\begin{equation*}
B_{I_{n+1}^1 \times I_1^2, \ldots, I_{n+1}^1 \times I_{n+1}^2}^{n+1,n+1}
=\sum_{m_2=1}^{n+1} \prod_{j=1}^n \langle D^2_{K^2,k_2}(j,m_2)f_j \rangle _{I^1_{n+1} \times I^2_j} \langle f_{n+1}, h_{R_{n+1}} \rangle
\end{equation*}
and
\begin{equation*}
B_{I_1^1 \times I_{n+1}^2, \ldots, I_{n+1}^1 \times I_{n+1}^2}^{n+1,n+1}
=\sum_{m_1=1}^{n+1} \prod_{j=1}^n \langle D^1_{K^1,k_1}(j,m_1)f_j \rangle _{I^1_{j} \times I^2_{n+1}} \langle f_{n+1}, h_{R_{n+1}} \rangle,
\end{equation*}
which gives that
$$
\sum_{K} \sum_{\substack{R_1, \ldots, R_{n+1} \\ R_j^{(k)} = K }}
B_{I_{n+1}^1 \times I_1^2, \ldots, I_{n+1}^1 \times I_{n+1}^2}^{n+1,n+1}
=: \sum_{m_2=1}^{n+1} \Sigma_{m_2}^2
$$
and 
$$
\sum_{K} \sum_{\substack{R_1, \ldots, R_{n+1} \\ R_j^{(k)} = K }}
B_{I_1^1 \times I_{n+1}^2, \ldots, I_{n+1}^1 \times I_{n+1}^2}^{n+1,n+1}
=: \sum_{m_1=1}^{n+1} \Sigma^3_{m_1}.
$$
Finally, we write that
$$
\sum_{K} \sum_{\substack{R_1, \ldots, R_{n+1} \\ R_j^{(k)} = K }}
B_{I_{n+1}^1 \times I_{n+1}^2, \ldots, I_{n+1}^1 \times I_{n+1}^2}^{n+1,n+1}=: \Sigma^4.
$$

Using the above decompositions we have the identity
\begin{equation*}
\begin{split}
\langle Q_{k}(f_1, \ldots, f_n), f_{n+1}\rangle
&= \sum_{m_1,m_2=1}^n \Sigma^1_{m_1,m_2}
+ \sum_{m_2=1}^n (\Sigma^1_{n+1,m_2}-\Sigma^2_{m_2})\\
&+\sum_{m_1=1}^n (\Sigma^1_{m_1,n+1}-\Sigma^3_{m_1})
+ (\Sigma^1_{n+1,n+1}-\Sigma^2_{n+1}-\Sigma^3_{n+1}+\Sigma^4).
\end{split}
\end{equation*}
The terms $\Sigma^1_{m_1,m_2}$ with $m_1,m_2 \in \{1, \dots, n\}$ and the terms inside the parentheses will be written
as sums of standard shifts.

First, we take one $\Sigma^1_{m_1,m_2}$ with $m_1,m_2 \in \{1, \dots, n\}$. For convenience of notation
we choose the case $m_1=m_2=:m$. Recall that
$$
\Sigma^1_{m,m}
=\sum_{K}  \sum_{\substack{R_1, \ldots, R_{n+1} \\ R_j^{(k)} = K }}
b_{K,(R_j)} \prod_{j=1}^{m-1} \langle f_j \rangle_K \langle P_{K,(k_1-1,k_2-1)} f_m \rangle_{R_m} 
\prod_{j=m+1}^n \langle f_j \rangle_{R_j} \langle f_{n+1}, h_{R_{n+1}} \rangle.
$$
Expanding
$$
\langle P_{K,(k_1-1,k_2-1)} f_m \rangle_{R_m}
= \sum_{i_1=0}^{k_1-1} \sum_{i_2=0}^{k_2-1}\sum_{L^{(i_1,i_2)}=K}
\langle f_m , h_L \rangle \langle h_L \rangle_{R_m}
$$
there holds that
\begin{equation*}
\begin{split}
\Sigma^1_{m,m}
=\sum_{i_1=0}^{k_1-1} \sum_{i_2=0}^{k_2-1} \sum_{K}  
\sum_{L^{(i_1,i_2)}=K}
&\sum_{\substack{R_{m+1}, \ldots, R_{n+1} \\ R_j^{(k)} = K }}
\Big(\sum_{\substack{R_1, \ldots, R_{m-1} \\ R_j^{(k)} = K }}\sum_{\substack{R_m \subset L \\ R_m^{(k)}=K}} 
\frac{b_{K,(R_j)} \langle h_{L} \rangle_{R_m} }
{|K|^{(m-1)/2} |R_{n+1}|^{(n-m)/2} }\Big) \\
&\prod_{j=1}^{m-1} \langle f_j, h^0_K \rangle 
\langle f_m , h_{L} \rangle
\prod_{j=m+1}^n \langle f_j, h_{R_j}^0 \rangle \langle f_{n+1}, h_{R_{n+1}} \rangle.
\end{split}
\end{equation*}
Since
$$
\Big|\sum_{\substack{R_1, \ldots, R_{m-1} \\ R_j^{(k)} = K }}\sum_{\substack{ R_m \subset L \\ R_m^{(k)}=K}} 
\frac{b_{K,(R_j)} \langle h_{L} \rangle_{R_m} }
{|K|^{(m-1)/2} |R_{n+1}|^{(n-m)/2} }\Big|
\le \frac{|K|^{(m-1)/2} |L|^{1/2} |R_{n+1}|^{(n-m+1)/2}}{|K|^n},
$$
we see that
$$
\Sigma^1_{m,m}
=\sum_{i_1=0}^{k_1-1} \sum_{i_2=0}^{k_2-1}
\langle S_{(0, \dots, 0,(i_1,i_2), k, \dots, k)}(f_1, \dots, f_n),f_{n+1} \rangle,
$$
where $S_{(0, \dots, 0,(i_1,i_2), k, \dots, k)}$ is a standard $n$-linear bi-parameter shift. The case
of general $m_1, m_2$ is analogous.

We turn to the terms $\Sigma^1_{n+1,m_2}-\Sigma^2_{m_2}$.
The terms $\Sigma^1_{m_1,n+1}-\Sigma^3_{m_1}$ are symmetrical.
Let $m_2 \in \{1, \dots, n\}$. After expanding $P^2_{K^2,k_2-1}$ in the slot $m_2$  we have that
$\Sigma^1_{n+1,m_2}-\Sigma^2_{m_2}$ can be written as
\begin{equation*}
\begin{split}
&\sum_{i_2=0}^{k_2-1} \sum_{K} \sum_{(L^2)^{(i_2)}=K^2}  
\sum_{\substack{R_1, \ldots, R_{n+1} \\ R_j^{(k)} = K }} 
b_{K,(R_j)} \langle h_{L^2} \rangle_{I^2_{m_2}}
\Big[\prod_{j=1}^{m_2-1} \langle f_j \rangle_{K} 
\Big\langle f_{m_2}, \frac{1_{K^1}}{|K_1|} \otimes h_{L^2} \Big\rangle
\prod_{j=m_2+1}^n \langle f_j \rangle_{K^1 \times I^2_j}  \\ 
&\hspace{2cm} -\prod_{j=1}^{m_2-1} \langle f_j \rangle_{I^1_{n+1} \times K^2} 
\Big\langle f_{m_2}, \frac{1_{I^1_{n+1}}}{|I^1_{n+1}|} \otimes h_{L^2} \Big\rangle
\prod_{j=m_2+1}^n \langle f_j \rangle_{I^1_{n+1} \times I^2_j} \Big] \langle f_{n+1}, h_{R_{n+1}} \rangle.
\end{split}
\end{equation*}
This splits the difference $\Sigma^1_{n+1,m_2}-\Sigma^2_{m_2}$ as
$$
\Sigma^1_{n+1,m_2}-\Sigma^2_{m_2}
=:\sum_{i_2=0}^{k_2-1} \Sigma^{1,2}_{m_2,i_2}.
$$
We fix one $i_2$ at this point.

Now, we do a splitting as in \eqref{eq:MultilinCollapse} and \eqref{eq:nCollapseRelated} 
with respect to the first parameter 
for the term inside the brackets $[ \ \cdot \ ]$ above. Let $g_j^{m_2}:=g_j= \langle f_j \rangle^2_{K^2}$ for 
$j \in \{1, \dots, m_2-1\}$, $g_{m_2}^{m_2}:=g_{m_2}= \langle f_{m_2}, h_{L^2} \rangle_2 $
and $g_j^{m_2}:=g_j= \langle f_j \rangle^2_{I^2_j}$ for $j \in \{m_2+1, \dots, n\}$. 
Using this notation we have that the term inside the brackets is 
$
\prod_{j=1}^n \langle g_j \rangle_{K^1}-\prod_{j=1}^n \langle g_j \rangle_{I^1_{n+1}}.
$
As in \eqref{eq:MultilinCollapse} we write that
\begin{equation*}
\prod_{j=1}^n \langle g_j \rangle_{K^1}-\prod_{j=1}^n \langle g_j \rangle_{I^1_{n+1}}
=-\sum_{i_1=0}^{k_1-1}\Big(\prod_{j=1}^n \langle g_j \rangle_{(I^1_{n+1})^{(i_1)}}-\prod_{j=1}^n \langle g_j \rangle_{(I^1_{n+1})^{(i_1+1)}}\Big).
\end{equation*}
Then, as in \eqref{eq:nCollapseRelated}, 
we write $\prod_{j=1}^n \langle g_j \rangle_{(I^1_{n+1})^{(i_1)}}-\prod_{j=1}^n \langle g_j \rangle_{(I^1_{n+1})^{(i_1+1)}}$
as the sum
\begin{equation*}
\sum_{m_1=1}^n
\prod_{j=1}^{m_1-1} \langle g_j \rangle_{(I^1_{n+1})^{(i_1+1)}} 
\langle \Delta_{(I^1_{n+1})^{(i_1+1)}} g_{m_1}\rangle_{I^1_{n+1}}  
\prod_{j=m_1+1}^n \langle g_j \rangle_{(I^1_{n+1})^{(i_1)}}.
\end{equation*}
Expanding 
$$
\langle \Delta_{(I^1_{n+1})^{(i_1+1)}} g_{m_1}\rangle_{I^1_{n+1}}
=\langle  g_{m_1}, h_{(I^1_{n+1})^{(i_1+1)}}  \rangle \langle h_{(I^1_{n+1})^{(i_1+1)}} \rangle_{I^1_{n+1}}
$$
we get that
$\prod_{j=1}^n \langle g_j \rangle_{K^1}-\prod_{j=1}^n \langle g_j \rangle_{I^1_{n+1}}$ equals
\begin{equation*}
-\sum_{i_1=0}^{k_1-1}
\sum_{m_1=1}^n 
\prod_{j=1}^{m_1-1} \langle g_j \rangle_{(I^1_{n+1})^{(i_1+1)}} 
\langle  g_{m_1}, h_{(I^1_{n+1})^{(i_1+1)}}  \rangle \langle h_{(I^1_{n+1})^{(i_1+1)}} \rangle_{I^1_{n+1}}  
\prod_{j=m_1+1}^n \langle g_j \rangle_{(I^1_{n+1})^{(i_1)}}.
\end{equation*}
This identity splits $\Sigma^{1,2}_{m_2,i_2}$ further as 
$\Sigma^{1,2}_{m_2,i_2}
=: -\sum_{i_1=0}^{k_1-1} \sum_{m_1=1}^n \Sigma^{1,2}_{m_1,m_2,i_1,i_2}$.

We fix some $m_1$ and $i_1$ and consider the corresponding term. For convenience of notation we 
look at the case $m_1=m_2=:m$. There holds that
\begin{equation*}
\begin{split}
&\Sigma^{1,2}_{m,m,i_1,i_2}
=\sum_{K} \sum_{(L^2)^{(i_2)}=K^2}  
\sum_{\substack{R_1, \ldots, R_{n+1} \\ R_j^{(k)} = K }} 
b_{K,(R_j)} \langle h_{(I^1_{n+1})^{(i_1+1)} \times L^2} \rangle_{I^1_{n+1}\times I^2_{m}} \\
&\prod_{j=1}^{m-1} \langle f_j \rangle_{(I^1_{n+1})^{(i_1+1)} \times K^2} 
\Big\langle f_{m}, h_{(I^1_{n+1})^{(i_1+1)} \times L^2} \Big\rangle
\prod_{j=m+1}^n \langle f_j \rangle_{(I^1_{n+1})^{(i_1)} \times I^2_j} \langle f_{n+1}, h_{R_{n+1}} \rangle.
\end{split}
\end{equation*}
This is seen as a standard shift once we reorganize the summation and verify the normalization.
We take $(I^1_{n+1})^{(i_1+1)}$ as the new ``top cube'' in the first parameter 
($(I^1_{n+1})^{(i_1+1)}$ corresponds to $(L^1)^{(1)}$ in the summation below). There holds that
$
\Sigma^{1,2}_{m,m,i_1,i_2}
$
equals
\begin{equation*}
\begin{split}
 \sum_{K^1}\sum_{(L^1)^{(k_1-i_1)}=K^1} &\sum_{(I_{n+1}^1)^{(i_1)}=L^1}
\sum_{K^2} \sum_{(L^2)^{(i_2)}=K^2} \sum_{\substack{I^2_{m+1}, \dots, I^2_{n+1} \\ (I^2_j)^{(k_2)}=K^2}} 
 c_{K^1,L^1,I^1_{n+1}, K^2, L^2, I^2_{m+1}, \dots, I^2_{n+1}}  \\
&\prod_{j=1}^{m-1} \langle f_j \rangle_{(L^1)^{(1)} \times K^2} 
\Big\langle f_{m}, h_{(L^1)^{(1)} \times L^2} \Big\rangle
\prod_{j=m+1}^n \langle f_j \rangle_{L^1 \times I^2_j} \langle f_{n+1}, h_{R_{n+1}} \rangle,
\end{split}
\end{equation*}
where 
\begin{equation*}
\begin{split}
&c_{K^1,L^1,I^1_{n+1}, K^2, L^2, I^2_{m+1}, \dots, I^2_{n+1}} \\
&= \sum_{\substack{I^1_1, \dots, I^1_{n} \\ (I^1_j)^{(k_1)}=K^1}}
\sum_{\substack{I^2_{1}, \dots, I^2_{m-1} \\ (I^2_j)^{(k_2)}=K^2}}
\sum_{\substack{ I^2_{m} \subset L^2 \\ (I^2_{m})^{(k_2)}=K^2 }} 
b_{K, (R_j)}\langle h_{(L^1)^{(1)} \times L^2} \rangle_{I^1_{n+1}\times I^2_{m}}.
\end{split}
\end{equation*}
We have the estimate
\begin{equation*}
\begin{split}
|c_{K^1,L^1,I^1_{n+1}, K^2, L^2, I^2_{m+1}, \dots, I^2_{n+1}}| 
 & \le  \frac{|(L^1)^{(1)}|^{n/2}|I^1_{n+1}|^{1/2}}{|(L^1)^{(1)}|^n}
\frac{|K^2|^{(m-1)/2} |L^2|^{1/2} |I^2_{n+1}|^{(n-m+1)/2}}{|K^2|^n} \\
&\times|(L^1)^{(1)}|^{(n-1)/2} |K^2|^{(m-1)/2} |I^2|^{(n-m)/2}.
\end{split}
\end{equation*}
Notice that the term in the first line in the right hand side is $2^{d_1(n-m)/2}$ times the right normalization of the shift, 
since in $\Sigma^{1,2}_{m,m,i_1,i_2}$ we have the cubes $L^1$  related to $f_j$ with $j \in \{m+1, \dots, n\}$.
Also, the term in the second line is almost cancelled out when one changes the averages in $\Sigma^{1,2}_{m,m,i_1,i_2}$
into pairings against non-cancellative Haar functions. 

We conclude that for some $C \ge 1$ we have
$$
C^{-1}\Sigma^{1,2}_{m,m,i_1,i_2}
=\langle S_{(0, \dots, 0,(0,i_2), (1,k_2), \dots, (1,k_2), (i_1+1,k_2)}(f_1, \dots, f_n),f_{n+1} \rangle,
$$ 
where $S$ is a standard $n$-linear bi-parameter shift of the given complexity. The case of general $m_1, m_2$ is analogous.

Finally, we look at the term $\Sigma^1_{n+1,n+1}-\Sigma^2_{n+1}-\Sigma^3_{n+1}+\Sigma^4$ which by definition is
\begin{equation}\label{eq:FourTerms}
\begin{split}
&\sum_{K}  \sum_{\substack{R_1, \ldots, R_{n+1} \\ R_j^{(k)} = K }} b_{K,(R_j)} \\
&\Big[ \prod_{j=1}^n \langle f_j \rangle_K
-\prod_{j=1}^n \langle f_j \rangle_{I^1_{n+1} \times K^2}
-\prod_{j=1}^n \langle f_j \rangle_{K^1 \times I^2_{n+1}}
+ \prod_{j=1}^n \langle f_j \rangle_{R_{n+1}} \Big] \langle f_{n+1}, h_{R_{n+1}} \rangle.
\end{split}
\end{equation} 
We will perform bi-parameter versions of the steps \eqref{eq:MultilinCollapse} and 
\eqref{eq:nCollapseRelated}.

Consider the rectangles $K, R_1, \dots, R_{n+1}$ as fixed for the moment. Expanding with respect to the first 
parameter using \eqref{eq:MultilinCollapse}  and \eqref{eq:nCollapseRelated} there holds that
$\langle f_j \rangle_K-\prod_{j=1}^n \langle f_j \rangle_{I^1_{n+1} \times K^2}$ equals
\begin{equation}\label{eq:1Half}
\begin{split}
-\sum_{i_1=0}^{k_1-1}
 &\sum_{m_1=1}^n 
\langle h_{(I^1_{n+1})^{(i_1+1)}}\rangle_{I^1_{n+1}} \\
& \prod_{j=1}^{m_1-1} \langle f_j \rangle_{(I^1_{n+1})^{(i_1+1)}\times K^2} 
\Big \langle  f_{m_1}, h_{(I^1_{n+1})^{(i_1+1)}} \otimes \frac{1_{K^2}}{|K^2|} \Big \rangle  
\prod_{j=m_1+1}^n \langle f_j \rangle_{(I^1_{n+1})^{(i_1)} \times K^2}.
\end{split}
\end{equation}
Similarly, we have that $-\prod_{j=1}^n \langle f_j \rangle_{K^1 \times I^2_{n+1}}
+ \prod_{j=1}^n \langle f_j \rangle_{R_{n+1}}$ equals
\begin{equation}\label{eq:2Half}
\begin{split}
\sum_{i_1=0}^{k_1-1}
 &\sum_{m_1=1}^n 
\langle h_{(I^1_{n+1})^{(i_1+1)}}\rangle_{I^1_{n+1}} \\
& \prod_{j=1}^{m_1-1} \langle f_j \rangle_{(I^1_{n+1})^{(i_1+1)}\times I^2_{n+1}} 
\Big \langle  f_{m_1}, h_{(I^1_{n+1})^{(i_1+1)}} \otimes \frac{1_{I^2_{n+1}}}{|I^2_{n+1}|} \Big \rangle  
\prod_{j=m_1+1}^n \langle f_j \rangle_{(I^1_{n+1})^{(i_1)} \times I^2_{n+1}}.
\end{split}
\end{equation}
Let $g^{m_1,i_1}_j= \langle f_j \rangle^1_{(I^1_{n+1})^{(i_1+1)}}$ for $j \in \{1, \dots, m_1-1\}$,
$g^{m_1,i_1}_{m_1}= \langle f_{m_1}, h_{(I^1_{n+1})^{(i_1+1)}} \rangle_1$  and
$g^{m_1,i_1}_j= \langle f_j \rangle^1_{(I^1_{n+1})^{(i_1)}}$ for $j \in \{m_1+1, \dots,n\}$.
The sum of \eqref{eq:1Half} and \eqref{eq:2Half} can similarly be split as
\begin{equation}\label{eq:BiParnCollapse}
\begin{split}
\sum_{i_1=0}^{k_1-1} \sum_{i_2=0}^{k_2-1}
 &\sum_{m_1,m_2=1}^n 
\langle h_{R_{n+1}^{(i_1+1,i_2+1)} }\rangle_{R_{n+1}} \\
& \prod_{j=1}^{m_2-1} \langle g^{m_1,i_1}_j \rangle_{(I^2_{n+1})^{(i_2+1)}} 
\langle  g^{m_1,i_1}_{m_2}, h_{(I^2_{n+1})^{(i_2+1)}} \rangle  
\prod_{j=m_2+1}^n \langle g^{m_1,i_1}_j \rangle_{(I^2_{n+1})^{(i_2)}}.
\end{split}
\end{equation}
When one recalls the definition of the functions $g_j^{m_1,i_1}$
and writes this in terms of the functions $f_j$, one has that in the first parameter
$f_j$ is paired with $1_{(I_{n+1}^1)^{(i_1+1)}}/|(I_{n+1}^1)^{(i_1+1)}|$ for $j=1, \dots, m_1-1$, 
$f_{m_1}$ with $h_{(I^1_{n+1})^{(i_1+1)}}$ and $f_j$ with 
$1_{(I_{n+1}^1)^{(i_1)}}/|(I_{n+1}^1)^{(i_1)}|$ for $j=m_1+1, \dots, n$. Each $f_j$ is paired similarly in the second parameter.
In the case $m_1=m_2=:m$ the summand in \eqref{eq:BiParnCollapse} can be written as
\begin{equation}\label{eq:OneFromIV}
\langle h_{R_{n+1}^{(i_1+1,i_2+1)} }\rangle_{R_{n+1}}\prod_{j=1}^{m-1} \langle f_j \rangle_{R_{n+1}^{(i_1+1,i_2+1)} }
\langle f_{m_1}, h_{R_{n+1}^{(i_1+1,i_2+1)}} \rangle 
\prod_{j=m+1}^n \langle f_j \rangle_{R_{n+1}^{(i_1,i_2)}}.
\end{equation}

The splitting in \eqref{eq:BiParnCollapse} gives us the identity
$$
\Sigma^1_{n+1,n+1}-\Sigma^2_{n+1}-\Sigma^3_{n+1}+\Sigma^4
=: \sum_{i_1=0}^{k_1-1} \sum_{i_2=0}^{k_2-1}
 \sum_{m_1,m_2=1}^n \Sigma^{1,2,3,4}_{m_1,m_2,i_1,i_2}.
$$
We fix some $i_1$ and $i_2$ and consider the case $m_1=m_2=:m$. 
From \eqref{eq:OneFromIV} we see that
\begin{equation*}
\begin{split}
\Sigma^{1,2,3,4}_{m,m,i_1,i_2}
&= \sum_{K} \sum_{L^{(k_1-i_1, k_2-i_2)}=K}
\sum_{R_{n+1}^{(i_1,i_2)}=L} c_{K, L,R_{n+1}}\\
 &\prod_{j=1}^{m-1} \langle f_j \rangle_{L^{(1,1)}}
\langle f_{m_1}, h_{L^{(1,1)}} \rangle 
\prod_{j=m+1}^n \langle f_j \rangle_{L}
\langle f_{n+1}, h_{R_{n+1}} \rangle,
\end{split}
\end{equation*}
where
$$
c_{K,L,R_{n+1}}
= \sum_{\substack{R_1, \dots, R_n \\ R_j^{(k)}=K}} 
b_{K,(R_j)} \langle h_{L^{(1,1)}}\rangle_{R_{n+1}}.
$$
The coefficient satisfies the estimate
\begin{equation*}
|c_{K,L,R_{n+1}}| \le  \frac{|R_{n+1}|^{1/2}}{|L^{(1,1)}|^{1/2}}=\frac{|L^{(1,1)}|^{n/2}|R_{n+1}|^{1/2}}{|L^{(1,1)}|^{n}} |L^{(1,1)}|^{(n-1)/2}.
\end{equation*}
Thus, we see that $C^{-1}\Sigma^{1,2,3,4}_{m,m,i_1,i_2}$ is a standard $n$-linear bi-parameter shift. 
The complexity of the shift is $((0,0), \dots, (0,0),(1,1),\dots,(1,1), (i_1+1,i_2+1))$ with $m$ zeros.
The case of general $m_1$ and $m_2$ is analogous.
\end{proof}

\subsection*{Estimates for model operators}
We would, ideally, like to prove that the \emph{modified} operators are weighted bounded with a bound that depends on the square root of the complexity. A weighted bound
even with some fixed exponents always yields the full multilinear range of boundedness via extrapolation.
However, due to technical problems we cannot achieve these weighted bounds currently. In fact, it seems to be a problem to achieve the boundedness of the modified model operators in the full multilinear range
with a bound that depends on the square root of the complexity. Another strategy would be to prove weak type estimates in the quasi Banach range in addition to the Banach range boundedness,
and then interpolate. This strategy also has some obstacles that we cannot overcome currently.
 In \cite{LMV} we successfully used both of these strategies with the standard operators in the bilinear bi-parameter setting.

With multilinear bi-parameter \emph{modified} model operators, due to the above reasons,  
we prove only the Banach range boundedness with the square root complexity dependence. 
This, together with estimates for standard operators, will lead to Banach range boundedness of $n$-linear bi-parameter 
$(\omega_1, \omega_2)$-CZOs with  
$\omega_i \in \operatorname{Dini}_{1/2}$ (Corollary \ref{cor:Dini12nlinbipar}).

Then, we prove weighted estimates for standard $n$-linear bi-parameter model operators. 
Via the decomposition of modified operators into standard operators we get the weighted boundedness of modified model operators with a linear dependence on the complexity. This
in turn leads to weighted boundedness of $n$-linear bi-parameter $(\omega_1, \omega_2)$-CZOs with  $\omega_i \in \operatorname{Dini}_{1}$ (Corollary \ref{cor:Dini1nlinbipar}).
These estimates extrapolate to the full range.

Finally, we prove the weighted boundedness of \emph{linear} modified bi-parameter shifts with a bound that depends on
the square root of the complexity. This demonstrates why the linear weighted estimates are easier. These estimates  then transfer to the weighted 
boundedness of linear bi-parameter 
$(\omega_1, \omega_2)$-CZOs with  $\omega_i \in \operatorname{Dini}_{1/2}$ (Equation \eqref{eq:Dini12biparw}).

\subsubsection*{Unweighted boundedness of modified model operators in the Banach range}

\begin{prop}\label{prop:QkBRange}
Let $p_j \in (1, \infty)$, $j=1, \dots,n+1$, be such that $\sum_{j=1}^{n+1} 1/p_j=1$.
Suppose that $Q_k$ is a modified $n$-linear bi-parameter shift. 
Then the estimate
$$
|\langle Q_k(f_1, \ldots, f_n), f_{n+1} \rangle|
 \lesssim \sqrt k_1 \sqrt k_2 \prod_{j=1}^{n+1} \|f_j\|_{L^{p_j}}
$$
holds. 

Suppose that $(QS)_{k,i}$ is a modified/standard shift (here $k \in \{1,2, \dots \}$ and $i=(i_1, \dots, i_{n+1})$). Then the estimate
$$
|\langle (QS)_{k,i}(f_1, \ldots, f_n), f_{n+1} \rangle|
 \lesssim \sqrt k \prod_{j=1}^{n+1} \|f_j\|_{L^{p_j}}
$$
holds. 
\end{prop}

\begin{proof}
We only prove the statement for the operator $Q_k$. This essentially contains the proof for $(QS)_{k,i}$.

We assume $Q_k$ has the explicit form
\begin{align*}
\langle Q_k(f_1, \ldots, f_n), f_{n+1}\rangle = \sum_{K} \sum_{\substack{R_1, \ldots, R_{n+1} \\ R_j^{(k)} = K }}&
a_{K, (R_j)} \Big[ \prod_{j=1}^{n} \langle f_j, h_{R_j}^0 \rangle - \prod_{j=1}^{n} \langle f_j, h_{I_{n+1}^1 \times I_j^2}^0 \rangle \\
&- \prod_{j=1}^{n} \langle f_j, h_{I_j^1 \times I_{n+1}^2}^0 \rangle + \prod_{j=1}^{n} \langle f_j, h_{R_{n+1}}^0 \rangle
 \Big] \langle f_{n+1}, h_{R_{n+1}} \rangle.
\end{align*}
Using the notation \eqref{eq:Dnot} (and recalling \eqref{eq:EPSplitting}) there holds that
\begin{equation*}
\prod_{j=1}^{n} \langle f_j, h_{R_j}^0 \rangle
=\sum_{m_1,m_2=1}^{n+1} \prod_{j=1}^{n} \langle D^1_{K^1,k_1}(j,m_1)D^2_{K^2,k_2}(j,m_2)f_j, h_{R_j}^0 \rangle.
\end{equation*}
We do the same decomposition with the other three terms inside the bracket $[ \, \cdot \,]$. This splits $[\, \cdot \, ]$
into a sum over $m_1,m_2 \in \{1, \dots, n+1\}$.
Then, we notice that all the terms in the sum with $m_1=n+1$ or $m_2=n+1$ cancel out. 
Thus, we get a  splitting of $\langle Q_k(f_1, \ldots, f_n), f_{n+1} \rangle$ into a sum over $m_1,m_2 \in \{1,  \dots, n\}$.
All the terms with different $m_1$ and $m_2$ are estimated separately.

In what follows -- for notational convenience -- we will focus on the case $m_1 = m_2 =: m \in \{1, \ldots, n\}$, and we define
$D^1_{K^1,k_1}(j,m)D^2_{K^2,k_2}(j,m) =: D_{K, k}(j,m)$.  The term in the splitting of 
$\langle Q_k(f_1, \ldots, f_n), f_{n+1} \rangle$  corresponding to $m=m_1=m_2$
can be written as the sum 
$$
\sum_{i=1}^4 \langle U_i(f_1, \ldots, f_n), f_{n+1} \rangle,
$$
where 
$$
\langle U_1(f_1, \ldots, f_n), f_{n+1}\rangle = \sum_{K} \sum_{\substack{R_1, \ldots, R_{n+1} \\ R_j^{(k)} = K }} a_{K, (R_j)} 
\prod_{j=1}^{n} \langle D_{K, k}(j,m)f_j, h_{R_j}^0 \rangle 
 \langle f_{n+1}, h_{R_{n+1}} \rangle,
$$
and $U_2$, $U_3$ and $U_4$ are defined similarly just by replacing $h^0_{R_j}$, $j \in \{1, \dots, n\}$, by
$h_{I_{n+1}^1 \times I_j^2}^0$, $h_{I_{j}^1 \times I_{n+1}^2}^0$ and $h_{R_{n+1}}^0$, respectively.

With some direct calculations it can be shown that for all $i \in \{1, \ldots, 4\}$ we have
\begin{equation}\label{eq:ShiftUpperB}
|\langle U_i(f_1, \ldots, f_n), f_{n+1}\rangle| 
\le \int \prod_{\substack{ j=1 \\ j \ne m}}^n Mf_j \Big( \sum_K |MP_{K, (k_1-1, k_2-1)} f_m|^2 \Big)^{1/2}
S_{\calD} f_{n+1},
\end{equation}
see \eqref{eq:StaCom1}, \eqref{eq:StaCom2} and the two estimates preceeding \eqref{eq:StaCom2} for the corresponding steps in the one-parameter case. From here the estimate 
can be finished by H\"older's inequality, the Fefferman--Stein inequality and square function estimates, 
see Lemma \ref{lem:standardEst}.
\end{proof}

Next, we look at the modified partial paraproducts. We will use the well known one-parameter $H^1$-$\BMO$ duality estimate
\begin{equation}\label{eq:1ParH1BMO}
\sum_I |a_I b_I| 
\lesssim \| (a_I) \|_{\BMO}
\Big \| \Big( \sum_{I} |b_I|^2 \frac{1_I}{|I|} \Big)^{1/2} \Big\|_{L^1},
\end{equation}
where the cubes $I$ are in some dyadic grid. 

\begin{prop}\label{prop:QpiBRange}
Let $p_j \in (1, \infty)$, $j=1, \dots,n+1$, be such that $\sum_{j=1}^{n+1} 1/p_j=1$.
Suppose $(Q\pi)_k$ is a modified $n$-linear partial paraproduct.
Then the estimate
$$
|\langle (Q\pi)_k(f_1, \ldots, f_n), f_{n+1} \rangle|
 \lesssim \sqrt k \prod_{j=1}^{n+1} \|f_j\|_{L^{p_j}}
$$
holds.
\end{prop}

\begin{proof}
We assume that $\langle (Q\pi)_k(f_1, \ldots, f_n), f_{n+1}\rangle$ has the form
\begin{equation*}
\begin{split}
\sum_{K} \sum_{\substack{I^1_1, \ldots, I^1_{n+1} \\ (I^1_j)^{(k)} = K^1 }}
a_{K, (I^1_j)} \Big[ \prod_{j=1}^{n} \Big \langle f_j,  h^0_{I^1_j} \otimes \frac{1_{K^2}}{|K^2|} \Big \rangle 
- \prod_{j=1}^{n} \Big \langle f_j, h^0_{I_{n+1}^1} \otimes \frac{1_{K^2}}{|K^2|} \Big\rangle 
 \Big] \langle f_{n+1}, h_{I^1_{n+1} \times K^2} \rangle.
\end{split}
\end{equation*}
We decompose
$$
\prod_{j=1}^{n} \Big \langle f_j,  h^0_{I^1_j} \otimes \frac{1_{K^2}}{|K^2|} \Big \rangle
=\sum_{m=1}^{n+1}
\prod_{j=1}^{n} \Big \langle D^1_{K^1, k_1}(j,m) f_j,  h^0_{I^1_j} \otimes \frac{1_{K^2}}{|K^2|} \Big \rangle
$$
and similarly with the other term inside the bracket $[ \, \cdot \, ]$. Notice that the terms with $m=n+1$
cancel out. Thus, we get a decomposition of $\langle (Q\pi)_k(f_1, \ldots, f_n), f_{n+1}\rangle$
into a sum over $m \in \{1, \dots, n\}$. The terms with different $m$ are estimated separately.

Fix one $m$. The term from the decomposition of $\langle (Q\pi)_k(f_1, \ldots, f_n), f_{n+1}\rangle$
related to $m$
is
\begin{equation*}
\sum_{i=1}^2\langle U_1(f_1, \ldots, f_n), f_{n+1}\rangle,
\end{equation*}
where $\langle U_1(f_1, \ldots, f_n), f_{n+1}\rangle$ equals
\begin{equation}\label{eq:DefU1}
\sum_{K} \sum_{\substack{I^1_1, \ldots, I^1_{n+1} \\ (I^1_j)^{(k)} = K^1 }}
a_{K, (I^1_j)} \prod_{j=1}^{n} \Big \langle D^1_{K^1, k_1}(j,m) f_j,  h^0_{I^1_j} \otimes \frac{1_{K^2}}{|K^2|} \Big \rangle
\langle f_{n+1}, h_{I^1_{n+1} \times K^2} \rangle
\end{equation}
and $\langle U_2(f_1, \ldots, f_n), f_{n+1}\rangle$ is defined similarly just be replacing $h^0_{I^1_j}$, $j=1, \dots, n$, 
with $h^0_{I^1_{n+1}}$.

We consider $U_1$ first. From the one-parameter $H^1$-$\BMO$ duality estimate \eqref{eq:1ParH1BMO} we have that,
with fixed $K^1$ and $I^1_1, \dots, I^1_{n+1}$, the sum over $K^2$ of the absolute value of the summand in 
\eqref{eq:DefU1} is dominated by
\begin{equation*}
\begin{split}
&\frac{|I^1_{n+1}|^{(n+1)/2}}{|K^1|^n}\int _{\R^{d_2}} 
\Big( \sum_{K^2}  \Big|\prod_{j=1}^{n} \Big \langle D^1_{K^1, k_1}(j,m) f_j,  h^0_{I^1_j} \otimes \frac{1_{K^2}}{|K^2|} \Big \rangle
\langle f_{n+1}, h_{I^1_{n+1}\times K^2} \rangle\Big|^2 \frac{1_{K^2}}{|K^2|} \Big)^{1/2} \\
& \le \frac{|I^1_{n+1}|^{(n+1)/2}}{|K^1|^n} \int _{\R^{d_2}} \prod_{j=1}^{n}  \langle  M^2 D^1_{K^1, k_1}(j,m) f_j,  h^0_{I^1_j} \rangle_1
\langle S^2_{\calD_2} \Delta_{K^1,k_1} f_{n+1} , h^0_{I^1_{n+1}} \rangle_1.
\end{split}
\end{equation*}
The sum of this over $K^1$ and $I^1_1, \dots, I^1_{n+1}$ such that $(I^1_j)^{(k)}=K^1$ is less than
\begin{equation}\label{eq:PPDom}
\int_{\R^d} \prod_{\substack{j=1 \\ j \not=m}}^n M^1 M^2 f_j
\Big(\sum_{K^1} (M^1 M^2 P^1_{K^1, k_1-1}f_m)^2 \Big)^{1/2} 
\Big(\sum_{K^1} (S^2_{\calD^2} \Delta_{K^1,k_1} f_{n+1})^2\Big)^{1/2}.
\end{equation}

Notice that the square function related to $f_{n+1}$ is just the bi-parameter square function $S_\calD$. To finish the estimate it remains to
use the Fefferman--Stein inequality and square function estimates, see  Lemma \ref{lem:standardEst}. 

The second term $|\langle U_2(f_1, \ldots, f_n), f_{n+1}\rangle|$ satisfies the same upper bound
\eqref{eq:PPDom}, and can therefore be estimated in the same way. The proof is concluded.
\end{proof}

\subsubsection*{Weighted estimates of standard model operators}
In \cite{LMV} we proved weighted estimates for standard bilinear bi-parameter shifts and partial paraproduts.
Since here we work in the $n$-linear setting we give the corresponding proofs below. The proofs are similar.
In \cite{LMV} we were not yet able to prove weighted estimates for full paraproducts. The bilinear case was proved later
in \cite{ALMV}, and we again give the $n$-linear version here for reader's convenience. We note that weighted estimates
in the linear case are well known, see for example \cite{HPW}.

\begin{prop}\label{prop:StandShiftWeight}
Let $p_j \in (1, \infty)$ and $1/r:= \sum_{j=1}^n 1/p_j$.
Let $w_j \in A_{p_j}$, $j=1, \dots, n$, be bi-parameter weights, and define
$w=\prod_{j=1}^n w_j^{r/p_j}$.

Suppose $S_i$ is a standard $n$-linear bi-parameter shift. Then we have
$$
\| S_i(f_1, \dots, f_n) \|_{L^r(w)} 
\lesssim \prod_{j=1}^n \| f_j \|_{L^{p_j}(w_j)}.
$$
\end{prop}

\begin{proof}
The first step of the estimate is determined by how many cancellative Haar functions there are in the slot $n+1$.
If there are two, then one uses the bi-parameter lower square function estimate. If there are none, then one just
puts the absolute values inside the summations of the shift. 
If there is one, then one uses the lower square function estimate in the cancellative parameter.

We demonstrate the argument with a shift that has one cancellative Haar function in the slot $n+1$.
Let us assume that $S_i$ has  the form
\begin{equation*}
S_i(f_1, \dots, f_n)
=\sum_{K} \sum_{\substack{R_1, \dots, R_{n+1} \\ R_j^{(i_j)}=K}}
a_{K,(R_j)} \langle f_1, h_{R_1} \rangle 
\langle f_2, \wt h_{I^1_2} \otimes  h_{I^2_2} \rangle
\prod_{j=3}^n \langle f_j,  \wt h_{R_j} \rangle  \cdot h_{I^1_{n+1}} \otimes \wt h_{I^2_{n+1}}.
\end{equation*}

First, we record that
\begin{equation*}
\begin{split}
\sum_{\substack{R_1, \dots, R_{n+1} \\ R_j^{(i_j)}=K}}
 | a_{K,(R_j)}\langle f_1, h_{R_1} \rangle &
\langle f_2, \wt h_{I^1_2} \otimes  h_{I^2_2} \rangle
\prod_{j=3}^n \langle f_j,  \wt h_{R_j} \rangle   h_{I^1_{n+1}} \otimes \wt h_{I^2_{n+1}} | \\
&\le M  \Delta_{K,i_1} f_1 M \Delta^2_{K^2,i^2_2} f_2 
\prod_{j=3}^n M f_j.
\end{split}
\end{equation*}
According to our usual bi-parameter notational conventions, here $\Delta_{K, i_1} = \Delta_{K^1, i_1^1}^1 \Delta_{K^2, i_1^2}^2$.
Recall that with fixed $x_2 \in \R^{d_2}$ the weight $x_1 \mapsto w(x_1,x_2)$ is uniformly in $A_{2r}(\R^{d_1}) \subset A_\infty(\R^{d_1})$.  
This allows us to use the lower square function estimate to get that
\begin{equation*}
\begin{split}
\| S_i(f_1, \dots, f_n) \|_{L^r(w)} 
&\lesssim \Big \| \sum_{K^1} 
\Big(\sum_{K^2} 
M  \Delta_{K,i_1} f_1 M \Delta^2_{K^2,i^2_2} f_2 
 \Big)^2 \Big)^{1/2}\prod_{j=3}^n M f_j
\Big \|_{L^r(w)} \\
& \le  \Big \| \Big( \sum_{K} 
( M  \Delta_{K,i_1} f_1 )^2 \Big)^{1/2}
\Big( \sum_{K^2} (M \Delta^2_{K^2,i^2_2} f_2)^2 \Big)^{1/2}
 \prod_{j=3}^n  M f_j
\Big \|_{L^r(w)}.
\end{split}
\end{equation*}

Since $w=\prod_{j=1}^n w_j^{r/p_j}$, the last quantity is, by H\"older's inequality, less than
\begin{equation*}
\Big \|\Big( \sum_{K} ( M  \Delta_{K,i_1} f_1 )^2 \Big)^{1/2} \Big\|_{L^{p_1}(w_1)}
\Big\| \Big( \sum_{K^2} (M \Delta^2_{K^2,i^2_2} f_2)^2 \Big)^{1/2}\Big\|_{L^{p_2}(w_2)}
\prod_{j=3}^n \| Mf_j \|_{L^{p_j}(w_j)}.
\end{equation*}
From here the estimate is concluded with the weighted Fefferman--Stein inequality and weighted square function estimates, see
Lemma \ref{lem:standardEst}.
\end{proof}

Next, we turn to standard partial paraproducts. For this 
we recall a certain estimate concerning one-parameter paraproducts.

Suppose that $\pi$ is an $n$-linear one-parameter paraproduct (normalized as in \eqref{eq:BMO}) and that $v$ is a one-parameter $A_\infty$ weight.
It is a consequence of the so-called sparse domination that
$$
\int | \pi(f_1, \dots, f_n)| v
\lesssim \int  M(f_1, \dots, f_n) v,
$$
where $M(f_1, \dots, f_n):= \sup_Q \prod_{j=1}^n \langle |f_j| \rangle_Q1_Q$ is the  $n$-linear dyadic maximal function 
and the supremum is taken over the cubes in the dyadic grid related to $\pi$.
For references and discussion related to this well known fact see \cite{DLMV1}.

The $A_\infty$ extrapolation result \cite[Theorem 2.1]{CUMP} then implies that
$$
\int | \pi(f_1, \dots, f_n)|^p v
\lesssim \int M(f_1, \dots, f_n)^p v
$$
for all $p \in (0, \infty)$ and $v \in A_\infty$, which directly gives that
$$
\int \Big( \sum_{k} | \pi_k(f_{1,k}, \dots, f_{n,k})|^p \Big)^{p/p} v
\lesssim \int \Big( \sum_{k}  M(f_{1,k}, \dots, f_{n,k})^p \Big)^{p/p} v,
$$
where we have a family of paraproducts $\pi_k$. Extrapolating once more yields the
estimate
\begin{equation}\label{eq:SpaCon}
\int \Big( \sum_{k} | \pi_k(f_{1,k}, \dots, f_{n,k})|^p \Big)^{q/p} v
\lesssim \int \Big( \sum_{k}  M(f_{1,k}, \dots, f_{n,k})^p \Big)^{q/p} v
\end{equation}
for all $p,q \in (0,\infty)$ and $v \in A_\infty$.

\begin{prop}\label{prop:StandPPWeight}
Let $p_j \in (1, \infty)$ and $1/r:= \sum_{j=1}^n 1/p_j$.
Let $w_j \in A_{p_j}$, $j=1, \dots, n$, be bi-parameter weights, and define
$w=\prod_{j=1}^n w_j^{r/p_j}$.

Suppose $(S\pi)_i$ is a standard $n$-linear partial paraproduct. Then we have
$$
\| (S\pi)_i(f_1, \dots, f_n) \|_{L^r(w)} 
\lesssim \prod_{j=1}^n \| f_j \|_{L^{p_j}(w_j)}.
$$
\end{prop}

\begin{proof}
We begin by considering the case that $(S\pi)_i(f_1, \dots, f_n)$ has the form
\begin{equation*}
\sum_{K^1} \sum_{\substack{ I^1_1, \dots, I^1_{n+1} \\ (I^1_j)^{(i_j)}=K^1}}
h_{I^1_{n+1}} \otimes \pi_{K^1,(I_j^1)} (\langle f_1, h_{I^1_1}  \rangle_1,
 \langle f_2,  \wt h_{I^1_2} \rangle_1,  \dots, \langle f_n,  \wt h_{I^1_n} \rangle_1),
\end{equation*}
where $\pi_{K^1, (I^1_j)}$ is an $n$-linear one-parameter paraproduct, whose collection of coefficients is in $\BMO$ with 
the normalization $\prod_{j=1}^{n+1} |I^1_j|^{1/2} / |K^1|^n$. 
The important point about this form is that the Haar functions related to the cubes $I^1_{n+1}$ are cancellative. 
First, the lower square function estimate in the first parameter
gives that
\begin{equation*}
\begin{split}
\| & (S\pi)_i(f_1, \dots, f_n) \|_{L^r(w)} \\
&\lesssim \Big\| \Big(\sum_{K^1} 
\Big(\sum_{\substack{ I^1_1, \dots, I^1_{n+1} \\ (I^1_j)^{(i_j)}=K^1}}
|h_{I^1_{n+1}} \otimes \pi_{K^1,(I_j^1)} (\langle f_1, h_{I^1_1}  \rangle_1,
\langle f_2,  \wt h_{I^1_2} \rangle_1,  \dots, \langle f_n,  \wt h_{I^1_n} \rangle_1)| \Big)^2 
\Big)^{1/2} \Big\|_{L^r(w)}.
\end{split}
\end{equation*}

In the above norm we can with fixed $x_1$ use \eqref{eq:SpaCon}. This gives that the function inside $( \, \cdot \, )^2$
may be replaced by
\begin{equation}\label{eq:Replacement}
\sum_{\substack{ I^1_1, \dots, I^1_{n+1} \\ (I^1_j)^{(i_j)}=K^1}}
\frac{\prod_{j=1}^{n+1} |I^1_j|^{1/2}}{|K^1|^n}
h^0_{I^1_{n+1}} \otimes  M\langle f_1, h_{I^1_1}  \rangle_1
 \prod_{j=2}^n M \langle f_j,  \wt h_{I^1_j} \rangle_1. 
\end{equation}
The estimates 
$M\langle f_1, h_{I^1_1}  \rangle_1
\le \langle M^2 \Delta_{K^1,k} f_1, h^0_{I^1_1} \rangle_1$
and $M\langle f_j, \wt h_{I^1_j}  \rangle_1
\le \langle M^2  f_j, h^0_{I^1_j} \rangle_1$
give that \eqref{eq:Replacement} is less than
\begin{equation*}
M^1M^2 \Delta_{K^1,k}f_1 \prod_{j=2}^n M^1M^2f_j.
\end{equation*}

In conclusion, we have shown that
\begin{equation*}
\| (S\pi)_i(f_1, \dots, f_n) \|_{L^r(w)}
\lesssim \| \Big( \sum_{K^1} (M^1M^2 \Delta_{K^1,k}f_1)^2 \Big)^{1/2} \prod_{j=2}^n M^1M^2f_j \Big\|_{L^r(w)}.
\end{equation*}
From here the estimate is concluded with H\"older's inequality, weighted Fefferman--Stein inequality and the weighted square function estimate (Lemma \ref{lem:standardEst}).

If $(S\pi)_i(f_1, \dots, f_n)$ is of the form where there are 
non-cancellative Haar functions related to the cubes $I^1_{n+1}$, 
then one can directly use the estimate \eqref{eq:SpaCon}
without first applying the lower square function estimate. After this,
one can proceed with similar estimates as before.
\end{proof}

Now, we consider the full paraproducts.
The following bi-parameter $H^1$-$\BMO$ duality (a weighted version appears in Proposition 4.1 of \cite{HPW})
\begin{equation}\label{eq:BiParH1BMO}
\sum_{R} |a_{R}| |b_{R}| \lesssim \| (a_{R} ) \|_{\BMO_{\operatorname{prod}}} \Big\| \Big( \sum_{R} |b_{R}|^2 \frac{1_{R}}{|R|} \Big)^{1/2} \Big \|_{L^1}
\end{equation}
holds. This key estimate is enough to handle the linear situation $\Pi f$  with a duality based proof. We will need the following generalization
\begin{equation}\label{eq:genH1BMO}
\sum_{R} |a_{R}| \langle w \rangle_R |b_{R}| \lesssim \| (a_{R} ) \|_{\BMO_{\operatorname{prod}}} \Big\| \Big( \sum_{R} |b_{R}|^2 \frac{1_{R}}{|R|} \Big)^{1/2} \Big \|_{L^1(w)}
\end{equation}
whenever $w \in A_{\infty}$ is a bi-parameter weight. This is \cite[Corollary 3.5]{ALMV}.

\begin{prop}\label{prop:FullPWeight}
Let $p_j \in (1, \infty)$ and $1/r:= \sum_{j=1}^n 1/p_j$. Let $w_j \in A_{p_j}$ be bi-parameter weights and define $w := \prod_{j=1}^n w_j^{r/p_j}$.
Then we have
$$
\|\Pi(f_1, \ldots, f_n)\|_{L^r(w)} \lesssim \prod_{j=1}^n \|f_j\|_{L^{p_j}(w_j)}.
$$
\end{prop}
\begin{proof}
\textbf{Case 1.} Suppose that $\Pi$ has a form, where we have $\langle f_{n+1} \rangle_K$. For example, we consider the explicit case that
$$
\Pi(f_1, \ldots, f_n) = \sum_{K = K^1 \times K^2} a_K \prod_{j=1}^{n-2} \langle f_j \rangle_K  \Big \langle f_{n-1}, h_{K^1} \otimes \frac{1_{K^2}}{|K^2|} \Big\rangle
 \Big \langle f_{n}, \frac{1_{K^1}}{|K^1|} \otimes h_{K^2} \Big\rangle \frac{1_K}{|K|}.
$$
Let $p_j \in (1, \infty)$ satisfy $\sum_{j=1}^n 1/p_j = 1$. We now have
\begin{align*}
\| \Pi(f_1, \ldots, f_n)  \|_{L^1(w)} \le \sum_{K = K^1 \times K^2} |a_K| \langle w \rangle_K \prod_{j=1}^{n-2} \langle |f_j| \rangle_K \Big| \Big \langle f_{n-1}, h_{K^1} \otimes \frac{1_{K^2}}{|K^2|} \Big\rangle
 \Big \langle f_{n}, \frac{1_{K^1}}{|K^1|} \otimes h_{K^2} \Big\rangle \Big|.
\end{align*}
Using \eqref{eq:genH1BMO} this can be dominated by
\begin{align*}
\Big\| \Big(&  \sum_{K = K^1 \times K^2} \prod_{j=1}^{n-2} \langle |f_j| \rangle_K^2 \Big| \Big \langle f_{n-1}, h_{K^1} \otimes \frac{1_{K^2}}{|K^2|} \Big\rangle
 \Big \langle f_{n}, \frac{1_{K^1}}{|K^1|} \otimes h_{K^2} \Big\rangle \Big|^2 \frac{1_K}{|K|} \Big)^{1/2} \Big\|_{L^1(w)} \\
 &\le \Big\| \prod_{j=1}^{n-2} M_{\calD_{\sigma}} f_j \cdot S_{\calD_{\sigma}, M}^1 f_{n-1} S_{\calD_{\sigma}, M}^2 f_{n} \Big\|_{L^1(w)} \\
 &\le \prod_{j=1}^{n-2} \|M_{\calD_{\sigma}} f_j\|_{L^{p_j}(w_j)} \cdot \|S_{\calD_{\sigma}, M}^1 f_{n-1}\|_{L^{p_{n-1}}(w_{n-1})}
 \|S_{\calD_{\sigma}, M}^2 f_{n}\|_{L^{p_{n}}(w_{n})} 
 \lesssim  \prod_{j=1}^{n} \| f_j\|_{L^{p_j}(w_j)}.
\end{align*}
To achieve the full range of exponents we extrapolate using \cite{GM}.

\textbf{Case 2.} We are not in Case 1. This means that there is at most one Haar function associated with the functions $f_1, \ldots, f_n$.
The key implication of this is that we have for at least $n-1$ indices $j \in \{1, \ldots, n\}$ the term $\langle f_j \rangle_K$ with a full average.
We consider the explicit case
$$
\Pi(f_1, \ldots, f_n) =  \sum_{K = K^1 \times K^2} a_K \prod_{j=1}^{n-1} \langle f_j \rangle_K  \Big \langle f_{n}, \frac{1_{K^1}}{|K^1|} \otimes h_{K^2} \Big\rangle
h_{K^1} \otimes \frac{1_{K^2}}{|K^2|}.
$$
Using the weighted lower square function estimate in the first parameter, Lemma \ref{lem:LowerFS}, we dominate $\|\Pi(f_1, \ldots, f_n)\|_{L^r(w)}$ with
\begin{align*}
\Big\| \Big(& \sum_{K^1} \Big[ \sum_{K^2} |a_K|  \prod_{j=1}^{n-1} \langle |f_j| \rangle_K  \Big| \Big \langle f_{n}, \frac{1_{K^1}}{|K^1|} \otimes h_{K^2} \Big\rangle \Big| 
\frac{1_{K^2}}{|K^2|} \Big]^2 \frac{1_{K^1}}{|K^1|} \Big)^{1/2} \Big\|_{L^r(w)} \\
&\le \Big\|   \prod_{j=1}^{n-1} M_{\calD_{\sigma}} f_j \Big( \sum_{K^1} \Big[ \sum_{K^2} |a_K|  \Big| \Big \langle f_{n}, \frac{1_{K^1}}{|K^1|} \otimes h_{K^2} \Big\rangle \Big| 
\frac{1_{K^2}}{|K^2|} \Big]^2 \frac{1_{K^1}}{|K^1|} \Big)^{1/2} \Big\|_{L^r(w)}  \\
&\le \prod_{j=1}^{n-1} \|M_{\calD_{\sigma}} f_j\|_{L^{p_j}(w_j)} \Big\| \Big( \sum_{K^1} \Big[ \sum_{K^2} |a_K|  \Big| \Big \langle f_{n}, \frac{1_{K^1}}{|K^1|} \otimes h_{K^2} \Big\rangle \Big| 
\frac{1_{K^2}}{|K^2|} \Big]^2 \frac{1_{K^1}}{|K^1|} \Big)^{1/2} \Big\|_{L^{p_n}(w_n)}.
\end{align*} 
We can, of course, remove the maximal functions. For the last term we notice that by the weighted upper square function estimate we have
\begin{align*}
 \Big\| \Big(& \sum_{K^1} \Big[ \sum_{K^2} |a_K|  \Big| \Big \langle f_{n}, \frac{1_{K^1}}{|K^1|} \otimes h_{K^2} \Big\rangle \Big| 
\frac{1_{K^2}}{|K^2|} \Big]^2 \frac{1_{K^1}}{|K^1|} \Big)^{1/2} \Big\|_{L^{p_n}(w_n)} \\
&\lesssim \Big\|  \sum_{K = K^1 \times K^2} |a_K|  \Big| \Big \langle f_{n}, \frac{1_{K^1}}{|K^1|} \otimes h_{K^2} \Big\rangle \Big| h_{K^1} \otimes \frac{1_{K^2}}{|K^2|}\Big\|_{L^{p_n}(w_n)}.
\end{align*}
Now, this is a linear bi-parameter full paraproduct, and thus bounded. For the reader's convenience, we highlight that this is easy by \emph{duality} (which is the key reason why we wanted to reduce to the linear situation): 
\begin{align*}
 \sum_{K = K^1 \times K^2}& |a_K|  \Big| \Big \langle f_{n}, \frac{1_{K^1}}{|K^1|} \otimes h_{K^2} \Big\rangle \Big| \Big| \Big\langle g, h_{K^1} \otimes \frac{1_{K^2}}{|K^2|} \Big\rangle \Big| \\
 &\lesssim \Big\| \Big( \sum_{K = K^1 \times K^2}  \Big| \Big \langle f_{n}, \frac{1_{K^1}}{|K^1|} \otimes h_{K^2} \Big\rangle \Big|^2 \Big|
 \Big\langle g, h_{K^1} \otimes \frac{1_{K^2}}{|K^2|} \Big\rangle \Big|^2 \frac{1_K}{|K|} \Big)^{1/2} \Big\|_{L^1} \\
 &\le \|S_{\calD_{\sigma}, M}^2 f_n \cdot S_{\calD_{\sigma}, M}^1 g \|_{L^1} = \int  S_{\calD_{\sigma}, M}^2 f_n \cdot S_{\calD_{\sigma}, M}^1 g \cdot w_n^{1/p_n} w_n^{-1/p_n}
  \\ &\le \|S_{\calD_{\sigma}, M}^2 f_n \|_{L^{p_n}(w_n)} 
 \|S_{\calD_{\sigma}, M}^1 g \|_{L^{p_n'}(w_n^{1-p_n'})} \lesssim \| f_n \|_{L^{p_n}(w_n)} 
 \| g \|_{L^{p_n'}(w_n^{1-p_n'})}.
\end{align*} 
\end{proof}

\subsubsection*{Weighted estimates of linear modified model operators}
Here we discuss the weighted estimates of linear modified model operators with a bound depending on the square root of the complexity. 
Notice that in principle we have already done all the necessary work. For example, if we  want to estimate 
$\| Q_k f \|_{L^p(w)}$, we study the unweighted pairings $\langle Q_k f,g \rangle$. Then, 
we proceed as in the linear case of Proposition \ref{prop:QkBRange}. Depending on the form
of the shift this leads us to terms corresponding to \eqref{eq:ShiftUpperB} such as
$$
\int \Big( \sum_K |MP_{K, (k_1-1, k_2-1)} f|^2 \Big)^{1/2}
S_{\calD} g.
$$
By H\"older's inequality this is less than
$$
\Big\|\Big( \sum_K |MP_{K, (k_1-1, k_2-1)} f|^2 \Big)^{1/2} \Big \|_{L^p(w)}
\|S_{\calD} g \|_{L^{p'}(w^{1-p'})}
\lesssim \sqrt k_1 \sqrt k_2 \| f\|_{L^p(w)} \| g \|_{L^{p'}(w^{1-p'})}.
$$

Recall that in the linear case we may use the formalism involving the functions $H_{I,J}$ to present the modified operators
as described in Remark \ref{rem:BiParH}. 
To demonstrate how this formalism works we give the proof of the weighted estimates of modified shifts below. 
Because of the above reason we only state the weighted estimate of modified partial paraproducts but omit the proof.

We point out to avoid confusion, that in the formalism involving the functions $H_{I,J}$ there are also shifts with zero complexity, whereas
in the formalism we use in the general $n$-linear case modified shifts with either $k_1=0$ or $k_2=0$ are automatically zero.
For this reason  we have $\sqrt{k_i+1}$ in the next statement, but on the other hand we had $\sqrt k_i$ in Proposition \ref{prop:QkBRange}.

\begin{prop}\label{prop:LinQkWeight}
For every $p \in (1, \infty)$ and bi-parameter $A_p$ weight $w$ we have
$$
\|Q_{k} f\|_{L^p(w)} \lesssim \sqrt{k_1+1}\sqrt{k_2+1} \|f\|_{L^p(w)}. 
$$
\end{prop}

\begin{proof}
Assume that the shift is e.g. of the form
$$
\langle Q_{k}f, g\rangle 
= \sum_{K} \sum_{\substack{ R_1,R_2 \\  R_j^{(k)}=K}}
a_{K,R_1,R_2} \langle f, h_{I^1_1} \otimes H_{I^2_1,I^2_2} \rangle \langle g, H_{I^1_1, I^1_2} \otimes h_{I^2_2} \rangle.
$$
We write for the moment that $\varphi_{K}f:=|\Delta^1_{K^1,k_1}P^2_{K^2,k_2}f|$ and 
$\phi_{K}g:= | P^1_{K^1,k_1}\Delta^2_{K^2,k_2}g|$. First, recalling \eqref{eq:gAndH} we estimate
\begin{equation*}
\begin{split}
&|\langle Q_{k}f, g\rangle| \\
&\le \sum_{K} \sum_{\substack{ R_1,R_2 \\  R_j^{(k)}=K}}
\frac{|R_1|}{|K|} \langle \varphi_{K}f, h^0_{I_1^1} \otimes (h^0_{I^2_1}+h^0_{I^2_2}) \rangle
 \langle \phi_{K}g, (h^0_{I^1_1}+h^0_{I^1_2}) \otimes h^0_{I^2_2} \rangle.
\end{split}
\end{equation*}
This is split into four terms according to the sums inside the pairings.

For brevity we explicitly demonstrate the estimate only with the term
\begin{equation*}
\begin{split}
\sum_{K} & \sum_{\substack{ R_1,R_2 \\  R_j^{(k)}=K}}
\frac{|R_1|}{|K|} \langle \varphi_{K}f, h^0_{I^1_1} \otimes h^0_{I^2_2} \rangle
\langle \phi_{K}g, h^0_{I^1_1} \otimes h^0_{I^2_2} \rangle,
\end{split}
\end{equation*}
which can be written as
$$
\sum_{K} \langle E_{K, k} \varphi_{K}f, \phi_{K}g \rangle,
$$
where $E_{K,k}:= E_{K^1, k_1}^1 E_{K^2, k_2}^2$.
This is dominated by
$$
\Big \| \Big( \sum_{K} (E_{K,k} |\Delta^1_{K^1,k_1}P^2_{K^2,k_2}f|)^2\Big)^{1/2} \Big\|_{L^p(w)}
\Big \| \Big( \sum_{K} | P^1_{K^1,k_1}\Delta^2_{K^2,k_2}g|^2\Big)^{1/2} \Big\|_{L^{p'}(w^{1-p'})}.
$$
From here the proof is finished by using weighted Stein's inequality (or the maximal function estimate)  and weighted square function estimates.
\end{proof}

\begin{prop}\label{prop:LinQpiWeight}
For every $p \in (1, \infty)$ and bi-parameter $A_p$ weight $w$ we have
$$
\|(Q\pi)_{k} f\|_{L^p(w)}+\|(\pi Q)_{k} f\|_{L^p(w)} \lesssim \sqrt{k+1} \|f\|_{L^p(w)}. 
$$
\end{prop}

\subsection*{Bi-parameter representation theorem}
We set $$
\sigma = (\sigma_1, \sigma_2) \in (\{0,1\}^{d_1})^{\Z} \times (\{0,1\}^{d_2})^{\Z}, \qquad
\sigma_i = (\sigma^k_i)_{k \in \Z},
$$
and denote the expectation over the product probability space by $$\E_{\sigma} = \E_{\sigma_1} \E_{\sigma_2} = \E_{\sigma_2} \E_{\sigma_1} =
\iint \ud \mathbb{P}_{\sigma_1} \ud \mathbb{P}_{\sigma_2}.$$
We also set $\calD_0 = \calD^1_0 \times \calD^2_0$, where $\calD_0^i$ is the standard dyadic grid of $\R^{d_i}$.
As in the one-parameter case we use the notation
$$
I_i + \sigma_i := I_i + \sum_{k:\, 2^{-k} < \ell(I_i)} 2^{-k}\sigma_i^k, \qquad I_i \in \calD_0^i.
$$
 Given $\sigma = (\sigma_1, \sigma_2)$ and $R = I_1 \times I_2 \in \calD_0$ we set
$$
R + \sigma = (I_1+\sigma_1) \times (I_2+\sigma_2) \qquad \textup{and} \qquad \calD_{\sigma} = \{R + \sigma\colon \, R \in \calD_0\} = \calD_{\sigma_1} \times \calD_{\sigma_2}.
$$
\begin{thm}\label{thm:rep2par}
Suppose that $T$ is an $n$-linear bi-parameter $(\omega_1, \omega_2)$-CZO, where $\omega_i \in \operatorname{Dini}_{1/2}$. Then we have
$$
\langle T(f_1,\ldots,f_n), f_{n+1} \rangle= C \E_{\sigma} \sum_{k = (k_1, k_2) \in \N^2} \sum_{u=0}^{c_{d,n}} \omega_1(2^{-k_1})\omega_2(2^{-k_2}) \langle V_{k,u,\sigma}(f_1,\ldots,f_n), f_{n+1} \rangle,
$$
where
\begin{align*}
V_{k,u, \sigma} \in \{Q_k, S_{((k_1, k_2), \ldots, (k_1, k_2))}, &(QS)_{k_1, (k_2, \ldots, k_2)}, (SQ)_{(k_1, \ldots, k_1), k_2}, \\
& (Q\pi)_{k_1}, (\pi Q)_{k_2}, (S\pi)_{(k_1, \ldots k_1)}, (\pi S)_{(k_2, \ldots, k_2)}, \Pi\}
\end{align*}
defined in $\calD_{\sigma}$, and if the operator does not depend on $k_1$ or $k_2$ then that particular $k_i = 0$.
\end{thm}
\begin{rem}
We do not write the dependence of the constant $C$ on the various kernel and $T1$ assumptions as explicitly as in the one-parameter case, but the
dependence is analogous.
\end{rem}
\begin{proof}
We decompose
\begin{equation*}
\begin{split}
\langle T(f_1, \ldots, f_n),f_{n+1} \rangle
&= \E_{\sigma} \sum_{R_1, \ldots, R_{n+1} } \langle T(\Delta_{R_1}f_1, \ldots, \Delta_{R_n}f_n),\Delta_{R_{n+1}}f_{n+1} \rangle \\
&= \sum_{j_1, j_2 =1}^{n+1}  \E_{\sigma} \sum_{ \substack{ R_1, \ldots, R_{n+1} \\ \ell(I_{i_1}^1) > \ell(I_{j_1}^1) \textup{ for } i_1 \ne j_1
 \\ \ell(I_{i_2}^2) > \ell(I_{j_2}^2) \textup{ for } i_2 \ne j_2}} \langle T(\Delta_{R_1}f_1, \ldots, \Delta_{R_n}f_n),\Delta_{R_{n+1}}f_{n+1} \rangle \\
& \hspace{4cm}+  \E_{\sigma} \operatorname{Rem}_{\sigma},
\end{split}
\end{equation*}
where $R_1, \ldots, R_{n+1} \in \calD_\sigma = \calD_{\sigma_1} \times \calD_{\sigma_2}$ for some $\sigma =  (\sigma_1, \sigma_2)$
and $R_j = I_j^1 \times I_j^2$.
\subsubsection*{The main terms}
For $j_1, j_2$ we let
$$
\Sigma_{j_1, j_2, \sigma} = \sum_{ \substack{ R_1, \ldots, R_{n+1} \\ \ell(I_{i_1}^1) > \ell(I_{j_1}^1) \textup{ for } i_1 \ne j_1
 \\ \ell(I_{i_2}^2) > \ell(I_{j_2}^2) \textup{ for } i_2 \ne j_2}} \langle T(\Delta_{R_1}f_1, \ldots, \Delta_{R_n}f_n),\Delta_{R_{n+1}}f_{n+1} \rangle.
$$
These are symmetric and we choose to deal with $\Sigma_{\sigma} := \Sigma_{n, n+1, \sigma}$. After collapsing the relevant sums using \eqref{eq:collapse} we have
$$
\Sigma_{\sigma} =  \sum_{ \substack{ R_1, \ldots, R_{n+1} \\  \ell(R_1) = \cdots = \ell(R_{n+1}) }}
\langle T(E_{R_1} f_1, \ldots, E_{R_{n-1}} f_{n-1}, \Delta_{I_n^1}^1 E_{I_n^2}^2 f_n),  E_{I_{n+1}^1}^1 \Delta_{I_{n+1}^2}^2 f_{n+1} \rangle,
$$
where $\ell(R_j) := ( \ell(I_j^1), \ell(I_j^2))$ for $R_j = I_j^1 \times I_j^2$. 

For $R = I^1 \times I^2$ we define $$
h_R = h_{I^1} \otimes h_{I^2}, \,\, h_R^0 = h_{I^1}^0 \otimes h_{I^2}^0,\,\, h_R^{1,0} = h_{I^1} \otimes h_{I^2}^0\, \textup{ and } \,
h_R^{0,1} = h_{I^1}^0 \otimes h_{I^2}.
$$
Using this notation we write
\begin{equation*}
\begin{split}
\langle &T(E_{R_1} f_1, \ldots, E_{R_{n-1}} f_{n-1}, \Delta_{I_n^1}^1 E_{I_n^2}^2 f_n),  E_{I_{n+1}^1}^1 \Delta_{I_{n+1}^2}^2 f_{n+1} \rangle \\
&= \langle T( h_{R_1}^0, \ldots, h_{R_{n-1}}^0, h_{R_n}^{1,0}), h_{R_{n+1}}^{0,1} \rangle 
A_{R_1, \dots, R_{n+1}}^{n,n+1}(f_1, \dots, f_{n+1}), 
\end{split}
\end{equation*}
where $A_{R_1, \dots, R_{n+1}}^{n,n+1}(f_1, \dots, f_{n+1})=A_{R_1, \dots, R_{n+1}}^{n,n+1}$ is defined in \eqref{eq:DefAR}.

Next, we start adding and subtracting suitable terms as we did in the proof of 
the one-parameter representation theorem \ref{thm:rep1par}, but here we do it separately in each parameter.
First,  we have
\begin{equation}\label{eq:FirstAddSub}
A_{R_1, \dots, R_{n+1}}^{n,n+1}
=A_{R_1, \dots, R_{n+1}}^{n,n+1}
-A_{I_n^1 \times I^2_1, \dots, I^1_n \times I^2_{n+1}}^{n,n+1}+A_{I_n^1 \times I^2_1, \dots, I^1_n \times I^2_{n+1}}^{n,n+1}
\end{equation}
and 
\begin{equation}\label{eq:SecondSubAdd}
A_{I_n^1 \times I^2_1, \dots, I^1_n \times I^2_{n+1}}^{n,n+1}
=(A_{I_n^1 \times I^2_1, \dots, I^1_n \times I^2_{n+1}}^{n,n+1}
-A_{I_n^1 \times I^2_{n+1}, \dots, I^1_n \times I^2_{n+1}}^{n,n+1})
+A_{I_n^1 \times I^2_{n+1}, \dots, I^1_n \times I^2_{n+1}}^{n,n+1}.
\end{equation}
Then, we further have that the difference of the first two terms in the right hand side of \eqref{eq:FirstAddSub} equals
\begin{equation}\label{eq:ThirdSubAdd}
\begin{split}
[A_{R_1, \dots, R_{n+1}}^{n,n+1}
-A_{I_n^1 \times I^2_1, \dots, I^1_n \times I^2_{n+1}}^{n,n+1}
&-A_{I^1_1 \times I^2_{n+1}, \dots, I^1_{n+1} \times I^2_{n+1}}^{n,n+1}
+A_{I^1_n \times I^2_{n+1}, \dots, I^1_n \times I^2_{n+1}}^{n,n+1}] \\
&+\{A_{I^1_1 \times I^2_{n+1}, \dots, I^1_{n+1} \times I^2_{n+1}}^{n,n+1}
-A_{I^1_n \times I^2_{n+1}, \dots, I^1_n \times I^2_{n+1}}^{n,n+1}\}.
\end{split}
\end{equation} 
This gives us the decomposition 
\begin{equation}\label{eq:Decomposition}
A_{R_1, \dots, R_{n+1}}^{n,n+1}
= [\, \cdot \,] + \{\, \cdot \, \} + (\, \cdot \,) + A_{I^1_n \times I^2_{n+1}, \dots, I^1_n \times I^2_{n+1}}^{n,n+1},
\end{equation}
where inside the brackets we have the corresponding term as in \eqref{eq:SecondSubAdd} and  \eqref{eq:ThirdSubAdd}.

The identity \eqref{eq:Decomposition} splits $\Sigma_\sigma$ into four terms 
$\Sigma_\sigma=\Sigma_\sigma^1+\Sigma_\sigma^2+\Sigma_\sigma^3+\Sigma_\sigma^4$.

\subsubsection*{The shift case $\Sigma_\sigma^1$}
We begin by looking at $\Sigma_\sigma^1$, that is, the term coming from $[\, \cdot \,] $ in \eqref{eq:Decomposition}. 
Let us further define the abbreviation
\begin{equation}\label{eq:Phi2Par}
\begin{split}
&\varphi_{R_1, \dots, R_{n+1}} 
:=\langle T( h_{R_1}^0, \ldots, h_{R_{n-1}}^0, h_{R_n}^{1,0}), h_{R_{n+1}}^{0,1} \rangle \\
&\times[A_{R_1, \dots, R_{n+1}}^{n,n+1}
-A_{I_n^1 \times I^2_1, \dots, I^1_n \times I^2_{n+1}}^{n,n+1}
-A_{I^1_1 \times I^2_{n+1}, \dots, I^1_{n+1} \times I^2_{n+1}}^{n,n+1}
+A_{I^1_n \times I^2_{n+1}, \dots, I^1_n \times I^2_{n+1}}^{n,n+1}]
\end{split}
\end{equation}
so that
$$
\Sigma_{\sigma}^1 =  \sum_{ \substack{ R_1, \ldots, R_{n+1} \\  \ell(R_1) = \cdots = \ell(R_{n+1}) }}\varphi_{R_1, \dots, R_{n+1}}.
$$

If $R=I^1\times I^2$ is a rectangle and $m=(m^1,m^2) \in \Z^{d_1} \times \Z^{d_2}$, then we define 
$I^i \dot{+} m^i:=I^i+m^i\ell(I^i)$ and
$R \dot{+} m:= (I^1\dot{+}m^1)\times (I^2\dot{+}m^2)$.
Notice that if $I^1_i=I^1_j$  for all $i,j$ or $I^2_i=I^2_j$ for all $i,j$ then
$\varphi_{R_1, \dots, R_{n+1}}=0$. Thus, there holds that
\begin{equation*}
\begin{split}
\Sigma_\sigma^1
=& \sum_{\substack{m_1, \dots, m_{n+1} \in \Z^{d_1}\times \Z^{d_2} \\ (m_1^1, \dots, m_{n+1}^1) \not=0, \ m^1_n=0 \\ 
(m_1^2, \dots, m_{n+1}^2) \not=0, \ m^2_{n+1}=0}} \sum_{R}
\varphi_{R \dot{+} m_1, \dots,R \dot{+} m_{n+1} } \\
&=\sum_{k_1,k_2=2}^\infty 
\sum_{\substack{m_1, \dots, m_{n+1} \in \Z^{d_1}\times \Z^{d_2} \\ \max |m^1_j| \in (2^{k_1-3}, 2^{k_1-2}], \ m^1_n=0 \\ 
\max |m^2_j| \in (2^{k_2-3}, 2^{k_2-2}], \ m^2_{n+1}=0}}
\sum_{R} \varphi_{R \dot{+} m_1, \dots,R \dot{+} m_{n+1} }.
\end{split}
\end{equation*}

Next, we consider $\E_\sigma \Sigma^1_\sigma$ and add goodness to the rectangles $R$. 
Recall that $\E_\sigma=\E_{\sigma_1}\E_{\sigma_2}$.
We write $\calD_{\sigma, \good}(k_1,k_2):= \calD_{\sigma_1, \good}(k_1) \times \calD_{\sigma_2, \good}(k_2)$ 
and refer to Equation \eqref{eq:DefkGood} for the definition of the collections $\calD_{\sigma_i, \good}(k_i)$ of $k_i$-good cubes.
Similarly as in \eqref{eq:AddGoodness} there holds that
$$
\E_\sigma \sum_{R \in \calD_\sigma} \varphi_{R \dot{+} m_1, \dots,R \dot{+} m_{n+1} }
=2^{d} \E_\sigma \sum_{ R \in \calD_{\sigma,\good}(k_1,k_2)} \varphi_{R \dot{+} m_1, \dots,R \dot{+} m_{n+1} }.
$$
Therefore, we have shown that
\begin{equation}\label{eq:Sigma1}
\E_\sigma \Sigma^1_\sigma
=2^{d}C\sum_{k_1,k_2=2}^\infty\omega_1(2^{-k_1})\omega_2(2^{-k_2})
\langle  Q_{k_1,k_2}(f_1, \dots, f_n), f_{n+1} \rangle,
\end{equation}
where
\begin{equation*}
\begin{split}
\langle & Q_{k_1,k_2}(f_1, \dots, f_n), f_{n+1} \rangle \\
&:= \frac{1}{C\omega_1(2^{-k_1})\omega_2(2^{-k_2})}
\sum_{\substack{m_1, \dots, m_{n+1} \in \Z^{d_1}\times \Z^{d_2} \\ \max |m^1_j| \in (2^{k_1-3}, 2^{k_1-2}], \ m^1_n=0 \\ 
\max |m^2_j| \in (2^{k_2-3}, 2^{k_2-2}], \ m^2_{n+1}=0}}
\sum_{R \in \calD_{\sigma,\good}(k_1,k_2)} \varphi_{R \dot{+} m_1, \dots,R \dot{+} m_{n+1} }
\end{split}
\end{equation*}
and $C$ is a large  enough constant. 

Let $m_1, \dots, m_{n+1}$ and $R=I^1\times I^2$ be as in the definition of $ Q_{k_1, k_2}$. 
By \eqref{eq:GoodnessImplies} the goodness of the rectangle $R$ implies that $(R \dot{+} m_j)^{(k_1, k_2)} = R^{(k_1, k_2)} =: K$ 
for all $j \in \{1, \ldots, n+1\}$.
Recall the definition of $\varphi_{R \dot{+} m_1, \dots,R \dot{+} m_{n+1} }$ from \eqref{eq:Phi2Par}. Therefore, to conclude that
$ Q_{k_1,k_2}$ is a modified bi-parameter $n$-linear shift it remains to prove the normalization
\begin{equation}\label{eq:BiParNorma}
|\langle T( h_{R \dot{+}m_1}^0, \ldots, h_{R\dot{+}m_{n-1}}^0, h_{R\dot{+}m_n}^{1,0}), h_{R\dot{+}m_{n+1}}^{0,1} \rangle|
\lesssim \omega_1(2^{-k_1})\omega_2(2^{-k_2}) \frac{ |R|^{(n+1)/2}}{|K|^n}.
\end{equation}

Let us first assume that $k_1 \sim 1 \sim k_2$. Since $m^1_i \not=0$ and $m^2_j \not=0$ for some $i$ and $j$ we may use the
full kernel representation of $T$ to have that the left hand side of \eqref{eq:BiParNorma} is less than
\begin{equation*}
\begin{split}
\int \displaylimits_{\R^{(n+1)d}} | K(x_{n+1},x_1, \dots, x_n)| 
\prod_{j=1}^{n+1} h_{R \dot{+}m_j}^0(x_j) \ud x.
\end{split}
\end{equation*}
Applying the size of the kernel $K$ this is further dominated by
\begin{equation*}
\begin{split}
& \int \displaylimits_{\R^{(n+1)d_1}}  \frac{1}{\Big(\sum_{j=1}^n |x_{n+1}^1-x_j^1|\Big)^{nd_1}}
\prod_{j=1}^{n+1} h_{I^1 \dot{+} m_j^1}^0(x_j^1)
\ud x^1  \\
&  \times \int \displaylimits_{\R^{(n+1)d_2}} \frac{1}{\Big(\sum_{j=1}^n |x_{n+1}^2-x_j^2|\Big)^{nd_2}}
\prod_{j=1}^{n+1} h_{I^2 \dot{+} m_j^2}^0(x_j^2)
\ud x^2 
\lesssim \frac{1}{|I^1|^{(n-1)/2}|I^2|^{(n-1)/2}}.
\end{split}
\end{equation*}
The estimates for these one-parameter integrals appeared already in \eqref{eq:NearbyEst}.
Notice that this is the right estimate, since $\omega_i(2^{-k_i}) \sim 1$ and $|K|= |R^{(k_1,k_2)}| \sim |R|= |I^1| |I^2|$.

Suppose then that $k_1$ and $k_2$ are large enough so that we can use the continuity assumption of the full kernel $K$.
Using the zero integrals of $h_{I^1}$ and $h_{I^2}$ there holds that the left hand side of \eqref{eq:BiParNorma}
equals
\begin{equation}\label{eq:UseOfFullHolder}
\begin{split}
\Big| & \int \displaylimits_{\R^{(n+1)d}}\Big(
K(x_{n+1},x_1, \dots, x_n)-K(x_{n+1},x_1, \dots, x_{n-1}, (c_{I^1},x^2_n))\\
&-K((x_{n+1}^1,c_{I^2}),x_1, \dots, x_n)+K((x_{n+1}^1,c_{I^2}),x_1, \dots, x_{n-1}, (c_{I^1},x^2_n))\Big) \\
& \times \prod_{j=1}^{n-1} h_{R \dot{+}m_j}^0(x_j)h_{R \dot{+}m_n}^{1,0}(x_n) h_{R\dot{+}m_{n+1}}^{0,1}(x_{n+1})
\ud x \Big|,
\end{split}
\end{equation}
where $c_{I^i}$ denotes the center of the corresponding cube. Here one can use the continuity assumption of $K$
which leads to a product of two one-parameter integrals which can be estimated as in \eqref{eq:FarEst}.

What remains is the case that for example $k_1 \sim 1$ and $k_2$ is large.
This is done similarly as the above two cases using the mixed size and continuity assumption of $K$. This concludes the proof of
\eqref{eq:BiParNorma} and we are done dealing with $\E_\sigma \Sigma_\sigma^1$.

\subsubsection*{The partial paraproduct cases $\Sigma_\sigma^2$ and $\Sigma_\sigma^3$} Next, we look at the symmetric terms $\E_\sigma \Sigma_\sigma^2$ and $\E_\sigma \Sigma_\sigma^3$. We explicitly consider
$\E_\sigma \Sigma_\sigma^2$ here. Recall that $\Sigma_\sigma^2$ equals
$$
\sum_{ \substack{ R_1, \ldots, R_{n+1} \\  \ell(R_1) = \cdots = \ell(R_{n+1}) }}
\langle T( h_{R_1}^0, \ldots, h_{R_{n-1}}^0, h_{R_n}^{1,0}), h_{R_{n+1}}^{0,1} \rangle
\{A_{I^1_1 \times I^2_{n+1}, \dots, I^1_{n+1} \times I^2_{n+1}}^{n,n+1}
-A_{I^1_n \times I^2_{n+1}, \dots, I^1_n \times I^2_{n+1}}^{n,n+1}\}.
$$
Since the difference $A_{I^1_1 \times I^2_{n+1}, \dots, I^1_{n+1} \times I^2_{n+1}}^{n,n+1}
-A_{I^1_n \times I^2_{n+1}, \dots, I^1_n \times I^2_{n+1}}^{n,n+1}$ 
depends only on the cube $I_{n+1}^2$ in the second parameter 
we can further rewrite this as
\begin{equation}\label{eq:LeadsToPartialP}
\begin{split}
\Sigma_\sigma^2=\sum_{ \substack{ I_1^1, \ldots, I_{n+1}^1, I^2 \\  \ell(I_1^1) = \cdots = \ell(I_{n+1}^1) }}
&\langle T( h_{I_1^1}^0\otimes 1, \ldots, h_{I_{n-1}^1}^0\otimes 1, h_{I_n^1}\otimes 1), h_{I^1_{n+1}\times I^2}^{0,1} \rangle \\
&\times\Big\{\prod_{j=1}^{n-1} \Big\langle f_j, h^0_{I_j^1} \otimes \frac{1 _{I^2}}{|I^2|} \Big\rangle 
\cdot  \Big\langle f_n, h_{I_n^1} \otimes \frac{1 _{I^2}}{|I^2|} \Big\rangle 
\langle f_{n+1}, h_{I^1_{n+1}\times I^2}^{0,1}\rangle \\
&-\prod_{j=1}^{n-1} \Big\langle f_j, h^0_{I_n^1} \otimes \frac{1 _{I^2}}{|I^2|} \Big\rangle 
\cdot  \Big\langle f_n, h_{I_n^1} \otimes \frac{1 _{I^2}}{|I^2|} \Big\rangle 
\langle f_{n+1}, h_{I^1_{n}\times I^2}^{0,1}\rangle\Big\}.
\end{split}
\end{equation}

Let us write the summand in \eqref{eq:LeadsToPartialP} as $\varphi_{I_1^1, \dots, I_{n+1}^1,I^2}$. 
By proceeding in the same way as above with
$\E_\sigma \Sigma_\sigma^1$ we have that
\begin{equation}\label{eq:ReWSigma2}
\begin{split}
\E_\sigma \Sigma_\sigma^2
=2^{d_1}C\E_\sigma\sum_{k=2}^\infty \omega_1(2^{-k}) \langle  (Q\pi)_k (f_1, \dots, f_n), f_{n+1} \rangle,
\end{split}
\end{equation}
where
\begin{equation*}
\begin{split}
\langle (Q\pi)_k & (f_1, \dots, f_n), f_{n+1} \rangle \\
&:= \frac{1}{C\omega_1(2^{-k})}\sum_{\substack{m \in \Z^{(n+1)d_1} \\ \max | m_j | \in (2^{k-3}, 2^{k-2} ] \\ m_n=0}}
\sum_{\substack{I^1\in \calD_{\sigma_1,\good}(k) \\ I^2 \in \calD_{\sigma_2}}} \varphi_{I^1\dot{+}m_1, \dots, I^1\dot{+}m_{n+1},I^2}.
\end{split}
\end{equation*}
The $k$-goodness of $I^1$ implies that here $(I^1\dot{+}m_j)^{(k)}=(I^1)^{(k)}=:K^1$ for all $j$. Therefore, to conclude that
$(Q\pi)_k$ is a modified partial paraproduct with the paraproduct component in $\R^{d_2}$ it remains to show that if we fix 
$m_1, \dots, m_{n+1}$ and $I^1$ as in the above sum then
\begin{equation}\label{eq:PPBMO}
\begin{split}
\| (\langle  T( h_{I^1\dot{+}m_1}^0\otimes 1, \ldots, h_{I^1 \dot{+}m_{n-1}}^0\otimes 1,  h_{I^1}\otimes 1), & h_{(I^1\dot{+}m_{n+1})\times I^2}^{0,1} \rangle )_{I_2 \in \calD_{\sigma_2}} \|_{\BMO}  \\
&  \lesssim  \omega_1(2^{-k}) \frac{|I^1|^{(n+1)/2}}{|K^1|^n}. 
\end{split}
\end{equation}

We verify the above $\BMO$ condition by taking a cube $I^2$ and a function $a_{I^2}$ such that $a_{I^2} = a_{I^2}1_{I^2}$, 
$|a_{I^2}| \le 1$ and $\int a_{I^2}=0$, and showing that
\begin{equation}\label{eq:PPBMOTest}
\begin{split}
|\langle  T( h_{I^1\dot{+}m_1}^0\otimes 1, \ldots, h_{I^1 \dot{+}m_{n-1}}^0\otimes 1,& h_{I^1}\otimes 1), 
 h_{(I^1\dot{+}m_{n+1})}^0 \otimes a_{I^2} \rangle | \\
& \lesssim \omega_1(2^{-k}) \frac{|I^1|^{(n+1)/2}}{|K^1|^n} |I^2|.
\end{split}
\end{equation} 
For a suitably large constant $C$ (so that we can use the continuity assumption of the kernel below) 
we split the pairing as
\begin{equation}\label{eq:TestSplit1}
\begin{split}
&\langle  T( h_{I^1\dot{+}m_1}^0\otimes 1_{(CI^2)^c}, h_{I^1\dot{+}m_2}^0\otimes 1, \ldots, h_{I^1 \dot{+}m_{n-1}}^0\otimes 1, h_{I^1}\otimes 1), 
 h_{(I^1\dot{+}m_{n+1})}^0 \otimes a_{I^2} \rangle \\
 &+\langle  T( h_{I^1\dot{+}m_1}^0\otimes 1_{CI^2}, h_{I^1\dot{+}m_2}^0\otimes 1, \ldots, h_{I^1 \dot{+}m_{n-1}}^0\otimes 1, h_{I^1}\otimes 1), 
 h_{(I^1\dot{+}m_{n+1})}^0 \otimes a_{I^2} \rangle.
 \end{split}
 \end{equation}

Let us show that the first term in \eqref{eq:TestSplit1} is dominated by 
$\omega_1(2^{-k}) |I^1|^{(n+1)/2}|I^2|/|K^1|^n$.
We have two cases. The case that $k \sim 1$  is handled with the mixed size and continuity assumption of $K$. The case
that  $k$ is large is handled with the continuity assumption of $K$. We show the details for the case $k \sim 1$.
The other case is done  similarly (see also the paragraph containing \eqref{eq:UseOfFullHolder}).

We assume that $k \sim 1$. Since $a_{I^2}$ has zero integral the pairing that we are estimating equals (by definition)
\begin{equation*}
\begin{split}
\int \displaylimits _ {\R^{(n+1)d}}
&\Big(K(x_{n+1},x_1, \dots, x_n)-K((x^1_{n+1},c_{I^2}),x_1, \dots, x_n)\Big) \\
&\times \prod_{j=1}^{n-1} h_{I^1\dot{+}m_j}^0(x_j^1)h_{I^1}(x_n^1) 
 h_{(I^1\dot{+}m_{n+1})}^0(x_{n+1}^1)1_{(CI^2)^c}(x_1^2) a_{I^2}(x_{n+1}^2) \ud x.
\end{split}
\end{equation*}
The mixed size and continuity property of $K$ implies that the absolute value of the last integral is dominated by
\begin{equation*}
\begin{split}
& \int \displaylimits_{\R^{(n+1)d_1}}  \frac{1}{\Big(\sum_{j=1}^n |x_{n+1}^1-x_j^1|\Big)^{nd_1}}
\prod_{j=1}^{n+1} h_{I^1 \dot{+} m_j^1}^0(x_j^1)
\ud x^1  \\
&  \times \int \displaylimits_{\R^{(n+1)d_2}} \omega_2\Big( 
\frac{|x_{n+1}^2-c_{I^2}|}{\sum_{j=1}^n |c_{I^2}-x_j^2|} \Big) 
\frac{1}{\Big(\sum_{j=1}^n |c_{I^2}-x_j^2|\Big)^{nd_2}}
1_{(CI^2)^c}(x_1^2)1_{I^2}(x_{n+1}^2)
\ud x^2.
\end{split}
\end{equation*}
By \eqref{eq:NearbyEst} we know that the integral related to $\R^{d_1}$ is dominated by $|I^1|^{-(n-1)/2}$.

Consider the integral related to $\R^{d_2}$. By first estimating that
$$
\omega_2\Big( \frac{|x_{n+1}^2-c_{I^2}|}{\sum_{j=1}^n |c_{I^2}-x_j^2|} \Big)
\le \omega_2\Big( \frac{|x_{n+1}^2-c_{I^2}|}{|c_{I^2}-x_1^2|} \Big)
$$
and then repeatedly using estimates of the form \eqref{eq:lesssimInt} one sees that the integral over $\R^{(n+1)d_2}$
is dominated by
\begin{equation*}
\begin{split}
 \int_{I^2} &\int _{(CI^2)^c} \omega_2\Big( \frac{|x_{n+1}^2-c_{I^2}|}{|c_{I^2}-x_1^2|} \Big) 
\frac{1}{|c_{I^2}-x_1^2|^{d_2}}
\ud x_1^2 \ud x_{n+1}^2 \\
&\lesssim |I^2| \int _{(CI^2)^c} \omega_2\Big( \frac{\ell(I^2)}{|c_{I^2}-x_1^2|} \Big) 
\frac{1}{|c_{I^2}-x_1^2|^{d_2}}
\ud x_1^2
\lesssim |I^2| \sum_{k=0}^\infty \omega_2(2^{-k})
\lesssim |I^2|.
\end{split}
\end{equation*}
In conclusion, we showed that the first term in \eqref{eq:TestSplit1} is dominated by
$|I^1|^{-(n-1)/2}|I^2|$, which is the right estimate in the case $k \sim 1$.

We turn to consider the second term in \eqref{eq:TestSplit1}. We again split it into two by writing
$1=1_{(CI^2)^c}+1_{CI^2}$ in the second slot. The part with $1_{(CI^2)^c}$ is estimated in the same way as above and then
one continues with the part related to $1_{CI^2}$. This is repeated until we are only left with the term
\begin{equation}\label{eq:TestSplit2}
\langle  T( h_{I^1\dot{+}m_1}^0\otimes 1_{CI^2}, \ldots, h_{I^1 \dot{+}m_{n-1}}^0\otimes 1_{CI^2}, h_{I^1}\otimes 1_{CI^2}), 
 h_{(I^1\dot{+}m_{n+1})}^0 \otimes a_{I^2} \rangle.
\end{equation}
The estimate for this uses the partial kernel representations of $T$. Again, we have the two cases that either $k \sim 1$ or $k$ is large. These are handled in the same way using either the size or the continuity of the partial kernels. We consider explicitly the case that $k$ is large. Using the zero integral of $h_{I^1}$ we have that the above pairing equals
\begin{equation*}
\begin{split}
 \int_{\R^{(n+1)d_1}} \Big(K_{1_{CI^2}, \dots, 1_{CI^2}, a_{I^2}}(x_{n+1}^1,x_1^1,& \dots, x_n^1)
 -K_{1_{CI^2}, \dots, 1_{CI^2}, a_{I^2}}(x_{n+1}^1,x_1^1, \dots, c_{I^1})\Big) \\
 &\times \prod_{j=1}^{n-1} h^0_{I^1\dot{+}m_j}(x^1_j) h_{I^1}(x_n) h^0_{I^1\dot{+}m_{n+1}}(x^1_{n+1}) \ud x^1.
 \end{split}
\end{equation*}
Taking absolute values and using the continuity of the partial kernel leads to
\begin{equation*}
\begin{split}
C(1_{CI^2},\dots,1_{CI^2},a_{I^2})
&\int_{\R^{(n+1)d_1}}  \omega_1\Big(\frac{|x_n^1-c_{I^1}|}{\sum_{j=1}^n|x_{n+1}^1-x_j^1|}\Big) \\
&\times\frac{1}{\Big(\sum_{j=1}^n|x_{n+1}^1-x_j^1|\Big)^{nd_1}}
\prod_{j=1}^{n+1} h^0_{I^1\dot{+}m_j}(x^1_j) \ud x^1.
\end{split}
\end{equation*}
By assumption there holds that $C(1_{CI^2},\dots,1_{CI^2},a_{I^2}) \lesssim |I^2|$
and by \eqref{eq:FarEst} the integral is dominated by $\omega_1(2^{-k}) |I^1|^{(n+1)/2}{|K^1|^n}$. 
This concludes the proof of \eqref{eq:PPBMOTest} and also finishes our treatment of $\E_\sigma \Sigma_\sigma^2$.

\subsubsection*{The full paraproduct $\Sigma_\sigma^4$}
Recall that
\begin{align*}
\Sigma_\sigma^4 = \sum_{ \substack{ R_1, \ldots, R_{n+1} \\  \ell(R_1) = \cdots = \ell(R_{n+1}) }} \langle T(1_{R_1}, \ldots, 1_{R_{n-1}}, &h_{I_n^1} \otimes 1_{I_n^2}), 1_{I_{n+1}^1} \otimes h_{I_{n+1}^2} \rangle
\prod_{j=1}^{n-1} \langle f_j \rangle_{I_n^1 \times I_{n+1}^2} \\
& \times \Big \langle f_n, h_{I_n^1} \otimes \frac{1_{I_{n+1}^2}}{|I_{n+1}^2|} \Big\rangle
\Big \langle f_{n+1}, \frac{1_{I_{n}^1}}{|I_{n}^1|} \otimes h_{I_{n+1}^2} \Big\rangle,
\end{align*}
which equals
\begin{align*}
&\sum_{R = K^1 \times K^2} \langle T(1, \ldots, 1, h_{K^1} \otimes 1), &1 \otimes h_{K^2} \rangle
\prod_{j=1}^{n-1} \langle f_j \rangle_{R}\Big \langle f_n, h_{K^1} \otimes \frac{1_{K^2}}{|K^2|} \Big\rangle
\Big \langle f_{n+1}, \frac{1_{K^1}}{|K^1|} \otimes h_{K^2} \Big\rangle.
\end{align*}
This is directly a full paraproduct as
$$
 \langle T(1, \ldots, 1, h_{K^1} \otimes 1), 1 \otimes h_{K^2} \rangle =  \langle T^{n*}_1(1, \ldots, 1), h_R \rangle,
$$
and so we are done with this term. Therefore, we are done with the main terms, and no more full paraproducts will appear.

\subsubsection*{The remainder $\operatorname{Rem}_{\sigma}$}
To finish the proof of the bi-parameter representation theorem it remains to discuss the remainder term
$\operatorname{Rem}_{\sigma}$. We recall the collections $\calI_{\sigma_i}$ from the proof of the 
one-parameter representation theorem.
An $(n+1)$-tuple $(I^i_1, \dots, I_{n+1}^i)$ of cubes $I^i_j \in \calD_{\sigma_i}$ belongs to $\calI_{\sigma_i}$ if
the following holds: 
 if $j$ is an index such that $\ell(I^i_j) \le \ell(I^i_k)$ for all $k$, then there exists at least one index $k_0 \not = j$ so that 
$\ell(I^i_j) = \ell(I^i_{k_0})$. The remainder term can be written as
\begin{equation*}
\begin{split}
\operatorname{Rem}_\sigma
&=\sum_{j_1=1}^{n+1} \sum_{\substack{I^1_1, \dots, I^1_{n+1} \\ \ell(I^1_i) >\ell(I^1_{j_1}) \text{ for } i \not= j_1}} 
\sum_{(I^2_1, \dots, I^2_{n+1}) \in \calI_{\sigma_2}} 
\langle T(\Delta_{R_1}f_1, \ldots, \Delta_{R_n}f_n),\Delta_{R_{n+1}}f_{n+1} \rangle \\
&+\sum_{j_2=1}^{n+1}\sum_{\substack{I^2_1, \dots, I^2_{n+1} \\ \ell(I^2_i) >\ell(I^2_{j_2}) \text{ for } i \not= j_2}}
\sum_{(I^1_1, \dots, I^1_{n+1}) \in \calI_{\sigma_1}} 
\langle T(\Delta_{R_1}f_1, \ldots, \Delta_{R_n}f_n),\Delta_{R_{n+1}}f_{n+1} \rangle \\
&+\sum_{\substack{(I^1_1, \dots, I^1_{n+1}) \in \calI_{\sigma_1} \\ (I^2_1, \dots, I^2_{n+1}) \in \calI_{\sigma_2}}}
\langle T(\Delta_{R_1}f_1, \ldots, \Delta_{R_n}f_n),\Delta_{R_{n+1}}f_{n+1} \rangle,
\end{split}
\end{equation*}
where as usual $R_i=I^1_i \times I^2_i$.
Let us write this as 
$$
\operatorname{Rem}_\sigma=\sum_{j_1=1}^{n+1}\operatorname{Rem}_{\sigma,j_1}^1
+\sum_{j_2=1}^{n+1}\operatorname{Rem}_{\sigma,j_2}^2+\operatorname{Rem}_{\sigma}^3.
$$ 

First, we look at the terms $\operatorname{Rem}_{\sigma,j_1}^1$ and $\operatorname{Rem}_{\sigma,j_2}^2$ which are analogous.
Consider for example $\operatorname{Rem}_{\sigma,n+1}^1$. We further divide $\calI_{\sigma_2}$ into subcollections by specifying 
the slots where the smallest cubes are. For example, we consider here the part of the sum with the tuples $(I^2_1, \dots, I^2_{n+1})$ such that
$\ell(I^2_i)>\ell(I^2_n)=\ell(I^2_{n+1})$ for all $i=1, \dots,n-1$. By collapsing the relevant sums of martingale differences
the term we are dealing with can be written as
\begin{equation}\label{eq:ExRem}
\begin{split}
\sum_{\substack{R_1, \dots, R_{n+1} \\ \ell(R_i) =\ell(R_j)}}
\langle T(E_{R_1}f_1, \dots, E_{R_{n-1}}f_{n-1}, E^1_{I^1_n} \Delta^2_{I^2_n}f_n), \Delta_{R_{n+1}}f_{n+1} \rangle.
\end{split}
\end{equation}

In the first parameter there is only one martingale difference and in the second parameter there are two (in the general case at least two). 
Thus, the strategy is that we will write this in terms of model operators that have a  modified shift or 
a paraproduct structure in the first parameter and a standard shift structure in
the second parameter.

First, we add and subtract a suitable term which splits \eqref{eq:ExRem} into the sum of
\begin{equation}\label{eq:Mod/Stand}
\begin{split}
\sum_{\substack{R_1, \dots, R_{n+1} \\ \ell(R_i) =\ell(R_j)}}
\langle T(h^0_{R_1}, \dots, h^0_{R_{n-1}},& h^{0,1}_{R_{n}}), h_{R_{n+1}} \rangle 
\Big[\prod_{j=1}^{n-1} \langle f_j, h^0_{R_j} \rangle \langle f_n, h^{0,1}_{R_n} \rangle
\langle f_{n+1}, h_{R_{n+1}} \rangle \\
&-\prod_{j=1}^{n-1} \langle f_j, h^0_{I^1_{n+1} \times I^2_j} \rangle \langle f_n, h^{0,1}_{I^1_{n+1} \times I^2_n} \rangle
\langle f_{n+1}, h_{R_{n+1}} \rangle\Big]
\end{split}
\end{equation}
and
\begin{equation}\label{eq:P/Stand}
\begin{split}
&\sum_{\substack{R_1, \dots, R_{n+1} \\ \ell(R_i) =\ell(R_j)}}
\langle T(h^0_{R_1}, \dots,  h^0_{R_{n-1}}, h^{0,1}_{R_{n}}), h_{R_{n+1}} \rangle \\
& \quad \quad \quad  \times \prod_{j=1}^{n-1} \langle f_j, h^0_{I^1_{n+1} \times I^2_j} \rangle \langle f_n, h^{0,1}_{I^1_{n+1} \times I^2_n} \rangle
\langle f_{n+1}, h_{R_{n+1}} \rangle \\
&=\sum_{\substack{I^1,I^2_1, \dots, I^2_{n+1} \\
\ell(I^2_i)=\ell(I^2_j)}}
\langle T(1\otimes h^0_{I^2_1}, \dots, 1\otimes h^0_{I_{n-1}^2}, 1\otimes h_{I^2_n}), h_{I^1\times I^2_{n+1}} \rangle \\
&\quad \quad \quad \times \prod_{j=1}^{n-1} \Big\langle f_j,\frac{1_{I^1}}{|I^1|}\otimes h^0_{I^2_j} \Big \rangle 
\Big \langle f_n, \frac{1_{I^1}}{|I^1|}\otimes h_{ I^2_n} \Big \rangle
\langle f_{n+1}, h_{R_{n+1}} \rangle.
\end{split}
\end{equation}
Let us denote the summands in \eqref{eq:Mod/Stand} by $\varphi_{R_1, \dots, R_{n+1}}$. 

In the same way as we did in \eqref{eq:Sigma1} 
the expectation $\E_\sigma$ of \eqref{eq:Mod/Stand} can be written as
\begin{equation*}
\begin{split}
2^{d}\E_\sigma & \sum_{k_1=2}^\infty
\sum_{\substack{m_1, \dots, m_{n} \in \Z^{d_1}\times \Z^{d_2} \\ \max |m^1_j| \in (2^{k_1-3}, 2^{k_1-2}] \\ 
|m^2_j| \le 1}}
\sum_{R \in \calD_{\sigma,\good}(k_1,2)} \varphi_{R \dot{+} m_1, \dots,R \dot{+} m_{n}, R} \\
&+ 2^{d}\E_\sigma  \sum_{k_1=2}^\infty \sum_{k_2=3}^\infty
\sum_{\substack{m_1, \dots, m_{n} \in \Z^{d_1}\times \Z^{d_2} \\ \max |m^i_j| \in (2^{k_i-3}, 2^{k_i-2}] }}
\sum_{R \in \calD_{\sigma,\good}(k_1,k_2)} \varphi_{R \dot{+} m_1, \dots,R \dot{+} m_{n}, R} \\
&=2^dC\E_\sigma \sum_{k_1,k_2=2}^\infty \omega_1(2^{-k_1}) \omega_2(2^{-k_2})
\langle (QS)_{k_1,(k_2, \dots, k_2)} (f_1, \dots, f_n), f_{n+1} \rangle,
\end{split}
\end{equation*}
where $(QS)_{k_1,(k_2, \dots, k_2)}$ is an $n$-linear modified/standard shift. 

We discuss how to prove the right estimate for the pairings
$$
\langle T(h^0_{R\dot{+}m_1}, \dots,  h^0_{R\dot{+}m_{n-1}}, h^{0,1}_{R\dot{+}m_{n}}), h_{R} \rangle.
$$ 
In \eqref{eq:Sigma1} we always had that $\max |m^i_j| \not=0$ for both $i=1,2$.
Here we always have that $\max |m^1_j| \not=0$ but if $k_2=2$ then we may have
$\max |m^2_j|=0$. 
In the case that $\max |m^2_j| \not=0$ the normalization  is proved similarly as  \eqref{eq:BiParNorma}.

Suppose that $\max |m^2_j|=0$ which means that we have the same cube $I^2$ (here $R=I^1\times I^2$) in every slot in the second parameter.
Then we split all the Haar functions in the second parameter similarly as we did in \eqref{eq:DiagSplit1} and \eqref{eq:DiagSplit2},
which leads to two terms. 
The estimate for the term corresponding to \eqref{eq:DiagSplit1} 
is proved similarly as the case $k_2 \sim 1$ of \eqref{eq:BiParNorma}.
With the term corresponding to \eqref{eq:DiagSplit2} one uses the partial kernel representation of $T$ and 
the size estimate
$C(1_{I^2}, \dots, 1_{I^2}) \lesssim |I^2|$ of the partial kernel constant from \eqref{eq:PKWBP}.

Then we look at \eqref{eq:P/Stand}. Similarly as we did with $\E_\sigma \Sigma^2_\sigma$ in 
\eqref{eq:ReWSigma2} we can organize the summation and add goodness to
write the right hand side of \eqref{eq:P/Stand}
in the form
\begin{equation*}
2^{d_2}C \E_\sigma \sum_{k=2}^\infty \omega_2(2^{-k})
\langle  (\pi  S)_{(k,\dots, k)}(f_1, \dots, f_n), f_{n+1} \rangle,
\end{equation*}
where $(\pi  S)_{(k,\dots, k)}$ is a standard partial paraproduct. Notice that when $k=2$ we may have that
$\max |m^2_j|=0$;  in \eqref{eq:ReWSigma2}, where the parameters were in opposite roles,  we always had that $\max |m^1_j|\not=0$.

We discuss the normalization of the coefficients of the partial paraproduct. This means a $\BMO$ estimate of the form
\eqref{eq:PPBMO}, which is equivalent to an estimate of the form \eqref{eq:PPBMOTest}. If there holds
that $\max |m^2_j| \not=0$ then the normalization is proved
in the same way as we proved \eqref{eq:PPBMOTest}. If $\max |m^2_j| =0$ then
we again split the Haar functions in the second parameter as we did in \eqref{eq:DiagSplit1} and \eqref{eq:DiagSplit2}.
The estimate in the case corresponding to \eqref{eq:DiagSplit1} is proved similarly as the case $k \sim 1$ of 
\eqref{eq:PPBMOTest}. In the case corresponding to \eqref{eq:DiagSplit2} one splits into  terms  with separation 
as in \eqref{eq:TestSplit1}
and  a local term as \eqref{eq:TestSplit2}. The terms as in \eqref{eq:TestSplit1} are 
handled using the partial kernel representation of $T$. 
The estimate for the local  term as in \eqref{eq:TestSplit2} follows from the diagonal $\BMO$ assumption
$$
| \langle T(1_R, \dots, 1_R), a_{I^1} \otimes 1_{I^2} \rangle | \lesssim | R |
$$
 from \eqref{eq:DiagBMO}.
This concludes our discussion of $\operatorname{Rem}_{\sigma, n+1}^1$.

Finally, we consider $\operatorname{Rem}_{\sigma}^3$.  This is also divided into several cases by specifying 
the places of the smallest cubes in both parameters. For example,  for notational convenience we take the part where $\ell(I^1_1)=\ell(I^1_{n+1}) < \ell(I^1_i)$ and 
$\ell(I^2_1)=\ell(I^2_{n+1}) < \ell(I^2_i)$ for all $i=2, \dots, n$. Notice that in general the places and the number of the smallest
cubes do not need to be the same in both parameters. After collapsing the relevant sums of martingale differences the term we are looking at
is
\begin{equation}\label{eq:ExRem3}
\begin{split}
\E_\sigma \sum_{\substack{R_1, \dots, R_{n+1} \\ \ell(R_i) =\ell(R_j)}}
\langle T(\Delta_{R_1}f_1, E_{R_2} f_2, \dots, E_{R_{n}}f_{n}), \Delta_{R_{n+1}}f_{n+1} \rangle.
\end{split}
\end{equation}
Here we have two (in the general case at least two) martingale differences in each parameter so this will be written in terms of standard bi-parameter $n$-linear shifts.

In the same way as we did with $\E_\sigma \Sigma^1_\sigma$ in \eqref{eq:Sigma1} 
we can rewrite  \eqref{eq:ExRem3} as
$$
2^dC\E_\sigma\sum_{k_1,k_2=2}^\infty \omega_1(2^{-k_1}) \omega_2(2^{-k_2}) 
\langle  S_{((k_1,k_2), \dots, (k_1,k_2))}(f_1, \dots, f_n),f_{n+1} \rangle,
$$ 
where $S_{((k_1,k_2), \dots, (k_1,k_2))}$ is a standard bi-parameter $n$-linear shift.
Again, if $k_i=2$ then the case $\max |m^i_j|=0$ is included.

We discuss the estimate of
$$
\langle T(h_{R \dot{+}m_1}, h_{R \dot{+}m_2}^0 \dots,h_{R \dot{+}m_n}^0 ), h_R \rangle.
$$ 
If $\max |m^i_j| \ne 0$ for at least one $i \in \{1,2\}$, then we have essentially already seen how to estimate this.
If $\max |m^i_j| =0$ for $i=1,2$ then one splits the Haar functions in both parameters, which leads to four terms.
Three of them are already familiar. The term where we have the case \eqref{eq:DiagSplit2} 
in both parameters, is estimated with the weak boundedness property \eqref{eq:2ParWBP}.
This concludes our consideration of $\operatorname{Rem}_\sigma^3$
and also finishes the proof of Theorem \ref{thm:rep2par}.
\end{proof}

We use the representation theorem, and then the boundedness properties of the model operators from above to obtain the following corollaries. 
The latter corollary also requires the representation of the modified model operators as sums of standard operators (Lemma \ref{lem:QasSBiPar}).
\begin{cor}\label{cor:Dini12nlinbipar}
Let $p_j \in (1, \infty)$, $j=1, \dots,n+1$, be such that $\sum_{j=1}^{n+1} 1/p_j=1$.
Suppose that $T$ is an $n$-linear bi-parameter $(\omega_1, \omega_2)$-CZO, where $\omega_i \in \operatorname{Dini}_{1/2}$.
Then we have the Banach range estimate
\begin{equation}\label{eq:Dini12nlinBRange}
|\langle T(f_1, \ldots, f_n), f_{n+1} \rangle|
 \lesssim \prod_{j=1}^{n+1} \|f_j\|_{L^{p_j}}.
\end{equation}

In the linear case $n=1$ we have the weighted estimate
\begin{equation}\label{eq:Dini12biparw}
\|Tf\|_{L^p(w)} \lesssim \|f\|_{L^p(w)}
\end{equation}
whenever $p \in (1,\infty)$ and $w \in A_p$ is a bi-parameter weight.
\end{cor}

\begin{proof}
Consider first \eqref{eq:Dini12nlinBRange}. 
By Theorem \ref{thm:rep2par} it is enough to have the corresponding estimate 
for the model operators with a bound that depends on the square roots of the complexities.
For modified shifts  this estimate is proved in Proposition \ref{prop:QkBRange} and for
modified partial paraproducts in Proposition \ref{prop:QpiBRange}.
For standard model operators we have the complexity free weighted estimates from propositions
 \ref{prop:StandShiftWeight},  \ref{prop:StandPPWeight} and \ref{prop:FullPWeight}.

Consider then the dualized form $|\langle Tf,g \rangle | \lesssim \| f \|_{L^p(w)} \| g \|_{L^{p'}(w^{1-p'})}$ of \eqref{eq:Dini12biparw}. 
The representation theorem again reduces the estimate to weighted  bounds of model operators with the square root dependence on the complexities.
For modified shifts the bounds are in Proposition \ref{prop:LinQkWeight} and for modified partial paraproducts in 
Proposition \ref{prop:LinQpiWeight}. Recall that in the linear situation this encompasses everything except the full paraproduct as the proofs are ran with the general $H$ formalism.
For the full paraproducts see Proposition \ref{prop:FullPWeight}.
\end{proof}

\begin{cor}\label{cor:Dini1nlinbipar}
Let $p_j \in (1, \infty)$ and $1/r:= \sum_{j=1}^n 1/p_j$. Let $w_j \in A_{p_j}$ be bi-parameter weights and define $w := \prod_{j=1}^n w_j^{r/p_j}$.
Suppose that $T$ is an $n$-linear bi-parameter $(\omega_1, \omega_2)$-CZO, where $\omega_i \in \operatorname{Dini}_{1}$.
Then we have
\begin{equation}\label{eq:Dini1nlinbipar}
\|T(f_1, \ldots, f_n)\|_{L^r(w)} \lesssim \prod_{j=1}^n \|f_j\|_{L^{p_j}(w_j)}.
\end{equation}
\end{cor}

\begin{proof}
The estimate \eqref{eq:Dini1nlinbipar} is equivalent with the estimate
$$
|\langle T(f_1, \ldots, f_n), f_{n+1} \rangle|
 \lesssim \prod_{j=1}^{n} \|f_j\|_{L^{p_j}(w_j)} \| f_{n+1} \|_{L^{r'}(w^{1-r'})}.
$$
Notice that here $w^{1-r'}$ is not necessarily in $A_{r'}$, but that for this trivial usage of duality (to be able to apply the representation theorem) it does not matter.
Because of the representation theorem it suffices to have the estimate 
\begin{equation}\label{eq:ModDini1W}
|\langle V_{k,u,\sigma}(f_1, \ldots, f_n), f_{n+1} \rangle|
 \lesssim (k_1+1)(k_2+1)\prod_{j=1}^{n} \|f_j\|_{L^{p_j}(w_j)} \| f_{n+1} \|_{L^{r'}(w^{1-r'})},
\end{equation}
where $V_{k,u,\sigma}$ is an arbitrary model operator appearing in the representation theorem.

From Lemma \ref{lem:QasSBiPar} we know that each modified model operator is the sum of at most $ck_1k_2$
standard model operators. Furthermore, from propositions \ref{prop:StandShiftWeight},  \ref{prop:StandPPWeight} and \ref{prop:FullPWeight}
we know that standard model operators satisfy complexity free  weighted estimates. Thus, \eqref{eq:ModDini1W} follows and the corollary is proved.
\end{proof}

\section{Commutator estimates}
The basic form of a commutator is $[b,T]\colon f \mapsto bTf - T(bf)$. We are interested in various iterated versions in the multi-parameter setting and with mild kernel regularity.

For a bi-parameter weight $w \in A_2(\R^{d_1} \times \R^{d_2})$ and a locally integrable function $b$ we define
the weighted product $\BMO$ norm
\begin{equation}\label{eq:eq4}
\|b\|_{\BMO_{\textup{prod}}(w)} = \sup_{\calD} \sup_{\Omega}\Bigg(\frac{1}{w(\Omega)}\sum_{\substack{R\in \calD\\ R\subset \Omega}}  \frac {|\langle b, h_R\rangle|^2}{\bla w \bra_R}\Bigg)^{\frac 12},
\end{equation}
where the supremum is over all dyadic grids $\calD^i$ on $\R^{d_i}$ and $\calD = \calD^1 \times \calD^2$, and over all open sets
$\Omega \subset \R^d := \R^{d_1} \times \R^{d_2}$ for which $0 < w(\Omega) < \infty$. The following theorem, which is the two-weight Bloom version of \cite{DaO}, was proved in
\cite{LMV:Bloom2} with $\omega_i(t) = t^{\gamma_i}$.
\begin{thm}\label{thm:comv1}
Suppose that $T_i$ is an $\omega_i$-CZO, where $\omega_i \in \operatorname{Dini}_{3/2}$.
Let $b \colon \R^d \to \C$, $p \in (1, \infty)$, $\mu, \lambda \in A_p(\R^d)$ be bi-parameter weights and $\nu = \mu^{1/p} \lambda^{-1/p} \in A_2(\R^d)$ be the
associated bi-parameter Bloom weight. Then we have
$$
\| [T_1, [T_2, b]] \|_{L^p(\mu) \to L^p(\lambda)} \lesssim \|b\|_{\BMO_{\textup{prod}}(\nu)}.
$$
\end{thm}
\begin{proof}
Let $\|b\|_{\BMO_{\textup{prod}}(\nu)} = 1$.
By Theorem \ref{thm:rep1par} we need to e.g. bound $\| [Q_{k_1}, [Q_{k_2}, b]]f\|_{L^p(\lambda)}$.
It seems non-trivial to fully exploit the operators $Q_{k}$ here and we content on using
Lemma \ref{lem:CompModStand} to reduce to bounding
$$
\sum_{j_1 = 0}^{k_1} \sum_{j_2 = 0}^{k_2} \| [S_{k_1, j_1}, [S_{k_2, j_2}, b]]f\|_{L^p(\lambda)}
$$
and other similar terms, where $S_{k_i, j_i}$ is a linear one-parameter shift on $\R^{d_i}$ of complexity $(k_i, j_i)$.
 Reaching $\operatorname{Dini}_{1}$ would
require replacing this step with a sharper estimate.

On page 11 of \cite{LMV:Bloom2} it is recorded that
$$
\| [S_{u_1, v_1}, [S_{u_2, v_2}, b]]f\|_{L^p(\lambda)} \lesssim (1+\max(u_1, v_1))(1+\max(u_2, v_2)) \|f\|_{L^p(\mu)}.
$$
Interestingly, this part of the argument can be improved: there actually holds that
\begin{equation}\label{eq:eq3}
\|  [S_{u_1, v_1}, [S_{u_2, v_2}, b]]f\|_{L^p(\lambda)} \lesssim (1+\max(u_1, v_1))^{1/2}(1+\max(u_2, v_2))^{1/2} \|f\|_{L^p(\mu)}.
\end{equation}
We will get back to this after completing the proof.
Therefore, we have
$$
\sum_{j_1 = 0}^{k_1} \sum_{j_2 = 0}^{k_2} \| [S_{k_1, j_1}, [S_{k_2, j_2}, b]]f\|_{L^p(\lambda)} \lesssim (1+k_1)^{3/2} (1+k_2)^{3/2} \|f\|_{L^p(\mu)}.
$$
Handling the other terms of the shift expansion of $[Q_{k_1}, [Q_{k_2}, b]]$  similarly, we get
$$
\| [Q_{k_1}, [Q_{k_2}, b]]f\|_{L^p(\lambda)} \lesssim (1+k_1)^{3/2} (1+k_2)^{3/2} \|f\|_{L^p(\mu)}.
$$
Controlling commutators like $[Q_{k_1}, [\pi, b]]$ similarly we get the claim.

We return to \eqref{eq:eq3} now. Decompositions are very involved in the bi-commutator case, and we prefer
to give the idea of the improvement \eqref{eq:eq3} by studying the simpler one-parameter situation $[b, S_{i,j}]$, where
$$
S_{i,j} = \sum_{K} \sum_{I^{(i)} = J^{(j)} = K} a_{IJK} \langle f, h_I \rangle h_J
$$
is a one-parameter shift on $\R^d$ and $b \in \BMO(\nu)$;
\begin{align*}
\|b\|_{\BMO(\nu)} &:= \sup_{I \subset \R^d \textup{ cube}} \frac{1}{\nu(I)} \int_I |b - \langle b \rangle_I| 
\sim \sup_{\calD} \sup_{I_0 \in \calD}\Bigg(\frac{1}{\nu(I_0)}\sum_{\substack{I \in \calD\\ I \subset I_0}}  \frac {|\langle b, h_I\rangle|^2}{\bla \nu \bra_I}\Bigg)^{\frac 12}
 < \infty.
\end{align*}
Here we only have use for the expression on the right-hand side, which is the analogue of the bi-parameter definition \eqref{eq:eq4}.
However, it is customary to define things as on the left-hand side in this one-parameter situation. 
The equivalence follows from the weighted John-Nirenberg \cite{MW}
$$
\sup_{I \subset \R^d \textup{ cube}} \frac{1}{\nu(I)} \int_I |b - \langle b \rangle_I| \sim \sup_{I \subset \R^d \textup{ cube}} \Big( \frac{1}{\nu(I)} \int_I |b - \langle b \rangle_I|^2 \nu^{-1} \Big)^{1/2}, \qquad \nu \in A_2.
$$

Of course, one-parameter commutators
$[b,T]$ can be handled even with $\operatorname{Dini}_{0}$, but e.g. sparse domination proofs \cite{LOR1, LOR2} are restricted to one-parameter, unlike these decompositions.
To get started, we define the one-parameter paraproducts (with some implicit dyadic grid)
$$
A_1(b,f) = \sum_{I} \Delta_{I} b \Delta_{I} f, \, \,
A_2(b,f) = \sum_{I} \Delta_{I} b E_{I} f \,\, \textup{ and } \,\, A_3(b, f) = \sum_{I} E_{I} b \Delta_{I} f.
$$
By writing $b = \sum_{I} \Delta_{I} b$ and $f = \sum_{J} \Delta_{J} f$, and collapsing sums such as
$1_I \sum_{J \colon I \subsetneq J} \Delta_{J} f = E_{I} f$,
we formally have
$$
bf = \sum_{I} \Delta_{I} b \Delta_{I} f 
 + \sum_{I \subsetneq J} \Delta_{I} b \Delta_{J} f + \sum_{J \subsetneq I } \Delta_{I} b \Delta_{J} f
= \sum_{k=1}^3 A_k(b,f).
$$

We now decompose the commutator as follows
\begin{align*}
[b, S_{i,j}]f &= b S_{i,j} f - S_{i,j}(bf) \\
&= \sum_{k=1}^2  A_k(b,S_{i,j} f) - \sum_{k=1}^2 S_{i,j}(A_k(b,f))
+ [A_3(b,S_{i,j} f) - S_{i,j}(A_3(b,f))].
\end{align*}
We have the well-known fact that $\|A_k(b, f)\|_{L^p(\lambda)} \lesssim \|b\|_{\BMO(\nu)} \|f\|_{L^p(\mu)}$ for $k=1,2$ -- this can be seen by using the weighted $H^1$-$\BMO$ duality \cite{Wu}
(with $a_I = \langle b, h_I\rangle$)
\begin{equation}\label{eq:wei1ParH1BMO}
\sum_I |a_I| |b_I| \lesssim \|(a_I)\|_{\BMO(\nu)} \Big\| \Big( \sum_I |b_I|^2 \frac{1_I}{|I|} \Big)^{1/2} \Big \|_{L^1(\nu)},
\end{equation}
where
$$
\|(a_I)\|_{\BMO(\nu)} = \sup_{I_0 \in \calD}\Bigg(\frac{1}{\nu(I_0)}\sum_{\substack{I \in \calD\\ I \subset I_0}}  \frac {|a_I|^2}{\bla \nu \bra_I}\Bigg)^{\frac 12}.
$$
Combining this with the well-known estimate $\|S_{i,j} f\|_{L^p(w)} \lesssim \|f\|_{L^p(w)}$ for all $w \in A_p$ it follows that
$$
\Big\| \sum_{k=1}^2  A_k(b,S_{i,j} f) - \sum_{k=1}^2 S_{i,j}(A_k(b,f)) \Big\|_{L^p(\lambda)} \lesssim \|b\|_{\BMO(\nu)} \|f\|_{L^p(\mu)}.
$$

The complexity dependence is coming from the remaining term
$$
A_3(b,S_{i,j} f) - S_{i,j}(A_3(b,f)) = \sum_{K} \sum_{I^{(i)} = J^{(j)} = K} [\langle b \rangle_J - \langle b \rangle_I] a_{IJK} \langle f, h_I \rangle h_J.
$$
There are many ways to bound this, but the following way based on the $H^1$-$\BMO$ duality -- and executed in the particular way that we do below --
gives the best dependence that we are aware of:
$$
\| A_3(b,S_{i,j} f) - S_{i,j}(A_3(b,f))\|_{L^p(\lambda)} \lesssim (1+\max(i,j))^{1/2} \|b\|_{\BMO(\nu)} \|f\|_{L^p(\mu)}.
$$
We write
$$
\langle b \rangle_J - \langle b \rangle_I = [\langle b \rangle_J - \langle b \rangle_K] - [\langle b \rangle_I - \langle b \rangle_K],
$$
where we further write
$$
\langle b \rangle_J - \langle b \rangle_K = \sum_{J \subsetneq L \subset K} \langle \Delta_L b \rangle_J = 
\sum_{J \subsetneq L \subset K} \langle b, h_L \rangle \langle h_L \rangle_J,
$$
and similarly for $\langle b \rangle_I - \langle b \rangle_K$. We dualize and e.g. look at
\begin{align*}
\sum_{K} & \sum_{I^{(i)} = J^{(j)} = K} \sum_{J \subsetneq L \subset K} |\langle b, h_L \rangle| \langle |h_L| \rangle_J |a_{IJK}| |\langle f, h_I \rangle| | \langle g, h_J\rangle| \\
&= \sum_{K} \sum_{\substack{L \subset K \\ \ell(L) > 2^{-j}\ell(K)}} |\langle b, h_L \rangle| |L|^{-1/2} \sum_{\substack{ I^{(i)} = J^{(j)} = K \\ J \subset L}}  |a_{IJK}| |\langle f, h_I \rangle| | \langle g, h_J\rangle| \\
&\lesssim \|b\|_{\BMO(\nu)} \sum_{K} \int \Big(  \sum_{\substack{L \subset K \\ \ell(L) > 2^{-j}\ell(K)}} \frac{1_L}{|L|^2} \Big[ \sum_{\substack{ I^{(i)} = J^{(j)} = K \\ J \subset L}}  |a_{IJK}| |\langle f, h_I \rangle| | \langle g, h_J\rangle| \Big]^2 \Big)^{1/2} \nu,
\end{align*}
where we used the weighted $H^1$-$\BMO$ duality.
Here
$$
\sum_{\substack{ I^{(i)} = J^{(j)} = K \\ J \subset L}}  |a_{IJK}| |\langle f, h_I \rangle| | \langle g, h_J\rangle| \le \frac{1}{|K|} \int_K |\Delta_{K, i} f| \int_L |\Delta_{K,j}g|,
$$
and we can bound
\begin{align*}
\sum_{K} \int \Big(  \sum_{\substack{L \subset K \\ \ell(L) > 2^{-j}\ell(K)}} &1_L \langle  |\Delta_{K, i} f| \rangle_K^2  \langle  |\Delta_{K, j} g| \rangle_L^2 \Big)^{1/2} \nu \\
&\le j^{1/2} \sum_{K} \int (M \Delta_{K, i} f) (M \Delta_{K, j} g) \nu \\
&\le j^{1/2} \Big\| \Big( \sum_K |M \Delta_{K, i} f|^2 \Big)^{1/2} \Big\|_{L^p(\mu)} \Big\| \Big( \sum_K |M \Delta_{K, j} g|^2 \Big)^{1/2} \Big\|_{L^{p'}(\lambda^{1-p'})} \\
&\lesssim j^{1/2} \|f\|_{L^p(\mu)} \|g\|_{L^{p'}(\lambda^{1-p'})}.
\end{align*}
We are done with the one-parameter case -- the desired bi-parameter case can now be done completely similarly by tweaking the proof in \cite{LMV:Bloom2} using the above idea.
\end{proof}
\begin{rem}
The previous way to use the $H^1$-$\BMO$ duality was to look at
$$
 \sum_{K} \sum_{\substack{L \subset K \\ \ell(L) = 2^{-l}\ell(K)}} |\langle b, h_L \rangle| |L|^{-1/2} \sum_{\substack{ I^{(i)} = J^{(j)} = K \\ J \subset L}}  |a_{IJK}| |\langle f, h_I \rangle| | \langle g, h_J\rangle|,
$$
where $l = 0, \ldots, j-1$ is fixed, and to apply the $H^1$-$\BMO$ duality to the whole $K, L$ summation. With $l$ fixed this yields a uniform estimate, and there
is also a curious 'extra' cancellation present -- we can even bound
$$
\sum_{\substack{ I^{(i)} = J^{(j)} = K \\ J \subset L}}  |a_{IJK}| |\langle f, h_I \rangle| | \langle g, h_J\rangle| \le \frac{1}{|K|} \int_K |\Delta_{K, i} f| \int_L |g|,
$$
that is, forget the $\Delta_{K,j}$ from $g$. Then it remains to sum over $l$ which yields the dependence $j$ instead of $j^{1/2}$. The way in our proof above
is more efficient and we see that we utilize all of the cancellation as well.
\end{rem}
\begin{rem}
An interesting question is can we have $\alpha = 1$ instead of $\alpha = 3/2$
by somehow more carefully exploiting the operators $Q_k$ -- this would appear to be the optimal result theoretically obtainable by the current methods.

We also note that it is certainly possible to handle higher order commutators, such as, $[T_1, [T_2, [b, T_3]]]$.
\end{rem}
We will continue with more multi-parameter commutator estimates -- the difference to the above is that now even the singular integrals
are allowed to be multi-parameter.

For a weight $w$ on $\R^d := \R^{d_1} \times \R^{d_2}$ we say that a locally integrable function $b \colon \R^d \to \C$ belongs to the weighted little BMO space $\bmo_{}(w)$ if
$$
\|b\|_{\bmo(w)} := \sup_{R} \frac{1}{w(R)} \int_R |b - \langle b \rangle_R| < \infty,
$$
where the supremum is over rectangles $R=I^1 \times I^2 \subset \R^d$. If $w=1$ we denote
the unweighted little $\BMO$ space by $\bmo$.
There holds that
\begin{equation}\label{eq:BMOfixedvar}
\| b \|_{\bmo(w)} \sim \max \big( \esssup_{x_1 \in \R^{d_1}} \| b(x_1, \cdot) \|_{\BMO(w(x_1, \cdot))}, 
\esssup_{x_2 \in \R^{d_2}} \| b(\cdot, x_2) \|_{\BMO(w(\cdot, x_2))} \big),
\end{equation}
see \cite{HPW}. Here $\BMO(w(x_1, \cdot))$ and $\BMO(w(\cdot, x_2))$ are the one-parameter weighted $\BMO$ spaces. For example,
$$
\| b(x_1, \cdot) \|_{\BMO(w(x_1, \cdot))} 
:=\sup_{I^2} \frac{1}{ w(x_1, \cdot) (I^2)} \int_{I^2} | b(x_1, y_2)-\ave{b(x_1, \cdot)}_{I^2}| \ud y_2,
$$
where the supremum is over cubes $I^2 \subset \R^{d_2}$.

We first record the following one-weight estimates. If $U$ is an $n$-linear operator, $k \in \{1, \ldots, n\}$ and $b \colon \R^d \to \C$, we formally define the commutator
$$
[b,U]_k(f_1, \ldots, f_n) = bU(f_1, \ldots, f_n) - U(f_1, \ldots, f_{k-1}, bf_k, f_{k+1}, \ldots, f_n).
$$
\begin{thm}\label{thm:CauchyTrick}
Let $n, m \in \{1, 2, \ldots\}$, $k_i \in \{1, \ldots, n\}$, $i = 1, \ldots, m$, and $b_1, \ldots, b_m \in \bmo$.
Let $p_j \in (1, \infty)$ and $1/r:= \sum_{j=1}^n 1/p_j$.
Suppose that $U$ is an $n$-linear operator satisfying for some increasing function $C$ that
$$
\|U(f_1, \ldots, f_n)\|_{L^r(w)} \lesssim \prod_{j=1}^n C([w_j]_{A_{p_j}}) \prod_{j=1}^n \|f_j\|_{L^{p_j}(w_j)}
$$
for all bi-parameter weights $w_j \in A_{p_j}$, $j=1, \dots, n$, and $w :=\prod_{j=1}^n w_j^{r/p_j}$.
Then for all weights like above we have
$$
\| [b_m,\cdots[b_2, [b_1, U]_{k_1}]_{k_2}\cdots]_{k_m}(f_1, \ldots, f_n)\|_{L^r(w)} \lesssim \prod_{i=1}^m\|b_i\|_{\bmo}  \prod_{j=1}^n \|f_j\|_{L^{p_j}(w_j)}.
$$

In particular, suppose that $T$ is an $n$-linear bi-parameter $(\omega_1, \omega_2)$-CZO. Then, the following holds.
\begin{enumerate}
\item If $n=1$ and $\omega_i \in \operatorname{Dini}_{1/2}$, we have
$$
\| [b_m,\cdots[b_2, [b_1, T]]\cdots]f\|_{L^p(w)} \lesssim \prod_{i=1}^m\|b_i\|_{\bmo} \cdot \|f\|_{L^p(w)}
$$
for all $p \in (1,\infty)$ and all bi-parameter $A_p$ weights $w$.
\item If $n \ge 2$ and $\omega_i \in \operatorname{Dini}_{1}$, we have
$$
\| [b_m,\cdots[b_2, [b_1, T]_{k_1}]_{k_2}\cdots]_{k_m}(f_1, \ldots, f_n)\|_{L^r(w)} \lesssim \prod_{i=1}^m\|b_i\|_{\bmo}  \prod_{j=1}^n \|f_j\|_{L^{p_j}(w_j)}
$$
for all $p_j \in (1, \infty)$ and $1/r:= \sum_{j=1}^n 1/p_j$, and for all bi-parameter weights $w_j \in A_{p_j}$, $j=1, \dots, n$, and $w :=\prod_{j=1}^n w_j^{r/p_j}$.
\end{enumerate}
\end{thm}
\begin{proof}
By \eqref{eq:Dini12biparw} and \eqref{eq:Dini1nlinbipar} the second part of the theorem concerning CZOs follows from the abstract claim.
On the other hand, by iteration, it is clearly enough to prove the case $m=1$ of the abstract claim. Then we may, for notational convenience, agree that $k_1 = 1$ and
denote $b_1 = b$. Lastly, by extrapolation we may assume that $r > 1$.

The claim follows from the weighted estimate by the well-known Cauchy trick \cite{CRW} for commutators. Here it is key that this
is a weighted estimate, even if we were just interested in the unweighted estimate for the commutator. Equally important is that the $\BMO$ space $\bmo$
is a simple one here -- it behaves like the one-parameter $\BMO$.

We give the details for the reader's convenience. 
For $z \in \C$ define the operator
$$
F(z)f = \exp(bz) U( \exp(-bz)f_1, f_2, \ldots, f_n).
$$
Next, we write
$$
[b, U]_1 =  F'(0) 
=  C_0 \oint F(z)  \frac{ \ud z}{z^2},
$$
where the integral is over some closed path around the origin. Of course, here we used the
Cauchy integral formula. It follows that, for any $\delta > 0$, we have
\begin{equation*}
\begin{split}
&\| [b, U]_{1}\|_{\prod_{j = 1}^n L^{p_j}(w_j) \to L^r(w)} \\ 
&\lesssim  \oint_{|z| = \delta}  \| F(z) \|_{\prod_{j = 1}^n L^{p_j}(w_j) \to L^r(w)}
\frac{ |\ud z|}{|z|^2} \\
&\le \oint_{|z| = \delta} \|U\|_{L^{p_1}( \exp( p_1\Re(bz))w_1) \times \prod_{j = 2}^n L^{p_j}(w_j) \to L^r( \exp(r \Re(b z))w)}
\frac{ |\ud z|}{\delta^2} \\
&\lesssim  \prod_{j = 2}^n C([w_j]_{A_{p_j}} )  \oint_{|z| = \delta}  C( [ \exp( p_1 \Re( b z))w_1 ]_{A_{p_1}})  
\frac{ |\ud z|}{\delta^2},
\end{split}
\end{equation*}
where we used the weighted bound with the weights $ \exp( p_1 \Re( b z))w_1, w_2, \ldots, w_n$ and
with $ \exp(r \Re(b z))w =  (\exp( p_1 \Re( b z))w_1)^{r/p_1} \prod_{j=2}^n w_j^{r/p_j}$.

Now, it remains to choose the radius $\delta$ intelligently.
This is based on the following standard fact. There is an $\epsilon$ (depending only on $p$ and $d$) so that
\begin{equation}\label{eq:eq1}
[e^{\Re(bz)} \omega]_{A_p} \le C([\omega]_{A_p})
\end{equation}
whenever $z \in \C$ satisfies
$$
|z| \le \frac{\epsilon}{(\omega)_{A_p} \|b\|_{\bmo}}, \qquad (\omega)_{A_p} := \max( [\omega]_{A_p}, [\omega^{1-p'}]_{A_{p'}} ) = [\omega]_{A_p}^{\max(1, p'-1)}.
$$
This follows from Lemma 2.1 of \cite{Hy3}, which is the above statement with the usual $\BMO$ and one-parameter weights $\omega$, and the right-hand side of
\eqref{eq:eq1} can even be replaced with $C_{p, d} [\omega]_{A_p}$. Indeed, simply notice that
\begin{align*}
[e^{\Re(bz)}&\omega]_{A_p} \\ &
\lesssim \max\big( \esssup_{x_1 \in \R^{d_1}} \,[e^{\Re(b(x_1, \cdot) z)}\omega(x_1, \cdot)]_{A_p(\R^{d_2})}, \esssup_{x_2 \in \R^{d_2}}\, [e^{\Re(b(\cdot, x_2)z)}\omega(\cdot, x_2)]_{A_p(\R^{d_1})} \big)^{\gamma},
\end{align*}
where e.g. by Lemma 2.1 of \cite{Hy3} we have
\begin{align*}
[e^{\Re(b(x_1, \cdot) z)}\omega(x_1, \cdot)]_{A_p(\R^{d_2})} \lesssim [\omega(x_1, \cdot)]_{A_p(\R^{d_2})} \le [\omega]_{A_p}
\end{align*}
whenever
$$
|z| \le \frac{\epsilon}{(\omega(x_1, \cdot))_{A_p(\R^{d_2})} \|b(x_1, \cdot)\|_{\BMO(\R^{d_2})}}.
$$
It remains to notice that
$$
(\omega)_{A_p} \|b\|_{\bmo} \gtrsim (\omega(x_1, \cdot))_{A_p(\R^{d_2})} \|b(x_1, \cdot)\|_{\BMO(\R^{d_2})}.
$$

Keeping the above in mind with $\omega = w_1$, for a suitable $\epsilon$ we set
$$
\delta = \frac{\epsilon}{(w_1)_{A_{p_1}} \|b\|_{\bmo}}.
$$
We now get that
\begin{align*}
\| [b, U]_{1}\|_{\prod_{j = 1}^n L^{p_j}(w_j) \to L^r(w)} &\lesssim \prod_{j = 1}^n C([w_j]_{A_{p_j}}) \frac{1}{\delta} 
\lesssim C([w_j]_{A_{p_j}}) \|b\|_{\bmo}.
\end{align*}

\end{proof}
We return to the linear theory. The following theorem was proved in \cite{LMV:Bloom} with $\omega_i(t) = t^{\gamma_i}$.
The first order case $[b,T]$ appeared before in \cite{HPW}. This two-weight Bloom case requires
a proof based on the analysis of the commutators of model operators, and this requires a higher $\alpha$ in the $\operatorname{Dini}_{\alpha}$ than what
is required in Theorem \ref{thm:CauchyTrick}, which is based on the Cauchy trick.
See also \cite{LMV:Bloom2} for the optimality of the space $\bmo(\nu^{1/m})$ in the case $b_1 = \cdots = b_m = b$.

\begin{thm}\label{thm:comv2}
Let $p \in (1,\infty)$, $\mu, \lambda \in A_p$ be bi-parameter weights and $\nu := \mu^{1/p}\lambda^{-1/p}$.
Suppose that $T$ is a bi-parameter $(\omega_1, \omega_2)$-CZO and $m \in \N$.
Then we have
$$
\| [b_m,\cdots[b_2, [b_1, T]]\cdots]\|_{L^p(\mu) \to L^p(\lambda)} \lesssim \prod_{i=1}^m\|b_i\|_{\bmo(\nu^{1/m})}
$$
if one of the following conditions holds:
\begin{enumerate}
\item $T$ is paraproduct free and $\omega_i \in \operatorname{Dini}_{m/2+1}$;
\item $m=1$ and $\omega_i \in \operatorname{Dini}_{3/2}$;
\item $\omega_i \in \operatorname{Dini}_{m+1}$.
\end{enumerate}
\end{thm}
\begin{proof}
The proof is similar in spirit to that of Theorem \ref{thm:comv1}. We use Lemma \ref{lem:QasSBiPar} and estimates for the commutators of the usual bi-parameter
model operators. If we use the bounds from \cite{LMV:Bloom} directly, we e.g. immediately get
\begin{equation}\label{eq:IterBloomQk}
\begin{split}
\| [b_m,\cdots[b_2,& [b_1, Q_{k_1, k_2}]]\cdots]\|_{L^p(\mu) \to L^p(\lambda)} \\ 
& \lesssim (1+k_1)(1+k_2)(1+\max(k_1, k_2))^{m} \prod_{i=1}^m\|b_i\|_{\bmo(\nu^{1/m})}.
\end{split}
\end{equation}
Similarly, we can read an estimate for all the other model operators from \cite{LMV:Bloom}. This gives us the result under the higher regularity assumption (3).
Indeed, when using the estimate \eqref{eq:IterBloomQk} in connection with the representation theorem one ends up with the series
\begin{equation*}
\begin{split}
\sum_{k_1=0}^\infty \sum_{k_2=0}^\infty \omega_1(2^{-k_1})\omega_2(2^{-k_2}) (1+k_1)(1+k_2)(1+\max(k_1, k_2))^{m}.
\end{split}
\end{equation*}
We split this into two according to whether $k_1 \le k_2$ or $k_1>k_2$ and, for example, there holds that
\begin{equation*}
\begin{split}
\sum_{k_1=0}^\infty \omega_1(2^{-k_1})(1+k_1)\sum_{k_2=k_1}^\infty \omega_2(2^{-k_2})(1+k_2)^{m+1}
& \lesssim \sum_{k_1=0}^\infty \omega_1(2^{-k_1})(1+k_1)\| \omega_2 \|_{\operatorname{Dini}_{m+1}} \\
& \lesssim \| \omega_1 \|_{\operatorname{Dini}_{1}}\| \omega_2 \|_{\operatorname{Dini}_{m+1}}.
\end{split}
\end{equation*}

The first order case $m=1$ with the desired regularity (assumption (2)) follows as the papers \cite{EA, ALMV, HPW} dealing with commutators of the form $[T_1, [T_2, \ldots [b, T_k]]]$, where each $T_k$ can be multi-parameter, include the proof of the first order case with the $H^1$-$\BMO$ duality strategy. And this strategy can be improved to give the additional square root save as in Theorem \ref{thm:comv1}.

For $m \ge 2$ the new square root save becomes tricky. 
The paper \cite{LMV:Bloom} is not at all
based on the $H^1$-$\BMO$ duality strategy on which this save is based on (see the proof of Theorem \ref{thm:comv1}).
We can improve the strategy of \cite{LMV:Bloom} for shifts. Thus,
we are able to make the square root save for paraproduct free $T$ (assumption (1)).
By this we mean that (both partial and full) paraproducts in the dyadic representation of $T$ vanish, which could also be stated in terms of (both partial and full) ``$T1=0$'' type conditions.
The reader can think of convolution form SIOs.

We start considering $[b_2, [b_1, S_i]]$, where $i=(i_1,i_2)$, $i_j=(i_j^1,i_j^2)$ and $S_i$ is a standard bi-parameter shift of complexity $i$. 
The reductions in pages 23 and 24 of \cite{LMV:Bloom} (Section 5.1) give that we only need to bound the key term
\begin{align*}
\langle U^{b_1, b_2}f, g \rangle :=
\sum_{K} \sum_{\substack{ R_1,R_2 \\  R_j^{(i_j)}=K}}&
a_{K,R_1,R_2}   [\langle b_1 \rangle_{R_2} - \langle b_1 \rangle_{R_1}]
[\langle b_2 \rangle_{R_2} - \langle b_2 \rangle_{R_1}]
\langle f, h_{R_1} \rangle \langle g, h_{R_2} \rangle,
\end{align*}
where as usual $K=K^1 \times K^2$ and $R_j=I^1_j \times I^2_j$.

We write
\begin{equation*}
\begin{split}
\langle b_i \rangle_{R_2} - \langle b_i \rangle_{R_1} &=
[\langle b_i \rangle_{R_2} - \langle b_i \rangle_{K^1 \times I^2_2}]
+ [\langle b_i \rangle_{K^1 \times I^2_2} - \langle b_i \rangle_{K}] \\
&+ [\langle b_i \rangle_{K} -  \langle b_i \rangle_{K^1 \times I_1^2}] +
[\langle b_i \rangle_{K^1 \times I_1^2} - \langle b_i \rangle_{R_1}].
\end{split}
\end{equation*}
This splits $U^{b_1, b_2}$ into $16$ different terms $U^{b_1, b_2}_{m_1, m_2}$, where
$m_i \in \{1, \ldots, 4\}$ tells which one of the above terms we have for $b_i$.
These can be handled quite similarly, but there are some variations in the arguments.
We will handle two representative ones. 

We begin by looking at the term
\begin{align*}
\langle U^{b_1, b_2}_{3,4} f, g \rangle :=
\sum_{K} \sum_{\substack{ R_1,R_2 \\  R_j^{(i_j)}=K}}&
a_{K,R_1,R_2}   [\langle b_1 \rangle_{K^1 \times I_1^2}-\langle b_1 \rangle_{K}]
[\langle b_2 \rangle_{R_1}-\langle b_2 \rangle_{K^1 \times I_1^2}]
\langle f, h_{R_1} \rangle \langle g, h_{R_2} \rangle.
\end{align*}
Write
\begin{equation}\label{eq:split}
\begin{split}
\langle b_1 \rangle_{K^1 \times I_1^2}-\langle b_1 \rangle_{K} &= \sum_{I_1^2 \subsetneq L^2 \subset K^2} \langle \Delta_{L^2} \langle b_1 \rangle_{K^1, 1} \rangle_{I_1^2} \\
&= \sum_{I_1^2 \subsetneq L^2 \subset K^2} \Big\langle b_1, \frac{1_{K^1}}{|K^1|} \otimes h_{L^2} \Big \rangle \langle h_{L^2} \rangle_{I_1^2}
\end{split}
\end{equation}
and
$$
\langle b_2 \rangle_{R_1}-\langle b_2 \rangle_{K^1 \times I_1^2} = \sum_{I_1^1 \subsetneq L^1 \subset K^1} \langle \Delta_{L^1} \langle b_2 \rangle_{I_1^2, 2} \rangle_{I_1^1} = 
\sum_{I_1^1 \subsetneq L^1 \subset K^1} \Big\langle b_2, h_{L^1} \otimes \frac{1_{I_1^2}}{|I_1^2|} \Big \rangle \langle h_{L^1} \rangle_{I_1^1}.
$$
Writing $\big\langle b_1, \frac{1_{K^1}}{|K^1|} \otimes h_{L^2} \big \rangle = \int_{\R^{d_1}} \langle b_1, h_{L^2} \rangle_2 \frac{1_{K^1}}{|K^1|}$ and similarly for
$\big\langle b_2, h_{L^1} \otimes \frac{1_{I_1^2}}{|I_1^2|} \big \rangle$ we arrive at
\begin{align*}
\int_{\R^d} \sum_{K} & \sum_{\substack{L=L^1 \times L^2 \subset K \\ \ell(L^j) > 2^{-i_1^j}\ell(K^j)}}  
|\langle b_1, h_{L^2} \rangle_2| |L^2|^{-1/2} |\langle b_2, h_{L^1} \rangle_1| |L^1|^{-1/2} \\
&\sum_{\substack{ R_1^{(i_1)}=R_2^{(i_2)}=K \\ R_1 \subset L}} 
| a_{K,R_1,R_2} \langle f, h_{R_1} \rangle \langle g, h_{R_2} \rangle| \frac{1_{K^1}}{|K^1|} \frac{1_{I_1^2}}{|I_1^2|}.
\end{align*}
The last line can be dominated by
$$
 |L^1| \langle M^2 \Delta_{K,i_1} f \rangle_{L^1,1} \langle  |\Delta_{K, i_2}  g|  \rangle_{K} \frac{1_{K^1}}{|K^1|} 1_{L^2}.
$$

We have now reached the term
\begin{align*}
\int_{\R^d} \sum_{K} \langle  |\Delta_{K, i_2}  g| & \rangle_{K} \frac{1_{K^1}}{|K^1|} 
\sum_{\substack{L^2 \subset K^2 \\ \ell(L^2) > 2^{-i_1^2}\ell(K^2)}} |\langle b_1, h_{L^2} \rangle_2| |L^2|^{-1/2} 1_{L^2} \\
&\sum_{\substack{L^1 \subset K^1 \\ \ell(L^1) > 2^{-i_1^1}\ell(K^1)}} |\langle b_2, h_{L^1} \rangle_1| |L^1|^{1/2}  \langle M^2 \Delta_{K, i_1} f \rangle_{L^1,1}. 
\end{align*}
Recall that with fixed $x_2$ we have $b(\cdot, x_2) \in \BMO(\nu^{1/2}(\cdot,x_2))$, see \eqref{eq:BMOfixedvar}. 
By weighted $H^1$-$\BMO$ duality we now have that
\begin{align*}
&\sum_{\substack{L^1 \subset K^1 \\ \ell(L^1) > 2^{-i_1^1}\ell(K^1)}} |\langle b_2, h_{L^1} \rangle_1(x_2) | |L^1|^{1/2}  
\langle M^2 \Delta_{K, i_1} f \rangle_{L^1,1}(x_2) \\
&\lesssim \|b_2\|_{\bmo(\nu^{1/2})} \int_{\R^{d_1}} \Big( \sum_{\substack{L^1 \subset K^1 \\ \ell(L^1) > 2^{-i_1^1}\ell(K^1)}} 1_{L^1}(y_1) (\langle M^2 \Delta_{K, i_1}  f \rangle_{L^1,1}(x_2))^2 \Big)^{1/2} \nu^{1/2}(y_1, x_2) \ud y_1 \\
&\le (i_1^1)^{1/2}  \|b_2\|_{\bmo(\nu^{1/2})} |K^1| \langle M^1 M^2 \Delta_{K, i_1}  f \cdot \nu^{1/2} \rangle_{K^1,1}(x_2).
\end{align*}
The term $(i_1^1)^{1/2}  \|b_2\|_{\bmo(\nu^{1/2})}$ is fine and we do not drag it along in the following estimates. We are left with the task of bounding
\begin{align*}
\int_{\R^d} \sum_{K} \langle  |\Delta_{K, i_2}  g|  \rangle_{K}1_{K^1}
\sum_{\substack{L^2 \subset K^2 \\ \ell(L^2) > 2^{-i_1^2}\ell(K^2)}}& |\langle b_1, h_{L^2} \rangle_2| |L^2|^{-1/2} 1_{L^2} \\
& M^1( M^1 M^2 \Delta_{K, i_1} f \cdot \nu^{1/2}).
\end{align*}
We now put the $\int_{\R^{d_2}}$ inside and get the term
$$
\int_{\R^{d_2}} 1_{L^2} M^1( M^1 M^2 \Delta_{K, i_1}  f \cdot \nu^{1/2}) 
= |L^2| \langle M^1( M^1 M^2 \Delta_{K, i_1}  f \cdot \nu^{1/2}) \rangle_{L^2, 2}.
$$
Then, we are left with
\begin{align*}
\int_{\R^{d_1}} \sum_{K} \langle  |\Delta_{K, i_2}  g|  \rangle_{K}1_{K^1}
\sum_{\substack{L^2 \subset K^2 \\ \ell(L^2) > 2^{-i_1^2}\ell(K^2)}}& |\langle b_1, h_{L^2} \rangle_2| |L^2|^{1/2}  \\
&\langle M^1( M^1 M^2 \Delta_{K, i_1}  f \cdot \nu^{1/2}) \rangle_{L^2, 2}.
\end{align*}

By weighted $H^1$-$\BMO$ duality we have analogously as above that
\begin{align*}
\sum_{\substack{L^2 \subset K^2 \\ \ell(L^2) > 2^{-i_1^2}\ell(K^2)}}& |\langle b_1, h_{L^2} \rangle_2| |L^2|^{1/2}  
\langle M^1( M^1 M^2 \Delta_{K, i_1}  f \cdot \nu^{1/2}) \rangle_{L^2, 2}  \\
&\lesssim  (i_1^2)^{1/2}  \|b_1\|_{\bmo(\nu^{1/2})} \int_{\R^{d_2}} 1_{K^2}  M^2 M^1( M^1 M^2 \Delta_{K, i_1} f \cdot \nu^{1/2}) \nu^{1/2}.
\end{align*}
Forgetting the factor $(i_1^2)^{1/2}  \|b_1\|_{\bmo(\nu^{1/2})}$, which is as desired, we are then left with
\begin{align*}
\int_{\R^{d}} \sum_{K}& \langle  |\Delta_{K, i_2}  g|  \rangle_{K}1_{K}  
M^2 M^1( M^1 M^2 \Delta_{K, i_1}  f \cdot \nu^{1/2}) \nu^{1/2} \\ 
&\le \int_{\R^{d}} \sum_{K} M^2 M^1( M^1 M^2 \Delta_{K, i_1}  f \cdot \nu^{1/2}) \cdot M^1 M^2 \Delta_{K, i_2} g \cdot \nu^{1/2}.
\end{align*}
Writing $\nu^{\frac{1}{2}} = \mu^{\frac{1}{2p}} \lambda^{\frac{1}{2p}} \cdot \lambda^{-\frac{1}{p}}$ we bound this with
$$
\Big\| \Big( \sum_{K} [M^2 M^1( M^1 M^2 \Delta_{K, i_1} f \cdot \nu^{1/2})]^2 \Big)^{1/2} \Big\|_{L^p(\mu^{1/2}\lambda^{1/2})} 
$$
multiplied by
$$
\Big\| \Big( \sum_{K} [M^1 M^2 \Delta_{K, i_2} g]^2 \Big)^{1/2} \Big\|_{L^{p'}(\lambda^{1-p'})}.
$$
It remains to use square function bounds together with the Fefferman--Stein inequality. For the more complicated term with the function $f$
the key thing to notice is that first $\mu^{1/2}\lambda^{1/2} \in A_p$ and then that
$\nu^{p/2} \mu^{1/2}\lambda^{1/2} = \mu$. We have controlled $\langle U^{b_1, b_2}_{3,4} f, g \rangle$. 

The bound for $\langle U^{b_1, b_2}f, g \rangle$ follows by handling the other similar terms $U^{b_1, b_2}_{m_1, m_2}$. 
There is a slight variation in the argument needed, for example, in the following term
\begin{align*}
\langle U^{b_1, b_2}_{1,1} f, g \rangle :=
\sum_{K} \sum_{\substack{ R_1,R_2 \\  R_j^{(i_j)}=K}}&
a_{K,R_1,R_2}   [\langle b_1 \rangle_{R_2}-\langle b_1 \rangle_{K^1 \times I^2_2}]
[\langle b_2 \rangle_{R_2}-\langle b_2 \rangle_{K^1 \times I^2_2}]
\langle f, h_{R_1} \rangle \langle g, h_{R_2} \rangle.
\end{align*}
We expand the differences of averages as 
\begin{equation*}
\begin{split}
 [\langle b_1 \rangle_{R_2}-&\langle b_1 \rangle_{K^1 \times I^2_2}]
 [\langle b_2 \rangle_{R_2}-\langle b_2 \rangle_{K^1 \times I^2_2}] \\
& = \sum_{I^1_2 \subsetneq U^1 \subset K^1}\sum_{I^1_2 \subsetneq V^1 \subset K^1}
\Big\langle b_1, h_{U^1} \otimes \frac{1_{I^2_2}}{|I^2_2|} \Big\rangle \langle h_{U^1} \rangle_{I^1_2}
\Big\langle b_2, h_{V^1} \otimes \frac{1_{I^2_2}}{|I^2_2|} \Big\rangle \langle h_{V^1} \rangle_{I^1_2}.
\end{split}
\end{equation*}
The key difference to the above term $U^{b_1, b_2}_{3,4}$ is that 
we need to further split this into two by comparing whether we have $V^1 \subset U^1$ or $U^1 \subsetneq V^1$.
The related two terms are handled symmetrically. The absolute value of the one coming from ``$V^1 \subset U^1$''
can be written as
\begin{equation*}
\begin{split}
\int_{\R^{d_2}} &\int_{\R^{d_2}}
\sum_{K}  \sum_{\substack{U^1 \subset K^1 \\ \ell(U^1) > 2^{-i_2^1}\ell(K^1)}} \sum_{\substack{V^1 \subset U^1 \\ \ell(V^1) > 2^{-i_2^1}\ell(K^1)}}
  |\langle b_1, h_{U^1} \rangle_1(x_2)| |{U^1}|^{-1/2} |\langle b_2, h_{V^1} \rangle_1(y_2)| |{V^1}|^{-1/2}  \\
& \sum_{\substack{ (I_1^1)^{(i_1^1)} = (I^1_2)^{(i_2^1)} = K^1 \\  I^1_2 \subset V^1}} \sum_{(I_1^2)^{(i_1^2)} = (I^2_2)^{(i_2^2)} = K^2}
| a_{K,R_1,R_2} \langle f, h_{R_1} \rangle \langle g, h_{R_2} \rangle| \frac{1_{I^2_2}(x_2)}{|I^2_2|} \frac{1_{I^2_2}(y_2)}{|I^2_2|}.
\end{split}
\end{equation*}
The last line can be dominated by
\begin{equation*}
\begin{split}
\langle | \Delta_{K,i_1}f | \rangle_{K}  
|V^1| \sum_{(I^2_2)^{(i_2^2)} = K^2} \langle  |\Delta_{K,i_2} g|  \rangle_{V^1 \times I^2_2} \frac{1_{I^2_2}(x_2)}{|I^2_2|} 1_{I^2_2}(y_2).
\end{split}
\end{equation*}
Using the weighted $H^1$-$\BMO$ duality as above we have
\begin{align*}
\int_{\R^{d_2}}  &\sum_{\substack{V^1 \subset U^1 \\ \ell(V^1) > 2^{-i_2^1}\ell(K^1)}} |\langle b_2, h_{V^1} \rangle_1(y_2)| |{V^1}|^{1/2} 
\langle  |\Delta_{K,i_2} g|  \rangle_{V^1 \times I^2_2} 1_{I^2_2}(y_2) \ud y_2 \\
&\le (i_2^1)^{1/2} \| b_2 \|_{\bmo(\nu^{1/2})} |U^1| |I^2_2| \langle M^1 M^2 \Delta_{K,i_2} g \cdot \nu^{1/2} \rangle_{U^1 \times I^2_2}.
\end{align*}
Forgetting the factor $ (i_2^1)^{1/2} \| b_2 \|_{\bmo(\nu^{1/2})}$ we have reached the term
\begin{align*}
\int_{\R^{d_2}} 
\sum_{K} \langle | \Delta_{K,i_1}f | \rangle_{K}  
 \sum_{(I^2_2)^{(i_2^2)} = K^2} 1_{I^2_2} &\sum_{\substack{U^1 \subset K^1 \\ \ell(U^1) > 2^{-i_2^1}\ell(K^1)}} 
|\langle b_1, h_{U^1} \rangle_1| |{U^1}|^{1/2} \\
&\langle M^1 M^2 \Delta_{K,i_2} g \cdot \nu^{1/2} \rangle_{U^1 \times I^2_2},
\end{align*}
which -- after using the $H^1$-$\BMO$ duality -- produces $(i_2^1)^{1/2} \| b_1 \|_{\bmo(\nu^{1/2})}$ multiplied by
\begin{align*}
\int_{\R^d} \sum_{K} \langle | \Delta_{K,i_1} f | \rangle_{K}  
M^1M^2 (M^1 M^2 \Delta_{K,i_2} g \cdot \nu^{1/2}) \nu^{1/2} 1_{K}.
\end{align*}
Similarly as with $U^{b_1,b_2}_{3,4}$, this term is under control. The term with $U^1 \subsetneq V^1$ is symmetric, and so we are
also done with $U^{b_1,b_2}_{1,1}$. 

This ends our treatment of $U^{b_1, b_2}$, since the above arguments
showcased the only major difference between the various terms  $U^{b_1, b_2}_{m_1, m_2}$.
Thus, we are done with $[b_2, [b_1, S_{i}]]$.
By Lemma \ref{lem:QasSBiPar} we conclude that
$$
\| [b_2, [b_1, Q_{k_1, k_2}]]\|_{L^p(\mu) \to L^p(\lambda)} 
\lesssim (1+k_1)(1+k_2) (1+\max(k_1, k_2)) \prod_{i=1}^2 \|b_i\|_{\bmo(\nu^{1/2})}.
$$
By handling the higher order commutators similarly, we get the claim related to assumption (1). We omit these details.

\end{proof}
\begin{rem}
The new square root save from the $H^1$-$\BMO$ arguments reduces the required regularity from $m+1$ to $m/2+1$. In these higher order commutators
this is more significant than the save that could theoretically be obtained by not using Lemma \ref{lem:QasSBiPar}.
This could change the $+1$ to $+1/2$. 
\end{rem}

Theorem \ref{thm:comv1} involves only one-parameter CZOs in its estimate
$$
\| [T_1, [T_2, b]] \|_{L^p(\mu) \to L^p(\lambda)} \lesssim \|b\|_{\BMO_{\textup{prod}}(\nu)},
$$
while the basic estimate
$$
\| [b, T]\|_{L^p(\mu) \to L^p(\lambda)} \lesssim \|b\|_{\bmo(\nu)}
$$
of Theorem \ref{thm:comv2} involves a bi-parameter CZO $T$. A joint generalization -- considered in the unweighted case in \cite{OPS} -- is an estimate for
$$
\| [T_1, [T_2, \ldots [b, T_k]]] \|_{L^p(\mu) \to L^p(\lambda)},
$$
where each $T_i$ can be a completely general $m$-parameter CZO. Then the appearing $\BMO$ norm is some suitable combination of little $\BMO$ and product $\BMO$.
See \cite{EA, ALMV} for a fully satisfactory Bloom type upper estimate in this generality -- however, only for CZOs with the standard kernel regularity.
The general case of \cite{EA, ALMV} is hard to digest, but let us formulate a model theorem of this type with mild kernel regularity.
\begin{thm}\label{thm:comv3}
Let $\R^d = \prod_{i=1}^4 \R^{d_i}$ be a product space of four parameters and let $\calI = \{\calI_1, \calI_2\}$, where $\calI_1 = \{1,2\}$ and
$\calI_2 = \{3,4\}$, be a partition of the parameter space $\{1, 2, 3, 4\}$.
Suppose that $T_i$ is a bi-parameter $(\omega_{1,i}, \omega_{2,i})$-CZO on $\prod_{j \in \calI_i} \R^{d_j}$, where $\omega_{j, i} \in \operatorname{Dini}_{3/2}$.
Let $b \colon \R^d \to \C$, $p \in (1, \infty)$, $\mu, \lambda \in A_p(\R^d)$ be $4$-parameter weights and $\nu = \mu^{1/p} \lambda^{-1/p}$ be the
associated Bloom weight. Then we have
$$
\| [T_1, [T_2, b]] \|_{L^p(\mu) \to L^p(\lambda)} \lesssim \|b\|_{\bmo^{\calI}(\nu)}.
$$
Here $\bmo^{\calI}(\nu)$ is the following weighted little product $\BMO$ space:
$$
\|b\|_{\bmo^{\calI}(\nu)} = \sup_{\bar u} \|b\|_{\BMO_{\operatorname{prod}}^{\bar u}(\nu)},
$$
where $\bar u = (u_i)_{i=1}^2$ is such that $u_i \in \calI_i$ and $\BMO_{\operatorname{prod}}^{\bar u}(\nu)$ is
the natural weighted bi-parameter product $\BMO$ space on the parameters $\bar u$. For example,
$$
\|b\|_{\BMO_{\operatorname{prod}}^{(1,3)}(\nu)}
:= \sup_{x_2 \in \R^{d_2}, x_4 \in \R^{d_4}} \| b(\cdot, x_2, \cdot, x_4) \|_{\BMO_{\operatorname{prod}}(\nu(\cdot, x_2, \cdot, x_4))},
$$
where the last weighted product $\BMO$ norm is defined in \eqref{eq:eq4}.
\end{thm}
The proof is again a combination of Lemma \ref{lem:QasSBiPar} with the known estimates for the commutators of standard model operators \cite{EA, ALMV}.
However, there is again the additional square root save. There are no new significant challenges with this, which was not the case with Theorem \ref{thm:comv2} above, since
these references are completely based on the $H^1$-$\BMO$ strategy. In this regard the situation is closer to that of Theorem \ref{thm:comv1}.

 \end{document}